\documentclass[11pt]{amsart}
\usepackage{graphicx}
\usepackage[mathcal]{euscript}
\usepackage{float}
\usepackage{cite}
\usepackage{verbatim}
\usepackage{url}
\usepackage{tikz}
\usepackage{adjustbox}
\usepackage{amsmath}
\usepackage{multicol}
\usepackage{rotating}

\usetikzlibrary{decorations.pathreplacing}
\usetikzlibrary{arrows.meta}

\usetikzlibrary{arrows, shapes, decorations, decorations.markings, backgrounds, patterns, hobby, knots, calc, positioning}

\pgfdeclarelayer{background}
\pgfdeclarelayer{background2}
\pgfdeclarelayer{background2a}
\pgfdeclarelayer{background2b}
\pgfdeclarelayer{background3}
\pgfdeclarelayer{background4}
\pgfdeclarelayer{background5}
\pgfdeclarelayer{background6}
\pgfdeclarelayer{background7}

\pgfsetlayers{background7,background6,background5,background4,background3,background2b,background2a,background2,background,main}

\definecolor{lsupurple}{RGB}{70,29,124}
\definecolor{lsugold}{RGB}{253,208, 35}

\restylefloat{figure}

\newtheorem{theorem}{Theorem}[section]
\newtheorem{lemma}[theorem]{Lemma}
\newtheorem{corollary}[theorem]{Corollary}
\newtheorem*{corollaryconj*}{Corollary of Conjecture \ref{conjecture:Signature}}
\newtheorem{proposition}[theorem]{Proposition}

\newtheorem{conjecture}[theorem]{Conjecture}
\newtheorem{question}[theorem]{Question}

\theoremstyle{definition}
\newtheorem{definition}[theorem]{Definition}
\newtheorem{example}[theorem]{Example}

\theoremstyle{remark}

\numberwithin{equation}{section}

\def\alt{\operatorname{alt}}
\def\dalt{\operatorname{dalt}}

\def\tb{\operatorname{tb}}
\def\mtb{\overline{\operatorname{tb}}}
\def\minus{\scalebox{0.75}[1.0]{$-$}}

\usepackage{tikz}

\usetikzlibrary{arrows,shapes,decorations,backgrounds,patterns}

\pgfdeclarelayer{background}
\pgfdeclarelayer{background2}
\pgfdeclarelayer{background2a}
\pgfdeclarelayer{background2b}
\pgfdeclarelayer{background3}
\pgfdeclarelayer{background4}
\pgfdeclarelayer{background5}
\pgfdeclarelayer{background6}
\pgfdeclarelayer{background7}

\pgfsetlayers{background7,background6,background5,background4,background3,background2b,background2a,background2,background,main}

\definecolor{lsupurple}{RGB}{70,29,124}
\definecolor{lsugold}{RGB}{253,208, 35}

\begin{document}

\title{Extremal Khovanov homology of Turaev genus one links}

\author{Oliver T. Dasbach}
\address{Department of Mathematics\\
Louisiana State University\\
Baton Rouge, LA}
\email{kasten@math.lsu.edu}

\author{Adam M. Lowrance}
\address{Department of Mathematics\\
Vassar College\\
Poughkeepsie, NY} 
\email{adlowrance@vassar.edu}
\thanks{The first author was supported in part by NSF grant DMS-1317942. The second author was supported in part by Simons Collaboration Grant for Mathematicians no. 355087 and NSF grant DMS-1811344.}

\date{}

\begin{abstract}
The Turaev genus of a link can be thought of as a way of measuring how non-alternating a link is. A link is Turaev genus zero if and only if it is alternating, and in this viewpoint, links with large Turaev genus are very non-alternating.
In this paper, we study Turaev genus one links, a class of links which includes almost alternating links. We prove that
the Khovanov homology of a Turaev genus one link is isomorphic to $\mathbb{Z}$ in at least one of its extremal quantum gradings. As an application, we compute or nearly compute the maximal Thurston Bennequin number of a Turaev genus one link. 
\end{abstract}

\maketitle

\section{Introduction}

\subsection{Turaev genus one and almost alternating links}
The Turaev surface associated to a link diagram is a closed, oriented surface that has close ties to the Jones polynomial of the link. The Turaev genus $g_T(L)$ of a link $L$ is the minimum genus of any Turaev surface of a diagram of $L$. Turaev \cite{Turaev:Jones} first constructed this surface to give a topological simplification of the proof that the span of the Jones polynomial is a lower bound on the crossing number of the link, which implies Tait's conjecture that reduced alternating diagrams have minimum crossing number. 

The Turaev genus of a link can be thought of as giving a filtration on all links where a link with large Turaev genus is qualitatively far away from being alternating. Alternating links are precisely those links with Turaev genus zero, and in this viewpoint, links of Turaev genus one are close to being alternating. Armond, Lowrance \cite{ArmLow:Turaev} and independently Kim \cite{Kim:TuraevClassification} classified Turaev genus one links by proving that each such link has a diagram as in Figure \ref{figure:tg1}. Using that classification, Dasbach and Lowrance \cite{DasLow:TuraevJones} proved that either the first or last coefficient of the Jones polynomial of a Turaev genus one link has absolute value one. It is this result that we generalize to Khovanov homology in the present paper. In order to do so, we will need to introduce almost alternating and $A$- and $B$-adequate links.

Adams et al. \cite{Adams:Almost} defined an \textit{almost alternating link} to be a non-alternating link with a diagram that can be transformed into an alternating diagram via a single crossing change. Such a diagram is called an \textit{almost alternating diagram}. A generic almost alternating diagram can be decomposed into the crossing that is changed to obtain an alternating diagram, called the \textit{dealternator}, and an alternating $2$-tangle $R$ as in Figure \ref{figure:aadiagram}. Let $u_1$ and $u_2$ be the two regions incident to the dealternator $\tikz[baseline=.6ex, scale = .4]{
\draw (0,0) -- (.3,.3);
\draw (.7,.7) -- (1,1);
\draw (0,1) -- (1,0);
}
~$ that are joined by an $A$-resolution$~ \tikz[baseline=.6ex, scale = .4]{
\draw[rounded corners = 1mm] (0,0) -- (.5,.45) -- (1,0);
\draw[rounded corners = 1mm] (0,1) -- (.5,.55) -- (1,1);
}~$
and let $v_1$ and $v_2$ be the two regions incident to the dealternator that are joined by a $B$-resolution $~ \tikz[baseline=.6ex, scale = .4]{
\draw[rounded corners = 1mm] (0,0) -- (.45,.5) -- (0,1);
\draw[rounded corners = 1mm] (1,0) -- (.55,.5) -- (1,1);
}~$. Suppose that the regions of the link diagram are checkerboard colored with $v_1$ and $v_2$ being colored black, while $u_1$ and $u_2$ are colored white.

\begin{definition}
\label{definition:ABalmostalternating}
An almost alternating diagram is \textit{$A$-almost alternating} if it satisfies the following conditions.
\begin{itemize}
\item [(1)] The regions $u_1$ and $u_2$ are distinct, and the regions $v_1$ and $v_2$ are distinct.
\item [(2)] There is no crossing in $R$ that is in the boundary of $u_1$ and $u_2$, and there is no crossing in $R$ that is in the boundary of $v_1$ and $v_2$.
\item [(3A)] There is no white region in $R$ that shares a crossing with each of $u_1$ and $u_2$.
\end{itemize}
A link with an $A$-almost alternating diagram is an \textit{$A$-almost alternating link}. An almost alternating diagram is \textit{$B$-almost alternating} if it satisfies conditions (1) and (2) above as well as condition (3B).
\begin{itemize}
\item [(3B)] There is no black region in $R$ that shares a crossing with each of $v_1$ and $v_2$.
\end{itemize}
A link with a $B$-almost alternating diagram is a \textit{$B$-almost alternating link}. 
\end{definition}
If an almost alternating diagram fails to satisfy condition (1) or (2), then the link is alternating (see Figure \ref{figure:Reducible}). In Theorem \ref{theorem:ABaa} we show that every almost alternating link is either $A$-almost alternating or $B$-almost alternating. See Figure \ref{figure:aadiagram} for an example of a generic almost alternating diagram and an $A$-almost alternating knot.
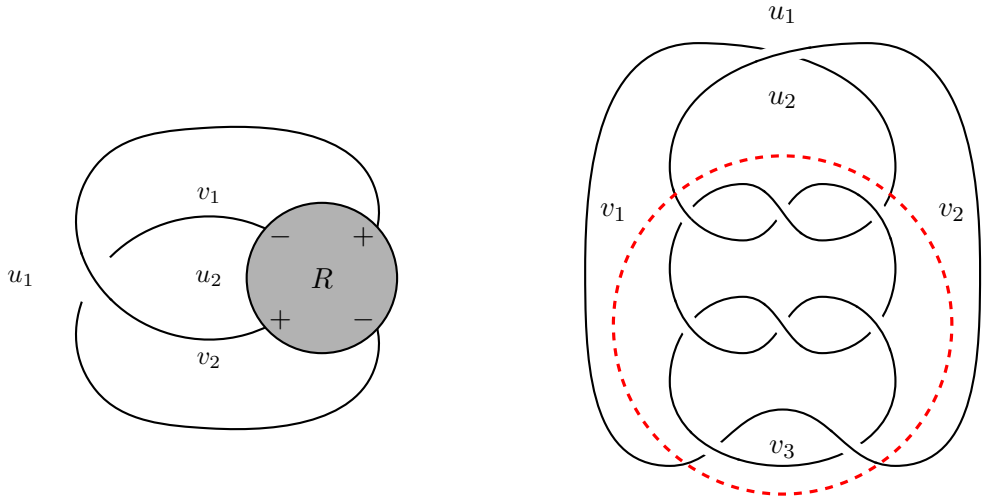
\begin{figure}[h]
$$\begin{tikzpicture}[thick]

\begin{knot}[
	consider self intersections,
 	clip width = 5,
 	ignore endpoint intersections = false,
 ]
 \strand
 (4,0) to [curve through = {(4.75,1) .. (2.5,2) .. (1,1.5)  .. (2,-.75) .. (3,-.75)}] (4,0);
 \strand 
 (4,0) to [curve through = {(4.75,-1) .. (2.5, -2) .. (1,-1.5) .. (2,.75) .. (3,.75)}] (4,0);

\end{knot}

\fill[white!70!black] (4,0) circle (1cm);
\draw[thick] (4,0) circle (1cm);
\draw (4,0) node{$R$};
\draw (3.45,.55) node{$-$};
\draw (4.55,.55) node{$+$};
\draw (4.55,-.55) node{$-$};
\draw (3.45,-.55)node{$+$};

\draw (2.5,1.1) node{\small{$v_1$}};
\draw (2.5,-1.1) node{\small{$v_2$}};
\draw (2.5,0) node{\small{$u_2$}};
\draw (0,0) node{\small{$u_1$}};

\begin{scope}[xshift =  9 cm, yshift = -2.5cm, scale = .75]

\draw (1.5,8) node {$u_1$};
\draw (1.5,6.5) node {$u_2$};

\draw (-1.5,4.5) node{$v_1$};
\draw (4.5,4.5) node{$v_2$};
\draw (1.5,0.3) node{$v_3$};

\begin{knot}[
	consider self intersections,
 	clip width = 3,
 	ignore endpoint intersections = true,
	end tolerance = 2pt
	]
	\flipcrossings{5, 7, 9, 1}
	\strand (-2,3.5) to [out = 270, in = 180]
	(-.5,0) to [out = 0, in = 180]
	(1.5,1) to [out = 0, in =180]
	(3.5,0) to [out =0, in=270]
	(5,3.5) to [out = 90, in = 0]
	(3,7.5) to [out = 180, in = 90]
	(-.5,5.3) to [out =270, in = 180]
	(.8,4) to [out = 0, in = 180]
	(2.2,5) to [out = 0, in = 90]
	(3.5,3.5) to [out = 270, in = 0]
	(2.2,2) to [out = 180, in = 0]
	(.8,3) to [out = 180, in = 90]
	(-.5,1.5) to [out = 270, in = 180]
	(1.5,0) to [out =0, in =270]
	(3.5,1.5) to [out = 90, in= 0]
	(2.2,3) to [out =180, in = 0]
	(.8,2) to [out =180, in = 270]
	(-.5,3.5) to [out = 90, in =180]
	(.8,5) to [out = 0, in = 180]
	(2.2,4) to [out = 0, in = 270]
	(3.5,5.3) to [out = 90, in = 0]
	(0,7.5) to [out = 180, in = 90]
	(-2,3.5);
	
	\end{knot}
	
\draw[red, dashed, very thick] (1.5,2.5) ellipse (3cm and 3cm);

\end{scope}

\end{tikzpicture}$$
\caption{On the left is a generic almost alternating diagram. The $2$-tangle $R$ is alternating. A label of $+$ on an incoming strand indicates that strand passes over the first other strand it encounters inside $R$. Similarly, a label of $-$ on an incoming strand indicates that stranded passes under the first other strand it encounters inside $R$. On the right is an $A$-almost alternating diagram of a knot $K$ where the alternating tangle $R$ is inside the dashed red circle. The region $v_3$ shares a crossing with each of $v_1$ and $v_2$, and so the diagram is not $B$-almost alternating. However no region in the alternating tangle $R$ shares a crossing with each of $u_1$ and $u_2$, and thus the diagram is $A$-almost alternating.}
\label{figure:aadiagram}
\end{figure}

A \textit{Kauffman state} of $D$ is the set of simple closed curves resulting from a choice at each crossing $\tikz[baseline=.6ex, scale = .4]{
\draw (0,0) -- (1,1);
\draw (1,0) -- (.7,.3);
\draw (.3,.7) -- (0,1);
}
~$ of an $A$-resolution $~ \tikz[baseline=.6ex, scale = .4]{
\draw[rounded corners = 1mm] (0,0) -- (.45,.5) -- (0,1);
\draw[rounded corners = 1mm] (1,0) -- (.55,.5) -- (1,1);
}~$or a $B$-resolution $~ \tikz[baseline=.6ex, scale = .4]{
\draw[rounded corners = 1mm] (0,0) -- (.5,.45) -- (1,0);
\draw[rounded corners = 1mm] (0,1) -- (.5,.55) -- (1,1);
}~$. The state obtained by choosing an $A$-resolution for every crossing is the \textit{all-$A$ state} of $D$, and the state obtained by choosing a $B$-resolution for every crossing is the \textit{all-$B$ state} of $D$. A link diagram $D$ is \textit{$A$-adequate} (respectively \textit{$B$-adequate}) if no two arcs in the $A$-resolution ($B$-resolution) of any crossing lie on the same component of the all-$A$ (all-$B$) state of $D$. A link is \textit{$A$-adequate} (respectively \textit{$B$-adequate}) if it has an $A$-adequate ($B$-adequate) diagram. A link diagram that is both $A$-adequate and $B$-adequate is called \textit{adequate}, and any link having such a diagram is also called \textit{adequate}. A link that is not adequate, but has a diagram that is either $A$-adequate or $B$-adequate is called \textit{semi-adequate}. A link that has no $A$-adequate or $B$-adequate diagrams is called \textit{inadequate}. Lickorish and Thistlethwaite \cite{LT:Adequate} introduced adequate links and proved that an adequate diagram of a link has the fewest number of crossings of any diagram of that link. 

It is straightforward to check that every almost alternating link has Turaev genus one. Kim \cite{Kim:TuraevClassification} proved that every inadequate Turaev genus one link is almost alternating (see Theorem \ref{theorem:TuraevInadequate}). 
\begin{definition}
A Turaev genus one diagram $D$ is \textit{$A$-Turaev genus one} if it is $A$-adequate or $A$-almost alternating. Likewise, a Turaev genus one diagram $D$ is \textit{$B$-Turaev genus one} if it is $B$-adequate or $B$-almost alternating. A Turaev genus one link $L$ is \textit{$A$-Turaev genus one} if it has an $A$-Turaev genus one diagram and is \textit{$B$-Turaev genus one} if it has a $B$-Turaev genus one diagram. 
\end{definition}
Corollary \ref{corollary:ABgt1} states that every Turaev genus one link is either $A$-Turaev genus one or $B$-Turaev genus one.

\subsection{Extremal Khovanov homology}
The Khovanov homology $Kh(L)$ of a link $L$ is a homological generalization of the Jones polynomial of $L$. It is a $\mathbb{Z}$-module equipped with two gradings: the homological grading $i$ and the quantum grading $j$. There is a direct sum decomposition $Kh(L) = \bigoplus_{i,j\in\mathbb{Z}} Kh^{i,j}(L)$ where $Kh^{i,j}(L)$ denotes the summand in homological grading $i$ and quantum grading $j$. The extremal Khovanov homology of a link refers to the Khovanov homology in the maximum or minimum quantum gradings. Define 
\begin{align*}
j_{\min}(L) = & \; \min\{j~|~Kh^{i,j}(L)\neq 0\},\\
j_{\max}(L) = & \; \max\{j~|~Kh^{i,j}(L) \neq 0\},\\
\delta_{\min}(L) = & \; \min\{2i-j~|~Kh^{i,j}(L)\neq 0\},~\text{and}\\
\delta_{\max}(L) = & \; \max \{2i-j~|~Kh^{i,j}(L)\neq 0\}.
\end{align*}
For a fixed $j_0\in\mathbb{Z}$, let $Kh^{*,j_0}(L)$ denote the direct sum $\bigoplus_{i\in\mathbb{Z}}Kh^{i,j_0}(L)$. As a shorthand, we often write $Kh^{*,j_0}(L) = Kh^{i_0,j_0}(L)$ to mean that the Khovanov homology of $L$ in quantum grading $j_0$ is entirely supported in the single homological grading $i_0$. Our main theorem characterizes at least one of the extremal Khovanov homology groups of Turaev genus one, almost alternating, and $A$- or $B$-adequate links.

\begin{theorem}
\label{theorem:Diagonal}
Suppose that $L$ is a non-split link.
\begin{enumerate}
\item If $L$ is $A$-Turaev genus one, $A$-almost alternating, or $A$-adequate, then there is an $i_0\in\mathbb{Z}$ such that $Kh^{*,j_{\min}(L)}(L) = Kh^{i_0,j_{\min}(L)}(L)\cong \mathbb{Z}$ and $2i_0-j_{\min}(L) = \delta_{\min}(L) + 2$.
\item If $L$ is $B$-Turaev genus one, $B$-almost alternating, or $B$-adequate, then there is an $i_0\in\mathbb{Z}$ such that $Kh^{*,j_{\max}(L)}(L) = Kh^{i_0,j_{\max}(L)}(L) \cong \mathbb{Z}$ and 
$2i_0 - j_{\max}(L) = \delta_{\max}(L) -2$.
\end{enumerate}
\end{theorem}

Despite the fact that every almost alternating link has Turaev genus one, we chose to include almost alternating links in the statement of the theorem due to their prominence in the proof. Theorem \ref{theorem:Diagonal} is proven separately for $A$- or $B$-adequate and almost alternating links. Then Kim's result \cite{Kim:TuraevClassification} that says an inadequate Turaev genus one link is almost alternating implies the result for Turaev genus one links. Khovanov \cite{Kh:Patterns} proved that the extremal Khovanov homology of an $A$- or $B$-adequate link is isomorphic to $\mathbb{Z}$. Our contribution in Theorem \ref{theorem:Diagonal} for an $A$- or $B$-adequate link is to specify the $\delta$-grading of the extremal summand.  

The proof of Theorem \ref{theorem:Diagonal} for almost alternating links is quite different than the proof of its Jones polynomial analogue in \cite{DasLow:TuraevJones}. Resolving the dealternator in an almost alternating diagram $D$ as either an $A$-resolution or a $B$-resolution results in alternating diagrams $D_A$ and $D_B$ respectively. The proof in \cite{DasLow:TuraevJones} combines formulas for the first few coefficients of the Jones polynomial of the  $D_A$ and $D_B$ to prove that either the first or last coefficient of the Jones polynomial of $D$ has absolute value one. We initially hoped to use a similar strategy to prove Theorem \ref{theorem:Diagonal}. Since $D_A$ and $D_B$ are alternating links, their Khovanov homology groups $Kh(D_A)$ and $Kh(D_B)$ are well-understood, and the long exact sequence (see Theorem \ref{theorem:LES}) relates $Kh(D_A)$, $Kh(D_B)$, and $Kh(D)$. However, in order to compute $Kh(D)$ in its extremal quantum gradings, one would need not only to know the Khovanov homology of $D_A$ and $D_B$, but also know explicit generators for each homology class in the first and last two quantum gradings of $Kh(D_A)$ and $Kh(D_B)$. Instead we take an inductive approach by resolving crossings in the alternating tangle $R$, resulting in almost alternating diagrams with fewer crossings than $D$. See Theorem \ref{theorem:AAKh} for the details of this result.

Theorem \ref{theorem:Diagonal} has the following corollary.
\begin{corollary}
\label{corollary:Main}
If $L$ is a non-split Turaev genus one link, then its Khovanov homology is isomorphic to $\mathbb{Z}$ in at least one of its extremal quantum gradings.
\end{corollary}

The study of extremal or near extremal Jones coefficients and Khovanov groups has received considerable attention. Bae and Morton \cite{BaeMorton:Jones} developed an algorithm for computing the extremal Jones coefficients from the all-$A$ and all-$B$ states. Gonz\'{a}lez-Meneses and Manch\'{o}n \cite{GMM:ExtremalJonesCoefficients} used Bae and Morton's work to prove that there are prime links with arbitrary extremal Jones coefficients. Gonz\'{a}lez-Meneses, Manch\'{o}n, and Silvero \cite{GMMS:ExtremalKh} gave a geometric description of the extremal Khovanov homology in terms of the all-$A$ and all-$B$ states. Przytycki and Silvero \cite{PS:Homotopy, PS:AlmostExtreme} and Mor\'an and Silvero \cite{MS:Spectra} further study extremal or near extremal Khovanov groups from various different perspectives.

Often, a result involving a knot polynomial can be strengthened in the homological setting. For example, Lickorish and Thistlethwaite \cite{LT:Adequate} proved that the first and last coefficients of the Jones polynomial of an adequate link have absolute value one, while Khovanov \cite{Kh:Patterns} proved that such links have extremal Khovanov homology isomorphic to $\mathbb{Z}$. The span of the Jones polynomial gives a lower bound on the crossing number of a link \cite{Kauffman:StateModels, Murasugi:Jones, Thistlethwaite:Jones}, and using essentially the same proof and basic facts from the construction of Khovanov homology \cite{Kh:Jones}, one can show that the span of the quantum grading of the Khovanov homology of a link gives a sometimes better lower bound on the crossing number. For example, the Jones polynomial lower bound for the crossing number of the knot $10_{132}$ implies that its crossing number is at least 5. However since the Khovanov homology of $10_{132}$ in its maximal quantum grading is isomorphic to $\mathbb{Z}\oplus\mathbb{Z}$ and the two $\mathbb{Z}$-summands cancel in the Euler characteristic, the Khovanov homology lower bound for the crossing number of the knot $10_{132}$ implies that its crossing number is at least 6. Theorem \ref{theorem:Diagonal} is another example of such a result. In addition to knowing that at least one extremal Khovanov homology group of a Turaev genus one link is isomorphic to $\mathbb{Z}$, Theorem \ref{theorem:Diagonal} also gives us information about in which diagonal grading that $\mathbb{Z}$-summand is supported. Example \ref{example:10_132} gives a knot whose Jones polynomial has leading and trailing coefficient of absolute value one, but both of whose extremal Khovanov homology groups have rank two. Examples \ref{example:11n376} and \ref{example:13crossings} give knots and links that have one extremal Khovanov group isomorphic to $\mathbb{Z}$, but that summand is supported in a different diagonal grading than the one specified in Theorem \ref{theorem:Diagonal}. Therefore, each of these examples is inadequate, is not almost alternating, and has Turaev genus at least two.

Theorem \ref{theorem:Diagonal} suggests potential relationships between $A$-/$B$-adequate links, almost alternating links, and Turaev genus one links. Any link $L$ that is almost alternating, Turaev genus one, or $A$- or $B$-adequate satisfies the following.
\begin{enumerate}
\item The leading or trailing coefficient of the Jones polynomial of $L$ has absolute value one \cite{LT:Adequate, DasLow:TuraevJones}.
\item Either the first two or last two coefficients of the Jones polynomial of $L$ alternate in sign \cite{Stoimenow:Semi, LS:AAJones}.
\item The Jones polynomial of $L$ is equal to the Jones polynomial of a trivial $m$-component link if and only if $L$ is the trivial $m$-component link for $m\geq 1$ \cite{Stoimenow:Semi, LS:AAJones}.
\end{enumerate}
The similarities above make it difficult to use invariants to distinguish between these three classes of links. One can show that every almost alternating link is Turaev genus one, and that there are $A$-adequate (and $B$-adequate) links that are neither Turaev genus one nor almost alternating (e.g. the $(3,7)$-torus knot). We ask the following open questions about almost alternating, Turaev genus one, and $A$-/$B$-adequate links. The authors suspect an affirmative answer to both questions.

\begin{question}
 Does there exist a link of Turaev genus one that is not almost alternating?
 \end{question}
 \begin{question}
Does there exist an almost alternating inadequate link?
\end{question}

Odd Khovanov homology is a categorification of the Jones polynomial introduced by Ozsv\'ath, Rasmussen, and Szab\'o \cite{ORS:OddKh}. It coincides with Khovanov homology with mod 2 coefficients, but is, in general, different. Theorem \ref{theorem:OddDiagonal} is a version of Theorem \ref{theorem:Diagonal} for odd Khovanov homology.

\subsection{Applications: signature and maximal Thurston Bennequin number}
Let $\sigma(L)$ be the signature of the link $L$, with the convention that the signature of positive trefoil is $-2$. We prove that a version of Theorem \ref{theorem:Diagonal} involving signature holds for links that are both $A$-Turaev genus one and $B$-Turaev genus one. In Section \ref{section:Signature}, we conjecture a strengthening of this result.

\begin{theorem}
\label{theorem:Signature}
Suppose $L$ is a non-split link that is both $A$-Turaev genus one and $B$-Turaev genus one. At least one of the following statements hold.
\begin{enumerate}
\item There is an $i_0\in\mathbb{Z}$ such that $Kh^{*,j_{\min}(L)}(L) = Kh^{i_0(L),j_{\min}(L)}(L)\cong \mathbb{Z}$ and $2i_0(L)-j_{\min}(L)  = \sigma(L)+1$.
\item There is an $i_0\in\mathbb{Z}$ such that $Kh^{*,j_{\max}(L)}(L) = Kh^{i_0(L),j_{\max}(L)}(L) \cong \mathbb{Z}$ and 
$2i_0(L) - j_{\max}(L)  = \sigma(L)-1$.
\end{enumerate}
\end{theorem}

The Thurston Bennequin number $\tb(\mathcal{L})$ of an oriented Legendrian link $\mathcal{L}$ measures the framing of the contact plane field around $\mathcal{L}$. Among all Legendrian links with a given topological link type $L$, the Thurston Bennequin number $\tb(\mathcal{L})$ is bounded from above. For a topological link $L$, its \textit{maximal Thurston Bennequin number} $\mtb(L)$ is defined to be the maximum of $\tb(\mathcal{L})$ over all Legendrian links $\mathcal{L}$ whose topological link type is $L$. Ng \cite{Ng:TBKh} proved that the Khovanov homology gives an upper bound on the maximal Thurston Bennequin number of a link and used that bound to compute $\mtb(L)$ for any non-split alternating link $L$. We generalize Ng's approach to Turaev genus one links. 

\begin{theorem}
\label{theorem:TBAA}
Let $L$ be a Turaev genus one link with diagram $D$. Let $w(D)$ be the writhe of $D$, and let $s_A(D)$ and $s_B(D)$ be the number of components in the all-$A$ and all-$B$ Kauffman states of $D$ respectively. 
\begin{itemize}
\item If $D$ is $A$-Turaev genus one, then 
$$w(D) - s_A(D)  \leq  \mtb(L)  \leq  w(D) - s_A(D) + 1.$$
\item If $D$ is $B$-Turaev genus one and $\overline{L}$ is the mirror of $L$, then
$$-w(D) -s_B(D) \leq  \mtb(\overline{L})  \leq  -w(D) - s_B(D) + 1.$$
\end{itemize}
\end{theorem}

This paper is organized as follows. In Section \ref{section:KhBackground}, we review the construction of Khovanov homology. In Section \ref{section:TGAA}, we give background information on Turaev genus one and almost alternating links. We also prove that every almost alternating link is $A$-almost alternating or $B$-almost alternating and thus that every Turaev genus one link is $A$-Turaev genus one or $B$-Turaev genus one. In Section \ref{section:Extremal}, we prove Theorem \ref{theorem:Diagonal}. In Section \ref{section:Examples}, we give examples of knots and links such that Theorem \ref{theorem:Diagonal} implies they are inadequate, not Turaev genus one, and not almost alternating. In Section \ref{section:Signature}, we prove Theorem \ref{theorem:Signature} and discuss progress on a related conjecture. Finally, in Section \ref{section:TB}, we recall some facts about Thurston Bennequin numbers and prove Theorem \ref{theorem:TBAA}.

\section{Khovanov homology background}
\label{section:KhBackground}

In this section, we briefly review the construction of Khovanov homology and a few Khovanov homology results relevant to the current paper. For a more detailed account, the reader can refer to one of \cite{Kh:Jones, BN:Khovanov, Viro:Khovanov, Kauffman:Khovanov, Turner:Khovanov}. An expert reader may wish to skim the remainder of this section.

The version of the Jones polynomial that Khovanov homology categorifies is defined by the following rules:
\begin{enumerate}
\item $\left\langle~
\tikz[baseline=.6ex, scale = .4]{
\draw (0,0) -- (1,1);
\draw (1,0) -- (.7,.3);
\draw (.3,.7) -- (0,1);
}
~\right\rangle = \left\langle ~ \tikz[baseline=.6ex, scale = .4]{
\draw[rounded corners = 1mm] (0,0) -- (.45,.5) -- (0,1);
\draw[rounded corners = 1mm] (1,0) -- (.55,.5) -- (1,1);
}~\right\rangle - q \left\langle~ \tikz[baseline=.6ex, scale = .4]{
\draw[rounded corners = 1mm] (0,0) -- (.5,.45) -- (1,0);
\draw[rounded corners = 1mm] (0,1) -- (.5,.55) -- (1,1);
}~\right\rangle,$
\item $\left\langle~D\sqcup \bigcirc ~\right\rangle = (q+q^{-1})\left\langle D \right\rangle,$
\item $\left\langle ~ \bigcirc ~\right\rangle = q+q^{-1},$
\item $V_L(q) = (-1)^{c_-(D)} q^{c_+(D) - 2c_-(D)} \langle D \rangle$,
\end{enumerate}
where $c_+(D)$ is the number of positive $\left(~\tikz[baseline=.6ex, scale = .4]{
\draw[->] (0,0) -- (1,1);
\draw (1,0) -- (.7,.3);
\draw[->] (.3,.7) -- (0,1);
}~\right)$  crossings in $D$
and $c_-(D)$ is the number of negative $\left(~\tikz[baseline=.6ex, scale = .4]{
\draw[->] (.7,.7) -- (1,1);
\draw[->] (1,0) -- (0,1);
\draw (0,0) -- (.3,.3);
}~\right)$ crossings in $D$. If we let $\widetilde{V}_L(t)$ be the version of the Jones polynomial in \cite{Kauffman:StateModels}, then $V_L(q) = (q+q^{-1})~\widetilde{V}_L(q^2).$ See \cite{Kauffman:Khovanov} for more discussion on the relationship between these two normalizations of the Jones polynomial.

Let $\mathcal{S}(D)$ be the set of Kauffman states of $D$, and let $s \in \mathcal{S}(D)$ be a Kauffman state. Define $a(s)$ and $b(s)$ to be the number of $A$-resolutions and the number of $B$-resolutions respectively in $s$, and define $|s|$ to be the number of components in $s$. The Kauffman bracket of $D$ can be represented by the state sum formula
$$\langle D \rangle = \sum_{s \in \mathcal{S}(D)} (-q)^{b(s)}(q+q^{-1})^{|s|}.$$

An \textit{enhanced state} $S$ of $D$ is a Kauffman state $s$ where each component is labeled either $1$ or $x$. Let $\mathcal{S}_{\operatorname{en}}(D)$ be the set of enhanced states of $D$. Define $a(S)$, $b(S)$, and $|S|$ to be $a(s)$, $b(s)$, and $|s|$ respectively where $s$ is the underlying Kauffman state of the enhanced state $S$. Furthermore, define $\theta(S)$ to be the difference between the number of $1$ labels and the number of $x$ labels on the enhanced state $S$. Then the Kauffman bracket of $D$ can be represented as the sum of monomials
$$\langle D \rangle = \sum_{S\in \mathcal{S}_{\operatorname{en}}(D)} (-1)^{b(S)} q^{b(S) + \theta(S)}.$$

In order to construct Khovanov homology, we first construct a categorification $\underline{Kh}(D)$ of $\langle D \rangle$ (which is not a link invariant). Then we shift gradings to obtain the link invariant known as Khovanov homology and equivalently denoted as $Kh(L)$ or $Kh(D)$. 

The homological grading $i(S)$ of the enhanced state $S$ is defined to be $i(S)=b(S)$, and the quantum grading $j(S)$ of $S$ is defined to be $j(S)=b(S) + \theta(S)$. Let $R$ be a commutative ring with identity, and let $CKh^{i,j}(D;R)$ be the free $R$-module with basis the enhanced states $S$ with homological grading $i$ and quantum grading $j$. Define $CKh(D;R) = \bigoplus_{i,j \in \mathbb{Z}} CKh^{i,j}(D;R).$ The ring $R$ will most frequently be the integers $\mathbb{Z}$, and in that case, we drop $R$ from the notation. For the sake of notational brevity, we describe the construction of Khovanov homology over $\mathbb{Z}$, but $\mathbb{Z}$ could be replaced with $R$ throughout to give $Kh(D;R)$ and $\underline{Kh}(D;R)$, Khovanov homology and unshifted Khovanov homology with coefficients in the ring $R$.

The differential in this complex is a map $d^{i,j}:CKh^{i,j}(D) \to CKh^{i+1,j}(D)$, i.e. it is a $\mathbb{Z}$-module map that increases the homological grading by one and preserves the quantum grading. The map $d^i$ is defined on enhanced states and then extended linearly. 
For any two enhanced states $S_0$ and $S_1$ with underlying Kauffman states $s_0$ and $s_1$ respectively, we define the \textit{incidence number} $\delta(S_0,S_1)$ as follows. Unless $s_1$ can be obtained from $s_0$ by changing a single resolution from an $A$-resolution to a $B$-resolution, then we define $\delta(S_0,S_1)=0$. Suppose that $s_1$ can be obtained from $s_0$ by changing a single $A$-resolution to a $B$-resolution. Then $s_1$ can be obtained from $s_0$ by either merging two components into one or by splitting one component into two. Call all other components of $s_0$ and $s_1$ constant components. Unless all labels on the constant components of $s_0$ are the same as they are in $s_1$, then $\delta(S_0,S_1)=0$. The cases where $\delta(S_0,S_1)=1$ are illustrated in Figure \ref{figure:Incidence} and described in the next two paragraphs. 

Suppose that $s_1$ can be obtained from $s_0$ by merging two components into one, and that all constant components are labeled the same in $S_0$ and $S_1$. If the two non-constant components in $S_0$ are both labeled $1$ and the non-constant component in $S_1$ is labeled $1$, then $\delta(S_0,S_1)=1$. If one of the non-constant components in $S_0$ is labeled $1$ and the other is labeled $x$ and the non-constant component in $S_1$ is labeled $x$, then $\delta(S_0,S_1)=1$. If both of the non-constant components in $S_0$ are labeled $x$, then $\delta(S_0,S_1)=0$.

Suppose that $s_1$ can be obtained from $s_0$ by splitting one component into two, and that all constant components are labeled the same in $S_0$ and $S_1$. If the non-constant component in $S_0$ is labeled $1$, one of the non-constant components in $S_1$ is labeled $1$, and the other non-constant component in $S_1$ is labeled $x$, then $\delta(S_0,S_1)=1$. If the non-constant component in $S_0$ is labeled $x$ and both non-constant components in $S_1$ are labeled $x$, then $\delta(S_0,S_1)=1$.

\begin{figure}[h]
$$\begin{tikzpicture}[thick]

\draw (1.5,2.5) node{$S_0$};
\draw (5.5,2.5) node{$S_1$};
\draw (9.5,2.5) node{$S_0$};
\draw (13,2.5) node{$S_1$};


\draw[white!50!black] (1,0) -- (2,1);
\draw[white!50!black] (2,0) -- (1.7,.3);
\draw[white!50!black] (1,1) -- (1.3,.7);

\draw[rounded corners = 3mm] (.5,1) -- (.8,1) -- (1.3,.5) -- (.8,0) -- (.5,0);
\draw[rounded corners = 3mm] (2.5,1) -- (2.2,1) -- (1.7,.5) -- (2.2,0) -- (2.5,0);
\draw[dashed, thick] (.5,1) arc (90:270:.4cm and .5cm);
\draw[dashed, thick] (2.5,1) arc (90:-90:.4cm and .5cm);
\draw(1,.5) node{\small{$1$}};
\draw(2,.5) node{\small{$1$}};

\begin{scope}[xshift = 4cm]
\draw[white!50!black] (1,0) -- (2,1);
\draw[white!50!black] (2,0) -- (1.7,.3);
\draw[white!50!black] (1,1) -- (1.3,.7);

\draw[rounded corners = 2.4mm] (.8,1.2) -- (1,1.2) -- (1.5,.7) -- (2,1.2) -- (2.2,1.2);
\draw[rounded corners = 2.4mm] (.8,-.2) -- (1,-.2) -- (1.5,.3) -- (2,-.2) -- (2.2,-.2);
\draw[dashed, thick] (.8,1.2) arc (90:270:.5cm and .7cm);
\draw[dashed, thick] (2.2,1.2) arc (90:-90:.5cm and .7cm); 
\draw (1.5,0) node{\small{$1$}};
\end{scope}

\draw[ultra thick, ->] (3.2,.5) -- (4,.5);

\begin{scope}[yshift = -3cm]

\draw[white!50!black] (1,0) -- (2,1);
\draw[white!50!black] (2,0) -- (1.7,.3);
\draw[white!50!black] (1,1) -- (1.3,.7);

\draw[rounded corners = 3mm] (.5,1) -- (.8,1) -- (1.3,.5) -- (.8,0) -- (.5,0);
\draw[rounded corners = 3mm] (2.5,1) -- (2.2,1) -- (1.7,.5) -- (2.2,0) -- (2.5,0);
\draw[dashed, thick] (.5,1) arc (90:270:.4cm and .5cm);
\draw[dashed, thick] (2.5,1) arc (90:-90:.4cm and .5cm);
\draw(1,.5) node{\small{$1$}};
\draw(2,.5) node{\small{$x$}};

\begin{scope}[xshift = 4cm]
\draw[white!50!black] (1,0) -- (2,1);
\draw[white!50!black] (2,0) -- (1.7,.3);
\draw[white!50!black] (1,1) -- (1.3,.7);

\draw[rounded corners = 2.4mm] (.8,1.2) -- (1,1.2) -- (1.5,.7) -- (2,1.2) -- (2.2,1.2);
\draw[rounded corners = 2.4mm] (.8,-.2) -- (1,-.2) -- (1.5,.3) -- (2,-.2) -- (2.2,-.2);
\draw[dashed, thick] (.8,1.2) arc (90:270:.5cm and .7cm);
\draw[dashed, thick] (2.2,1.2) arc (90:-90:.5cm and .7cm); 
\draw (1.5,0) node{\small{$x$}};
\end{scope}

\draw[ultra thick, ->] (3.2,.5) -- (4,.5);
\end{scope}

\begin{scope}[yshift = -6cm]

\draw[white!50!black] (1,0) -- (2,1);
\draw[white!50!black] (2,0) -- (1.7,.3);
\draw[white!50!black] (1,1) -- (1.3,.7);

\draw[rounded corners = 3mm] (.5,1) -- (.8,1) -- (1.3,.5) -- (.8,0) -- (.5,0);
\draw[rounded corners = 3mm] (2.5,1) -- (2.2,1) -- (1.7,.5) -- (2.2,0) -- (2.5,0);
\draw[dashed, thick] (.5,1) arc (90:270:.4cm and .5cm);
\draw[dashed, thick] (2.5,1) arc (90:-90:.4cm and .5cm);
\draw(1,.5) node{\small{$x$}};
\draw(2,.5) node{\small{$1$}};

\begin{scope}[xshift = 4cm]
\draw[white!50!black] (1,0) -- (2,1);
\draw[white!50!black] (2,0) -- (1.7,.3);
\draw[white!50!black] (1,1) -- (1.3,.7);

\draw[rounded corners = 2.4mm] (.8,1.2) -- (1,1.2) -- (1.5,.7) -- (2,1.2) -- (2.2,1.2);
\draw[rounded corners = 2.4mm] (.8,-.2) -- (1,-.2) -- (1.5,.3) -- (2,-.2) -- (2.2,-.2);
\draw[dashed, thick] (.8,1.2) arc (90:270:.5cm and .7cm);
\draw[dashed, thick] (2.2,1.2) arc (90:-90:.5cm and .7cm); 
\draw (1.5,0) node{\small{$x$}};
\end{scope}

\draw[ultra thick, ->] (3.2,.5) -- (4,.5);
\end{scope}


\begin{scope}[xshift = 8cm]

\draw[white!50!black] (1,0) -- (2,1);
\draw[white!50!black] (2,0) -- (1.7,.3);
\draw[white!50!black] (1,1) -- (1.3,.7);

\draw[rounded corners = 2.5mm] (.8,1.2) -- (.8,1) -- (1.3,.5) -- (.8,0) -- (.8,-.2);
\draw[rounded corners = 2.5mm] (2.2,1.2) -- (2.2,1) --(1.7,.5) -- (2.2,0) -- (2.2,-.2);
\draw[dashed, thick] (.8,1.2) arc (180:0:.7cm and .5cm);
\draw[dashed, thick] (.8,-.2) arc (-180:0:.7cm and .5cm);
\draw (1,.5) node{\small{$1$}};

\begin{scope}[xshift = 3.5cm]
\draw[white!50!black] (1,0) -- (2,1);
\draw[white!50!black] (2,0) -- (1.7,.3);
\draw[white!50!black] (1,1) -- (1.3,.7);

\draw[rounded corners = 2.5mm] (1,1.4) -- (1,1.2) -- (1.5,.7)  -- (2,1.2) -- (2,1.4);
\draw[rounded corners = 2.5mm] (1,-.4) -- (1,-.2) -- (1.5,.3) -- (2,-.2) -- (2,-.4);
\draw [dashed, thick] (1,1.4) arc (180:0:.5cm and .4cm);
\draw [dashed, thick] (1,-.4) arc (-180:0:.5cm and .4cm);

\draw (1.5,1) node{\small{$1$}};
\draw (1.5,0) node{\small{$x$}};

\end{scope}

\draw[ultra thick, ->] (3,.5) -- (3.8,.5);

\end{scope}


\begin{scope}[xshift = 8cm, yshift = -3cm]

\draw[white!50!black] (1,0) -- (2,1);
\draw[white!50!black] (2,0) -- (1.7,.3);
\draw[white!50!black] (1,1) -- (1.3,.7);

\draw[rounded corners = 2.5mm] (.8,1.2) -- (.8,1) -- (1.3,.5) -- (.8,0) -- (.8,-.2);
\draw[rounded corners = 2.5mm] (2.2,1.2) -- (2.2,1) --(1.7,.5) -- (2.2,0) -- (2.2,-.2);
\draw[dashed, thick] (.8,1.2) arc (180:0:.7cm and .5cm);
\draw[dashed, thick] (.8,-.2) arc (-180:0:.7cm and .5cm);
\draw (1,.5) node{\small{$1$}};

\begin{scope}[xshift = 3.5cm]
\draw[white!50!black] (1,0) -- (2,1);
\draw[white!50!black] (2,0) -- (1.7,.3);
\draw[white!50!black] (1,1) -- (1.3,.7);

\draw[rounded corners = 2.5mm] (1,1.4) -- (1,1.2) -- (1.5,.7)  -- (2,1.2) -- (2,1.4);
\draw[rounded corners = 2.5mm] (1,-.4) -- (1,-.2) -- (1.5,.3) -- (2,-.2) -- (2,-.4);
\draw [dashed, thick] (1,1.4) arc (180:0:.5cm and .4cm);
\draw [dashed, thick] (1,-.4) arc (-180:0:.5cm and .4cm);

\draw (1.5,1) node{\small{$x$}};
\draw (1.5,0) node{\small{$1$}};

\end{scope}

\draw[ultra thick, ->] (3,.5) -- (3.8,.5);

\end{scope}


\begin{scope}[xshift = 8cm, yshift = -6cm]

\draw[white!50!black] (1,0) -- (2,1);
\draw[white!50!black] (2,0) -- (1.7,.3);
\draw[white!50!black] (1,1) -- (1.3,.7);

\draw[rounded corners = 2.5mm] (.8,1.2) -- (.8,1) -- (1.3,.5) -- (.8,0) -- (.8,-.2);
\draw[rounded corners = 2.5mm] (2.2,1.2) -- (2.2,1) --(1.7,.5) -- (2.2,0) -- (2.2,-.2);
\draw[dashed, thick] (.8,1.2) arc (180:0:.7cm and .5cm);
\draw[dashed, thick] (.8,-.2) arc (-180:0:.7cm and .5cm);
\draw (1,.5) node{\small{$x$}};

\begin{scope}[xshift = 3.5cm]
\draw[white!50!black] (1,0) -- (2,1);
\draw[white!50!black] (2,0) -- (1.7,.3);
\draw[white!50!black] (1,1) -- (1.3,.7);

\draw[rounded corners = 2.5mm] (1,1.4) -- (1,1.2) -- (1.5,.7)  -- (2,1.2) -- (2,1.4);
\draw[rounded corners = 2.5mm] (1,-.4) -- (1,-.2) -- (1.5,.3) -- (2,-.2) -- (2,-.4);
\draw [dashed, thick] (1,1.4) arc (180:0:.5cm and .4cm);
\draw [dashed, thick] (1,-.4) arc (-180:0:.5cm and .4cm);

\draw (1.5,1) node{\small{$x$}};
\draw (1.5,0) node{\small{$x$}};

\end{scope}

\draw[ultra thick, ->] (3,.5) -- (3.8,.5);

\end{scope}
\end{tikzpicture}$$
\caption{Pairs of enhanced states $(S_0,S_1)$ such that the incidence number $\delta(S_0,S_1)=1$. The dashed arcs indicate how the arcs of the Kauffman state are  connected globally.}
\label{figure:Incidence}
\end{figure}
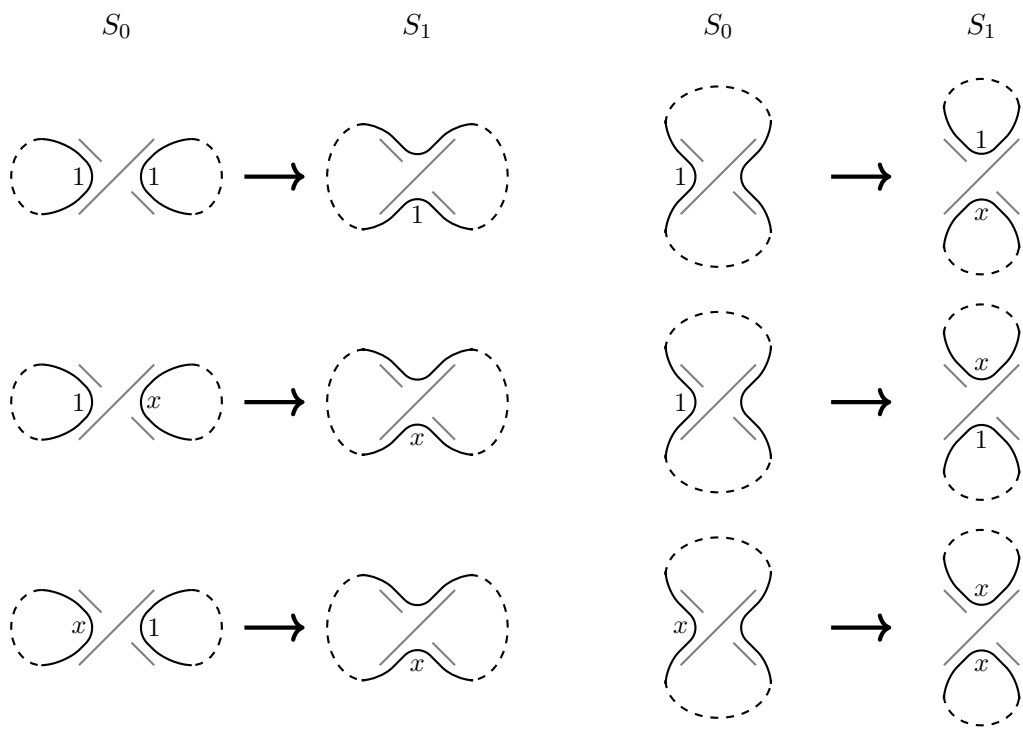

The last ingredient to define the differential is to assign signs to certain pairs of enhanced states. Arbitrarily number the crossings of $D$ from $1$ to $c=c(D)$. Suppose that $S_0$ and $S_1$ are enhanced states with underlying Kauffman states $s_0$ and $s_1$ such that 
\begin{enumerate}
\item $s_1$ can be obtained from $s_0$ by changing a single $A$-resolution to a $B$-resolution at crossing $k$ for some $k$ where $1\leq k \leq c$, and
\item the number of $B$-resolutions in $s_0$ associated to crossings with labels strictly less than $k$ is odd.
\end{enumerate}
Then define $\varepsilon(S_0,S_1)=-1$. In all other cases, define $\varepsilon(S_0,S_1)=1$.

The differential $d^{i,j}:CKh^{i,j}(D)\to CKh^{i+1,j}(D)$ is defined by
$$d^{i,j}(S_0) = \sum_{\substack{S_1\in\mathcal{S}_{\operatorname{en}}(D) \\
i(S_1)=i(S_0)+1\\
j(S_1)=j(S_0)}}
 \varepsilon(S_0,S_1)\delta(S_0,S_1)\; S_1.$$
In other words, in the differential of $S_0$, we sum over all enhanced states $S_1$ with the same quantum grading and whose homological grading is one greater than that of $S_0$. The term $S_1$ appears with nonzero coefficient if and only if the incidence number $\delta(S_0,S_1)\neq 0$. The coefficient of the term is $\pm 1$ as determined by the sign $\varepsilon(S_0,S_1)$.

It can be checked that $d^{i+1,j}\circ d^{i,j}:CKh^{i,j}(D)\to CKh^{i+2,j}(D)$ is the zero map. Thus for each $j$ we have a chain complex. Define the \textit{unshifted Khovanov homology} of $D$ to be the homology
$$\underline{Kh}^{i,j}(D) = \operatorname{ker} d^{i,j}/ \operatorname{im} d^{i-1,j}.$$
Suppose that $D$ is a diagram of the link $L$ such that $D$ has $c_+=c_+(D)$ positive crossings and $c_-=c_-(D)$ negative crossings. The Khovanov homology of $L$, denoted equivalently as $Kh(L)$ or $Kh(D)$, is defined by $Kh(L)=\bigoplus_{i,j \in \mathbb{Z}}Kh^{i,j}(L)$ where
\begin{equation}
\label{equation:shift}
Kh^{i,j}(L) = \underline{Kh}^{i+c_-,j-c_+ + 2c_-}(D).
\end{equation}

Although unshifted Khovanov homology $\underline{Kh}(D)$ is not a link invariant, it will be useful to describe how it changes under diagrammatic moves that preserve the link type, specifically, under the Reidemeister moves and under flyping. A \textit{flype} is the move on a $2$-tangle $R$ depicted in Figure \ref{figure:Flype}.
\begin{figure}[h]
$$\begin{tikzpicture}[thick]

\draw[rounded corners = 4mm] (3,.5) -- (-1,.5) -- (-2,-.5) -- (-3,-.5);
\draw[rounded corners = 4mm] (-3,.5) -- (-2,.5) -- (-1.7,.2);
\draw[rounded corners = 4mm] (-1.3,-.2) -- (-1,-.5) -- (3,-.5);

\fill[white!70!black] (0,0) circle (1cm);
\draw (0,0) circle (1cm);
\draw (0,0) node {$R$};

\draw[ultra thick, ->] (4,0) -- (5,0);

\begin{scope}[xshift = 9cm]
\draw[rounded corners = 4mm] (-3,.5) -- (1,.5) -- (1.3,.2);
\draw[rounded corners = 4mm] (1.7,-.2) -- (2,-.5) -- (3,-.5);
\draw[rounded corners = 4mm] (-3,-.5) -- (1,-.5) -- (2,.5) -- (3,.5);

\fill[white!70!black] (0,0) circle (1cm);
\draw (0,0) circle (1cm);
\draw (0,0) node {\reflectbox{\rotatebox[origin=c]{180}{$R$}}};
\end{scope}

\end{tikzpicture}$$
\caption{A flype of the tangle $R$.}
\label{figure:Flype}
\end{figure}
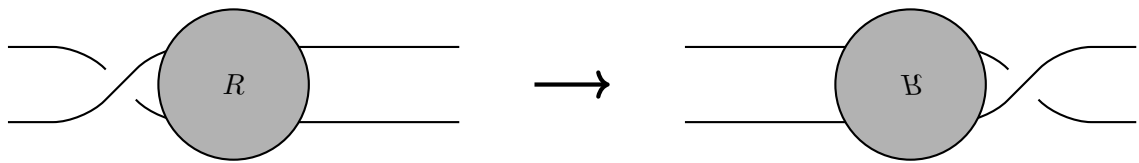

Describing how $\underline{Kh}(D)$ changes amounts to computing the change in the number of positive and negative crossings caused by the diagrammatic move. Equation \ref{equation:shift} implies the following proposition.

\begin{proposition}
\label{prop:Reidemeister}
Reidemeister moves and flypes change $\underline{Kh}(D)$ in the following ways.
\begin{itemize}
\item Positive Reidemeister 1: $\underline{Kh}^{i,j}\left(~
\tikz[baseline=.6ex, scale = .6]{
\draw[thick] (0,0) -- (0,1);
}~\right) \cong \underline{Kh}^{i,j-1}\left(~
\tikz[baseline = .6ex, scale = .6]{
\draw[thick, rounded corners = .8mm] (0,0) -- (0,.5) -- (.2,.7) -- (.4,.5) -- (.2,.3) -- (.1,.4);
\draw[thick, rounded corners = .8mm] (-.03,.55) -- (-.1,.7) -- (0,1);
}\right).$
\item Negative Reidemeister 1: $\underline{Kh}^{i,j}\left(~
\tikz[baseline=.6ex, scale = .6]{
\draw[thick] (0,0) -- (0,1);
}~\right) \cong \underline{Kh}^{i+1,j+2}\left(~
\tikz[baseline = .6ex, scale = .6]{
\draw[thick, rounded corners = .8mm] (0,1) -- (0,.5) -- (.2,.3) -- (.4,.5) -- (.2,.7) -- (.1,.6);
\draw[thick, rounded corners = .8mm] (-.03,.45) -- (-.1,.3) -- (0,0);
}\right).$ 
\item Reidemeister 2: $\underline{Kh}^{i,j}\left(~
\tikz[baseline=.6ex, scale = .6]{
\draw[thick] (0,0) -- (0,1);
\draw[thick] (0.5,0) -- (0.5,1);
}~\right) \cong \underline{Kh}^{i+1,j+1}\left(~
\tikz[baseline=.6ex, scale = .6]{
\draw[thick, rounded corners = 2mm] (0,0) -- (.5,.5) -- (0,1);
\draw[thick] (.5,1) -- (.3,.8);
\draw[thick] (.3,.2) -- (.5,0);
\draw[thick, rounded corners = 1mm] (.2,.7) -- (0,.5) -- (.2,.3);
}~\right).$

\item A flype or a Reidemeister move of type 3 does not change the unshifted Khovanov homology $\underline{Kh}(D)$.
\end{itemize}
\end{proposition}

Let $D$ be a link diagram with $c=c(D)$ crossings, and let $\overline{D}$ be the mirror of $D$. The Khovanov complex $CKh(\overline{D})$ is the dual of the complex $CKh(D)$. Consequently, the following isomorphisms hold:
\begin{align}
\begin{split}
\label{equation:Mirror}
\underline{Kh}^{i,j}(\overline{D};\mathbb{Q}) \cong & \; \underline{Kh}^{c-i,c-j}(D;\mathbb{Q})~\text{and}\\
\operatorname{Tor}\underline{Kh}^{i,j}(\overline{D}) \cong & \; \operatorname{Tor} \underline{Kh}^{c-i+1,c-j}(D).
\end{split}
\end{align}

For some crossing $X$ of $D$, let $D_A$ and $D_B$ represent the diagrams obtained by choosing an $A$-resolution and a $B$-resolution at the crossing $X$ respectively and not changing any of the other crossings. There is a natural bijection between the enhanced states of $D$ and the disjoint union of the enhanced states of $D_A$ and $D_B$: $\mathcal{S}_{\operatorname{en}}(D) \stackrel{1:1}{\longleftrightarrow} \mathcal{S}_{\operatorname{en}}(D_A)\sqcup \mathcal{S}_{\operatorname{en}}(D_B).$ Define $f:CKh(D_B)\to CKh(D)$ by $f(S)= S$ and extending linearly, where the first $S$ is an enhanced state considered as an element of $\mathcal{S}_{\operatorname{en}}(D_B)$ and the second $S$ is considered as an element of $\mathcal{S}_{\operatorname{en}}(D)$. The map $f$ increases both homological and quantum gradings by one, i.e. it is a map $f:CKh^{i-1,j-1}(D_B)\to CKh^{i,j}(D)$ for each $i$ and $j$. Define the map $g:CKh(D)\to CKh(D_A)$ by
$$g(S) = 
\begin{cases}
0 & \text{if $S$ has a $B$-resolution at crossing $X$,}\\
S & \text{if $S$ has an $A$-resolution at crossing $X$,}
\end{cases}$$
and extending linearly. The map $g$ preserves both homological and quantum gradings, i.e. it is a map $g:CKh^{i,j}(D)\to CKh^{i,j}(D_A)$. Both of the maps $f$ and $g$ are chain complex maps, that is both $f$ and $g$ commute with the differential. Also, $f$ is an injection, and $g$ is a surjection. Since $\operatorname{ker} g = \operatorname{im} f$, we have a short exact sequence of complexes
$$0 \to CKh(D_B)\xrightarrow{f} CKh(D) \xrightarrow{g} CKh(D_A)\to 0.$$
The long exact sequence in homology associated to this short exact sequence of complexes is given in the following theorem of Khovanov \cite{Kh:Jones}.
\begin{theorem}[Khovanov]
\label{theorem:LES}
Let $D$ be a link diagram, and let $D_A$ and $D_B$ be the $A$-resolution and $B$-resolution respectively of $D$ at a chosen crossing. For each $j$, there is a long exact sequence of unshifted Khovanov homology
$$\cdots \to \underline{Kh}^{i-1,j-1}(D_B) \xrightarrow{f_*} \underline{Kh}^{i,j}(D) \xrightarrow{g_*} \underline{Kh}^{i,j}(D_A) \xrightarrow{\partial} \underline{Kh}^{i,j-1}(D_B)\to\cdots$$
where $\partial$ is the boundary map in the snake lemma.
\end{theorem}
In other treatments of Khovanov homology, one typically sees two long exact sequences in homology: one for a positive crossing and one for a negative crossing. When working with unshifted Khovanov homology, there is only the single long exact sequence above. The long exact sequence above is at the center of most of the Khovanov homology computations in this paper. 

Lee \cite{Lee:Alternating} used the long exact sequence in Theorem \ref{theorem:LES} to prove that the Khovanov homology of a non-split alternating links is supported on two diagonals. Lee \cite{Lee:Khovanov} also defined a map $d_L:CKh^{i,j}(D;\mathbb{Q})\to CKh^{i+1,j+4}(D;\mathbb{Q})$ on the Khovanov complex with rational coefficients that anti-commutes with the Khovanov differential. The differential $d+d_L$ is non-decreasing with respect to the quantum grading and so one can define a filtration on the complex $CKh(D;\mathbb{Q})$ with differential $d+d_L$. Rasmussen \cite{Rasmussen:Slice} proved that this gives rise to a spectral sequence whose $E_2$ term is $Kh(L;\mathbb{Q})$ and that converges to $\mathbb{Q}^{2^\ell}$ where $L$ is an $\ell$-component link. An analysis of gradings yields the following theorem.
\begin{theorem}[Lee, Rasmussen]
\label{theorem:KnightMove}
Suppose that $\operatorname{rank} Kh^{i,j}(L) > 0$. Then $\operatorname{rank} Kh^{p,q}(L)>0$ where $(p,q) = (i,j+2), (i, j-2), (i+1,j+4k),$ or $(i-1,j-4k)$ for some positive integer $k$.
\end{theorem}
\noindent Theorem \ref{theorem:KnightMove} also applies to unshifted Khovanov homology $\underline{Kh}(D)$.

Recall that a link diagram $D$ is $A$-adequate if no two arcs in the $A$-resolution of any crossing lie on the same component of the all-$A$ state of $D$.  Similarly a link diagram is $B$-adequate if no two arcs in the $B$-resolution of any crossing lie on the same component of the all-$B$ state. A reduced alternating diagram is both $A$-adequate and $B$-adequate. Define $j_{\min}(D)$ to be the minimum quantum grading where $\underline{Kh}^{*,j}(D)$ is nontrivial. Khovanov \cite{Kh:Patterns} proves the following theorem. We recall its proof since it is short and plays an important role in several later proofs.
\begin{theorem}[Khovanov]
\label{theorem:ExtAdeq}
Let $D$ be a link diagram, and let $s_A(D)$ and $s_B(D)$ be the number of components in the all-$A$ and all-$B$ resolutions of $D$ respectively.
\begin{enumerate}
\item If $D$ is $A$-adequate, then $j_{\min}(D)= -s_A(D)$ and $$\underline{Kh}^{*,-s_A(D)}(D) = \underline{Kh}^{0,-s_A(D)}(D)\cong \mathbb{Z}.$$ 
\item If $D$ is $B$-adequate, then $j_{\max}(D) = c(D) + s_B(D)$ and $$\underline{Kh}^{*,c(D)+s_B(D)}(D) = \underline{Kh}^{c(D),c(D)+s_B(D)}(D)\cong \mathbb{Z}.$$
\end{enumerate}
\end{theorem}
\begin{proof}
Suppose that $s_0$ and $s_1$ are Kauffman states of an arbitrary diagram $D'$ such that $s_1$ can be obtained from $s_0$ by changing one $A$-resolution into a $B$-resolution. Then $|s_1|=|s_0| \pm 1$. The minimum quantum grading associated to an enhanced state $S_0$ with underlying Kauffman state either $s_0$ is $j(S_0)=b(s_0) - |s_0|$, and likewise, the minimum for an enhanced state $S_1$ with underlying state $s_1$ is $j(S_1)=b(s_1) - |s_1| = b(s_0) +1 - |s_1|.$

Since $D$ is $A$-adequate, each Kauffman state with exactly one $B$-resolution has $s_A(D) - 1$ components. Thus, the only enhanced state with quantum grading $-s_A(D)$ is the all-$A$ state with each component labeled with an $x$. Furthermore, there is no enhanced state with quantum grading less than $-s_A(D)$. Thus $j_{\min}(D)=-s_A(D)$. Since there is only one state in quantum grading $-s_A(D)$ and that state is in homological grading $0$, it follows that $\underline{Kh}^{*,-s_A(D)}(D) = \underline{Kh}^{0,-s_A(D)}(D)\cong \mathbb{Z}$. Statement (2) of the theorem is implied by statement (1) and Equation \ref{equation:Mirror}.
\end{proof}

\section{Turaev genus one and almost alternating links}
\label{section:TGAA}

In this section, we construct the Turaev surface, discuss results about the Turaev genus of a link (especially Turaev genus one links), and explore the connection between Turaev genus one and almost alternating links. Finally, we prove that every almost alternating link is $A$-almost alternating or $B$-almost alternating, and consequently, every Turaev genus one link is $A$-Turaev genus one or $B$-Turaev genus one.

\subsection{Definitions and lower bounds}

The Turaev surface of a link diagram $D$ is constructed as follows. Consider $D$ to be embedded on a sphere $S^2$ sitting inside of $S^3$. Embed the all-$A$ and all-$B$ states of $D$ in a neighborhood of $S^2$ but on opposite sides in $S^3$. Construct a cobordism between the all-$A$ and all-$B$ states that has saddles near the crossings of $D$ (see Figure \ref{figure:Saddle}) and consists of bands away from the crossings. The intersection of the cobordism with the sphere $S^2$ is precisely the diagram $D$. Capping off each boundary component of the cobordism with a disk forms the \textit{Turaev surface} of $D$. 

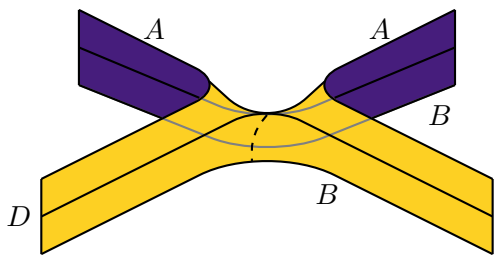
\begin{figure}[h]
$$\begin{tikzpicture}
\begin{scope}[thick]
\draw [rounded corners = 10mm] (0,0) -- (3,1.5) -- (6,0);
\draw (0,0) -- (0,1);
\draw (6,0) -- (6,1);
\draw [rounded corners = 5mm] (0,1) -- (2.5, 2.25) -- (0.5, 3.25);
\draw [rounded corners = 5mm] (6,1) -- (3.5, 2.25) -- (5.5,3.25);
\draw [rounded corners = 5mm] (0,.5) -- (3,2) -- (6,.5);
\draw [rounded corners = 7mm] (2.23, 2.3) -- (3,1.6) -- (3.77,2.3);
\draw (0.5,3.25) -- (0.5, 2.25);
\draw (5.5,3.25) -- (5.5, 2.25);
\end{scope}

\begin{pgfonlayer}{background2}
\fill [lsugold]  [rounded corners = 10 mm] (0,0) -- (3,1.5) -- (6,0) -- (6,1) -- (3,2) -- (0,1); 
\fill [lsugold] (6,0) -- (6,1) -- (3.9,2.05) -- (4,1);
\fill [lsugold] (0,0) -- (0,1) -- (2.1,2.05) -- (2,1);
\fill [lsugold] (2.23,2.28) --(3.77,2.28) -- (3.77,1.5) -- (2.23,1.5);

\fill [white, rounded corners = 7mm] (2.23,2.3) -- (3,1.6) -- (3.77,2.3);
\fill [lsugold] (2,2) -- (2.3,2.21) -- (2.2, 1.5) -- (2,1.5);
\fill [lsugold] (4,2) -- (3.7, 2.21) -- (3.8,1.5) -- (4,1.5);
\end{pgfonlayer}

\begin{pgfonlayer}{background4}
\fill [lsupurple] (.5,3.25) -- (.5,2.25) -- (3,1.25) -- (2.4,2.2);
\fill [rounded corners = 5mm, lsupurple] (0.5,3.25) -- (2.5,2.25) -- (2,2);
\fill [lsupurple] (5.5,3.25) -- (5.5,2.25) -- (3,1.25) -- (3.6,2.2);
\fill [rounded corners = 5mm, lsupurple] (5.5, 3.25) -- (3.5,2.25) -- (4,2);
\end{pgfonlayer}

\draw [thick] (0.5,2.25) -- (1.6,1.81);
\draw [thick] (5.5,2.25) -- (4.4,1.81);
\draw [thick] (0.5,2.75) -- (2.1,2.08);
\draw [thick] (5.5,2.75) -- (3.9,2.08);

\begin{pgfonlayer}{background}
\draw [black!50!white, rounded corners = 8mm, thick] (0.5, 2.25) -- (3,1.25) -- (5.5,2.25);
\draw [black!50!white, rounded corners = 7mm, thick] (2.13,2.07) -- (3,1.7)  -- (3.87,2.07);
\end{pgfonlayer}
\draw [thick, dashed, rounded corners = 2mm] (3,1.85) -- (2.8,1.6) -- (2.8,1.24);
\draw (0,0.5) node[left]{$D$};
\draw (1.5,3) node{$A$};
\draw (4.5,3) node{$A$};
\draw (3.8,.8) node{$B$};
\draw (5.3, 1.85) node{$B$};
\end{tikzpicture}$$
\caption{A saddle transitions between the all-$A$ and all-$B$ states in a neighborhood of each crossing of $D$.}
\label{figure:Saddle}
\end{figure}

If $D$ is a connected link diagram with $c(D)$ crossings whose all-$A$ state has $s_A(D)$ components and whose all-$B$ state has $s_B(D)$ components, then the genus $g_T(D)$ of the Turaev surface is given by
\begin{equation}
g_T(D) = \frac{1}{2}\left(2 + c(D) - s_A(D) - s_B(D)\right).
\end{equation}
Define the \textit{Turaev genus} $g_T(L)$ of the non-split link $L$ to be 
$$g_T(L) = \min \{g_T(D)~|~D~\text{is a diagram of}~L\}.$$

Turaev \cite{Turaev:Jones} showed that the Turaev surface of a link diagram $D$ is a sphere if and only if $D$ is the connected sum of alternating diagrams. Consequently, the Turaev genus of a link is zero if and only if it is alternating. Turaev also showed that the span of the Jones polynomial gives a lower bound for the difference between the crossing number and the Turaev genus of $L$
$$\operatorname{span} \widetilde{V}_L(t) \leq c(L) - g_T(L).$$

Thistlethwaite \cite{Thistlethwaite:Jones} proved that the Jones polynomial of an alternating link $L$ is an evaluation of the Tutte polynomial of the checkerboard graph of an alternating diagram of $L$. Dasbach, Futer, Kalfagianni, Lin, and Stoltzfus \cite{DFKLS:Graphs} extended this result by showing that the Jones polynomial of an arbitrary link is the evaluation of the Bollob\'as-Riordan-Tutte polynomial of a certain graph embedded in the Turaev surface. Dasbach and Lowrance \cite{DasLow:Approach} further extended this work by giving a Turaev surface model for Khovanov homology. Dasbach, Futer, Kalfagianni, Lin, and Stoltzfus \cite{DFKLS:Determinant} gave a formula for the determinant coming from the Turaev surface. This formula says that the determinant of a link of Turaev genus one is the difference between the number of spanning trees in two graphs dually embedded on the Turaev surface. See the recent survey \cite{CK:Survey} and chapter \cite{KK:Turaev} for more details.

Recall that a link diagram $D$ is almost alternating if one crossing change can transform $D$ into an alternating diagram, and that a link $L$ is almost alternating if it is non-alternating and has an almost alternating diagram. Almost alternating knots and links share many, but certainly not all, of the nice properties of alternating links. 

If $D$ is almost alternating, then $s_A(D) + s_B(D) = c(D)$, and thus $g_T(D)=1$. Therefore every almost alternating link is Turaev genus one. Define the \textit{dealternating number} $\dalt(D)$ of a link diagram $D$ to be the fewest number of crossings changes necessary to transform $D$ into an alternating diagram, and define the \textit{dealternating number} $\dalt(L)$ of the link $L$ by
$$\dalt(L)=\{ \dalt(D)~|~D~\text{is a diagram of}~L\}.$$
Alternating links are the class of links with dealternating number zero, and almost alternating links are the class of links with dealternating number one. 

Abe and Kishimoto \cite{AK:Dalt} proved that for any link $L$, the Turaev genus of $L$ is less than or equal to the dealternating number of $L$, i.e. $g_T(L)\leq \dalt(L)$. There are no known algorithms to compute either the Turaev genus or dealternating number of a link. Instead most computations of either invariant come from various lower bounds. Asaeda and Prztycki \cite{AP:Torsion} proved that the Khovanov homology of an almost alternating link is supported on at most three adjacent diagonals, and more generally that the Khovanov homology of a link is supported on at most $\dalt(L) + 2$ adjacent diagonals. Manturov \cite{Manturov:Minimal} and Champanerkar, Kofman, and Stoltzfus \cite{CKS:Graphs} proved a similar result using the Turaev genus of a link in place of the dealternating number, showing that the Khovanov homology of a link is supported on at most $g_T(L)+2$ adjacent diagonals. Dasbach and Lowrance \cite{DasLow:Concordance} refined these results to express the diagonal gradings of Khovanov homology where the link is supported in terms of the number of components of the all-$A$ and all-$B$ states.
\begin{theorem}[Dasbach, Lowrance]
\label{theorem:DiagonalGrading}
Let $L$ be a non-split link with diagram $D$. If $Kh^{i,j}(L)$ is nontrivial, then $2i-j \equiv s_A(D)-c_+(D)\mod 2$ and
$$s_A(D)-c_+(D) -2 \leq 2i-j \leq c_-(D) - s_B(D)+2.$$
If $D$ is a Turaev genus one diagram, then
$$s_A(D)-c_+(D) -2 \leq 2i-j \leq s_A(D)-c_+(D) +2.$$
\end{theorem}

Although it is not explicitly stated, the work of Ozsv\'ath and Szab\'o in \cite{OS:Alternating} can be used to show that the knot Floer homology of a link $L$ can be supported on at most $\dalt(L)+1$ adjacent diagonals. Lowrance \cite{Lowrance:Floer} showed that the knot Floer homology of $L$ can be supported on at most $g_T(L)+1$ adjacent diagonals. These results can now be seen as a consequence of the analogous results for reduced Khovanov homology and the spectral sequence from reduced Khovanov homology to knot Floer homology \cite{Dowlin:Spectral}.

Let $\sigma(K)$ be the signature of the knot $K$, let $s(K)$ be the Rasmussen $s$-invariant \cite{Rasmussen:Slice}, and let $\tau(K)$ be the Ozsv\'ath-Szab\'o $\tau$-invariant \cite{OS:Tau}. Abe \cite{Abe:Alternation} proved that the quantities $\frac{1}{2}|s(K)+\sigma(K)|$ and $|\tau(K)+\sigma(K)/2|$ are lower bounds for the Gordian distance between a knot $K$ and the set of alternating knots. Therefore, 
\begin{align}\label{equation:concordance1}
\begin{split}
\frac{1}{2}|s(K)+\sigma(K)| \leq & \; \dalt(K)~\text{and}\\
\left|\tau(K)+\frac{\sigma(K)}{2}\right| \leq & \;\dalt(K).
\end{split}
\end{align}
Dasbach and Lowrance \cite{DasLow:Concordance} showed that for any knot $K$, both 
\begin{align}\label{equation:concordance2}
\begin{split}
\frac{1}{2}|s(K)+\sigma(K)| \leq & \; g_T(K)~\text{and}\\
\left|\tau(K)+\frac{\sigma(K)}{2}\right| \leq & \; g_T(K).
\end{split}
\end{align}

In a different direction, one can use the Jones polynomial to obstruct being Turaev genus one and almost alternating. Kauffman \cite{Kauffman:StateModels} proved that the both the leading and trailing coefficients of the Jones polynomial of an alternating link have absolute value one. Thistlethwaite \cite{Thistlethwaite:Jones} proved that the coefficients of the Jones polynomial of a non-split alternating link alternate in sign.
Dasbach and Lowrance \cite{DasLow:TuraevJones} proved that at least one of the leading or trailing coefficients of the Jones polynomial of a Turaev genus one or almost alternating link has absolute value one. Theorem \ref{theorem:Diagonal} can be viewed as a Khovanov homology generalization of this result. Extending this work, Lowrance and Spyropoulos \cite{LS:AAJones} proved that at least one of the first two or last two coefficients of the Jones polynomial of a Turaev genus one or almost alternating link alternate in sign and showed that the Jones polynomial of an $\ell$-component Turaev genus one or almost alternating link is different than the Jones polynomial of an $\ell$-component unlink. 

There is no known example of a link $L$ with $g_T(L) < \dalt(L)$. In particular, it is an open question whether every Turaev genus one link is almost alternating. The difficulty in producing such an example is that nearly all of the lower bounds on the dealternating number of a knot or link are also lower bounds on its Turaev genus. For more about distinguishing these two invariants see \cite{Lowrance:AltDist} and \cite{Lowrance:Encyclopedia}.

There is a small but growing number of computations of Turaev genus and dealternating number in the literature. Among knots with 11 or fewer crossings, all but two ($11n_{95}$ and $11n_{118}$) are known to be either alternating or Turaev genus one (and hence alternating or almost alternating) \cite{Adams:Almost, GHY:Almost}.  The knot $11n_{95}$ has Turaev genus and dealternating number two \cite{DasLow:TuraevJones}, while the Turaev genus and dealternating numbers of $11n_{118}$ are either one or two. Abe \cite{Abe:Alternation} and Lowrance \cite{Lowrance:Twisted} proved that the almost alternating and Turaev genus one  torus knots are $T_{3,4}$, $T_{3,5}$, and their mirrors. Abe and Kishimoto \cite{AK:Dalt} and Lowrance \cite{Lowrance:Twisted} computed the Turaev genus and dealternating numbers of $3$-stranded torus links and many closed $3$-braids. Abe \cite{Abe:Adequate} proved that any adequate diagram is Turaev genus minimizing. Jin, Lowrance, Polston, and Zheng \cite{JLPZ:TuraevTorus} computed the Turaev genus of all $4$-stranded torus knots and of many $5$-stranded and $6$-stranded torus knots. 

The definitions of both alternating links and almost alternating links are intrinsically diagrammatic. A link is alternating (or almost alternating) if it has a diagram that satisfies certain properties. Fox famously asked ``What is an alternating knot?" \cite[p. 32]{Lickorish:Book} - where he was asking for a topological characterization of alternating knots. Two such characterizations, one by Greene \cite{Greene:Alternating} and one by Howie \cite{Howie:Alternating}, recently appeared in the literature. Building on the ideas of Greene and Howie, Ito \cite{Ito:AlmostAlt} and Kim \cite{Kim:Toroidally} gave topological characterizations of almost alternating knots. There is currently no topological characterization of Turaev genus one knots or links.

\subsection{$A$/$B$-Turaev genus one and $A$/$B$-almost alternating}

In this section, we prove that every almost alternating link is $A$-almost alternating or $B$-almost alternating, and that every Turaev genus one link is either $A$-Turaev genus one or $B$-Turaev genus one. Throughout this section we use the notation depicted in Figure \ref{figure:aadiagram} for an almost alternating link diagram. The alternating tangle is $R$ and the four regions with the dealternator in their boundary are $u_1$, $u_2$, $v_1$, and $v_2$.

Alternating and almost alternating links have almost alternating diagrams. Let $D$ be an almost alternating diagram.   If the regions $u_1$ and $u_2$ in $D$ are the same region or if the regions $v_1$ and $v_2$ in $D$ are the same region, then $L$ is alternating. Similarly, if there is a crossing in $R$ contained in the boundary of regions $u_1$ and $u_2$ or if there is a crossing in $R$ contained in the boundary of $v_1$ and $v_2$, then $L$ is again alternating. If $L$ is almost alternating (and hence non-alternating), these configurations cannot exist in $D$. See Figure \ref{figure:Reducible}. Unless otherwise stated, we assume that all almost alternating diagrams in the paper have distinct $u_1$ and $u_2$, distinct $v_1$ and $v_2$, no crossing in $R$ in the boundary of $u_1$ and $u_2$, and no crossing in $R$ in the boundary of $v_1$ and $v_2$.
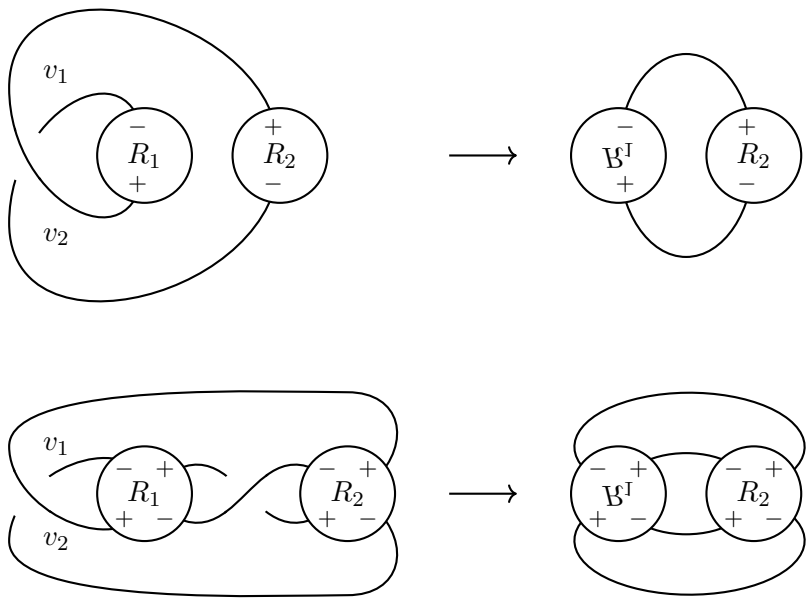
\begin{figure}[h]
$$\begin{tikzpicture}[thick, scale = .9]


\begin{knot}[
	consider self intersections,
 	clip width = 5,
 	ignore endpoint intersections = true,
	end tolerance = 2pt
	]
	\flipcrossings{1}
	\strand (2,2) to [out = 90, in = 90, looseness =2]
	(0,1) to [out = 270, in = 270, looseness = 1.3]
	(4,2) to [out = 90, in = 90, looseness = 1.3]
	(0,3) to [out = 270, in = 270, looseness = 2]
	(2,2);
\end{knot}

\draw (.7,3.2) node{$v_1$};
\draw (.7,.8) node{$v_2$};

\fill[white] (2,2) circle (.7cm);
\draw[thick] (2,2) circle (.7cm);
\draw (2,2) node {$R_1$};
\draw (1.9,2.7) node[below]{\footnotesize{$-$}};
\draw (1.9,1.3) node[above]{\footnotesize{$+$}};

\fill[white] (4,2) circle (.7cm);
\draw[thick] (4,2) circle (.7cm);
\draw (4,2) node{$R_2$};
\draw (3.9,2.7) node[below]{\footnotesize{$+$}};
\draw (3.9,1.3) node[above]{\footnotesize{$-$}};


\draw [->] (6.5,2) -- (7.5,2);

\begin{scope}[xshift = 7cm]

\draw (3,2) ellipse (1cm and 1.5cm);

\fill[white] (2,2) circle (.7cm);
\draw[thick] (2,2) circle (.7cm);
\draw (2,2)  node{\raisebox{\depth}{\scalebox{1}[-1]{$R_1$}}};
\draw (2.1,2.7) node[below]{\footnotesize{$-$}};
\draw (2.1,1.3) node[above]{\footnotesize{$+$}};

\fill[white] (4,2) circle (.7cm);
\draw[thick] (4,2) circle (.7cm);
\draw (4,2) node{$R_2$};
\draw (3.9,2.7) node[below]{\footnotesize{$+$}};
\draw (3.9,1.3) node[above]{\footnotesize{$-$}};

\end{scope}

\begin{scope}[yshift = -5cm]

\draw [->] (6.5,2) -- (7.5,2);
\draw (.7,2.7) node{$v_1$};
\draw (.7,1.3) node{$v_2$};
\begin{knot}[
	consider self intersections,
 	clip width = 5,
 	ignore endpoint intersections = true,
	end tolerance = 2pt
	]
	\flipcrossings{1,2}
	\strand (5,2) to [out = 225, in = 45, looseness = 1.8]
	(2,2) to [out = 315, in = 135, looseness = 1.8]
	(5,2);
	
	\strand (2,2.3) to [out = 135, in = 90, looseness = 1]
	(0,1.3) to [out = 270, in = 180, looseness = .5]
	(5,.5) to [out = 0, in = 315, looseness = 1.5]
	(5.4,1.8) to [out = 90, in = 270]
	(5.4,2.2) to [out = 45, in = 0, looseness = 1.5]
	(5,3.5) to [out = 180, in = 90, looseness = .5]
	(0,2.7) to [out = 270, in = 225, looseness = 1]
	(2,1.7);
\end{knot}

\fill[white] (2,2) circle (.7cm);
\draw[thick] (2,2) circle (.7cm);
\draw (2,2) node {$R_1$};
\draw (1.7,2.65) node[below]{\footnotesize{$-$}};
\draw (2.3,2.65) node[below]{\footnotesize{$+$}};
\draw (2.3,1.35) node[above]{\footnotesize{$-$}};
\draw (1.7, 1.35) node[above]{\footnotesize{$+$}};

\fill[white] (5,2) circle (.7cm);
\draw[thick] (5,2) circle (.7cm);
\draw (5,2) node{$R_2$};
\draw (4.7,2.65) node[below]{\footnotesize{$-$}};
\draw (4.7,1.35) node[above]{\footnotesize{$+$}};
\draw (5.3,2.65) node[below]{\footnotesize{$+$}};
\draw (5.3,1.35) node[above]{\footnotesize{$-$}};


\begin{scope}[xshift = 7cm]

\draw (3,2) ellipse (1cm and .6cm);
\draw (3.9,2) arc (-60:240: 1.7cm and .8cm);
\draw (3.9,2) arc (60:-240: 1.7cm and .8cm);

\fill[white] (2,2) circle (.7cm);
\draw[thick] (2,2) circle (.7cm);
\draw (2,2)  node{\raisebox{\depth}{\scalebox{1}[-1]{$R_1$}}};
\draw (1.7,2.65) node[below]{\footnotesize{$-$}};
\draw (1.7,1.35) node[above]{\footnotesize{$+$}};
\draw (2.3,2.65) node[below]{\footnotesize{$+$}};
\draw (2.3,1.35) node[above]{\footnotesize{$-$}};

\fill[white] (4,2) circle (.7cm);
\draw[thick] (4,2) circle (.7cm);
\draw (4,2) node{$R_2$};
\draw (3.7,2.65) node[below]{\footnotesize{$-$}};
\draw (3.7,1.35) node[above]{\footnotesize{$+$}};
\draw (4.3,2.65) node[below]{\footnotesize{$+$}};
\draw (4.3,1.35) node[above]{\footnotesize{$-$}};

\end{scope}

\end{scope}
\end{tikzpicture}$$
\caption{Two almost alternating diagrams of alternating links. On the top, the regions $v_1$ and $v_2$ are the same, and removing the nugatory crossing yields an alternating diagram. On the bottom, the regions $v_1$ and $v_2$ are incident to the same crossing in the alternating tangle $R$. A flype followed by a Reidemeister 2 move yields an alternating diagram.}
\label{figure:Reducible}
\end{figure}

We now set about proving that every almost alternating link is $A$-almost alternating or $B$-almost alternating. In order to do so, we introduce the checkerboard graph of the link diagram. Let $D$ be an almost alternating diagram as in Figure \ref{figure:aadiagram}. We associate two checkerboard graphs $G$ and $\overline{G}$ to $D$ as follows. Color the complementary regions of $D$ in a checkerboard fashion so that the shading at each crossing looks like $\tikz[baseline=.6ex, scale = .4, thick]{
\fill[black!20!white] (0,0) -- (.5,.5) -- (0,1);
\fill[black!20!white] (1,1) -- (.5,.5) -- (1,0);
\draw (0,0) -- (1,1);
\draw (0,1) -- (.3,.7);
\draw (1,0) -- (.7,.3);
}$ or $\tikz[baseline=.6ex, scale = .4, thick]{
\fill[black!20!white] (0,0) -- (.5,.5) -- (1,0);
\fill[black!20!white] (1,1) -- (.5,.5) -- (0,1);
\draw (0,0) -- (1,1);
\draw (0,1) -- (.3,.7);
\draw (1,0) -- (.7,.3);
}$. The shaded regions of $D$ are in one-to-one correspondence with the vertices of $G$, and the unshaded regions of $D$ are in one-to-one correspondence with the vertices of $\overline{G}$ (or vice versa). The crossings of $D$ are in one-to-one correspondence with the edges of $G$ and with the edges of $\overline{G}$. Two vertices $u$ and $v$ are incident to an edge $e$ in $G$ (or in $\overline{G}$) if the regions associated to vertices $u$ and $v$ meet at the crossing associated to $e$. The graphs $G$ and $\overline{G}$ are planar duals of one another.

By a slight abuse of notation, the vertices associated to the regions meeting at the dealternator are labeled $u_1$, $u_2$, $v_1$, and $v_2$ as in Figure \ref{figure:aadiagram}. Define $G$ to be the checkerboard graph of $D$ containing the vertices $u_1$ and $u_2$, and define $\overline{G}$ be the checkerboard graph of $D$ containing the vertices $v_1$ and $v_2$. Let $G'$ and $\overline{G}'$ be the simplifications of $G$ and $\overline{G}$ respectively, that is $G'$ is $G$ with each sets of multiple edges identified into a single edge, and likewise for $\overline{G}'$. Define $\operatorname{adj}(u_1,u_2)$ to be the number of paths of length two between $u_1$ and $u_2$ in $G'$, and $\operatorname{adj}(v_1,v_2)$ to be the number of paths of length two between $v_1$ and $v_2$ in $\overline{G}'$. With this definition, an almost alternating diagram of an almost alternating link is $A$-almost alternating if and only if $\operatorname{adj}(u_1,u_2) = 0$ and is $B$-almost alternating if and only if $\operatorname{adj}(v_1,v_2) = 0$. In the next few lemmas, we show that if $\operatorname{adj}(u_1,u_2)$ and $\operatorname{adj}(v_1,v_2)$ satisfy certain constraints, then the link is alternating.

\begin{lemma}
\label{lemma:alt}
If $D$ is an almost alternating diagram of a link $L$ such that $\operatorname{adj}(u_1,u_2)=2$ and $\operatorname{adj}(v_1,v_2)=1$ (or if $\operatorname{adj}(u_1,u_2)=1$ and $\operatorname{adj}(v_1,v_2)=2$), then $L$ is alternating.
\end{lemma}
\begin{proof}
Let $D$ be an almost alternating diagram of $L$ satisfying conditions (1) and (2) in Definition \ref{definition:ABalmostalternating} and also such that $\operatorname{adj}(u_1,u_2)=2$ and $\operatorname{adj}(v_1,v_2)=1$. Let $G$ be the checkerboard graph of $D$ containing $u_1$ and $u_2$. Suppose that the two paths of length two between $u_1$ and $u_2$ in $G$ have edge sets $\{e_1,e_2\}$ and $\{e_3,e_4\}$ respectively, where $e_1$ and $e_3$ are incident to $u_1$ and $e_2$ and $e_4$ are incident to $u_2$. 

The path of length two in $\overline{G}$ between $v_1$ and $v_2$ contains an edge dual to either $e_1$ or $e_2$ and an edge dual to either $e_3$ or $e_4$. If the path contains edges dual to $e_1$ and $e_3$ (or dual to $e_2$ and $e_4$), then up to symmetry, $D$ has a diagram as in the leftmost picture of Figure \ref{figure:altiso1}. If the path contains edges dual to $e_1$ and $e_4$ (or dual to $e_2$ and $e_3$), then up to symmetry, $D$ has diagram as in the leftmost picture of Figure \ref{figure:altiso2}. In both figures, the tangles $R_i$ are alternating tangles. Each figure shows an isotopy between $D$ and an alternating diagram.
\begin{figure}[h]
$$\begin{tikzpicture}[thick]

\draw (2.3,2) -- (2.9,1.15);
\draw (3.1,.85) -- (3.7,0);
\draw (2.3,0) -- (3.7,2);

\draw (2.3,0) -- (2.9,-.85);
\draw (3.1,-1.15) -- (3.7,-2);
\draw (2.3,-2) -- (3.7,0);

\begin{scope}[rounded corners=3mm]
\draw (3,1.6) -- (1.5,1.6) -- (1.5,-.2);
\draw(3,.4) -- (1.7,.4);
\draw(3,-.4) -- (1,-.4) -- (.2,.4) -- (.2,2.4) -- (3,2.4);
\draw(1.3,.4) -- (1,.4) -- (.75,.15);
\draw(1.5,-.6) -- (1.5,-1.6) -- (3,-1.6);
\draw (.45,-.15) -- (.2,-.4) -- (.2,-2.4) -- (3,-2.4);

\end{scope}

\fill[white] (3,0) circle (.7);
\draw (3,0) circle (.7);
\draw (3,0) node{$R_2$};
\draw (2.8,.5) node{\footnotesize{$-$}};
\draw (3.2,.5) node{\footnotesize{$+$}};
\draw (2.8,-.5) node{\footnotesize{$+$}};
\draw (3.2,-.5) node{\footnotesize{$-$}};
\draw (2.6,.3) node{\footnotesize{$+$}};
\draw (2.6,-.3) node{\footnotesize{$-$}};

\fill[white](3,2) circle (.7);
\draw (3,2) circle (.7);
\draw (3,2) node{$R_1$};
\begin{scope}[yshift = 2cm]
\draw (2.8,-.5) node{\footnotesize{$+$}};
\draw (3.2,-.5) node{\footnotesize{$-$}};
\draw (2.6,.3) node{\footnotesize{$+$}};
\draw (2.6,-.3) node{\footnotesize{$-$}};
\end{scope}

\fill[white](3,-2) circle (.7);
\draw (3,-2) circle (.7);
\draw (3,-2) node{$R_3$};
\begin{scope}[yshift = -2cm]
\draw (2.8,.5) node{\footnotesize{$-$}};
\draw (3.2,.5) node{\footnotesize{$+$}};
\draw (2.6,.3) node{\footnotesize{$+$}};
\draw (2.6,-.3) node{\footnotesize{$-$}};
\end{scope}

\begin{scope}[xshift = 5cm]
\draw (2.8,2) -- (2.8,-2);
\draw (3.2,2) -- (3.2,-2);

\begin{scope}[rounded corners=3mm]
\draw (3,1.7) -- (2.2,1.7) -- (1.6,2.3) -- (.2,2.3) --(.2,.4) -- (1,-.4) -- (3,-.4);
\draw[rounded corners = 2mm] (3,2.3) -- (2.2,2.3) -- (2.05,2.15);
\draw (1.75,1.85) -- (1.5,1.6) -- (1.5,-.2);
\draw (3,.4) -- (1.7,.4);

\draw(1.3,.4) -- (1,.4) -- (.75,.15);
\draw[rounded corners = 2mm] (3,-1.7) -- (2.2,-1.7) -- (2.05,-1.85);
\draw[rounded corners = 2mm] (.45,-.15) -- (.2,-.4) -- (.2,-2.3) -- (1.6,-2.3) -- (1.75,-2.15);
\draw (3,-2.3) -- (2.2,-2.3) -- (1.5,-1.6) -- (1.5,-.6);

\end{scope}

\fill[white] (3,0) circle (.7);
\draw (3,0) circle (.7);
\draw (3,0) node{$R_2$};
\draw (2.8,.5) node{\footnotesize{$-$}};
\draw (3.2,.5) node{\footnotesize{$+$}};
\draw (2.8,-.5) node{\footnotesize{$+$}};
\draw (3.2,-.5) node{\footnotesize{$-$}};
\draw (2.6,.3) node{\footnotesize{$+$}};
\draw (2.6,-.3) node{\footnotesize{$-$}};

\fill[white](3,2) circle (.7);
\draw (3,2) circle (.7);
\draw (3,2) node{\reflectbox{$R_1$}};
\begin{scope}[yshift = 2cm]
\draw (2.8,-.5) node{\footnotesize{$+$}};
\draw (3.2,-.5) node{\footnotesize{$-$}};
\draw (2.6,.3) node{\footnotesize{$+$}};
\draw (2.6,-.3) node{\footnotesize{$-$}};
\end{scope}

\fill[white](3,-2) circle (.7);
\draw (3,-2) circle (.7);
\draw (3,-2) node{\reflectbox{$R_3$}};
\begin{scope}[yshift = -2cm]
\draw (2.8,.5) node{\footnotesize{$-$}};
\draw (3.2,.5) node{\footnotesize{$+$}};
\draw (2.6,.3) node{\footnotesize{$+$}};
\draw (2.6,-.3) node{\footnotesize{$-$}};
\end{scope}

\end{scope}

\begin{scope}[xshift = 9.5cm]
\draw (2.8,2) -- (2.8,-2);
\draw (3.2,2) -- (3.2,-2);

\begin{scope}[rounded corners=3mm]

\draw (3,-.3) -- (2.2,-.3) -- (1.6,.3) -- (1.6,1.7) -- (3,1.7);
\draw[rounded corners = 2mm] (3,.3) -- (2.2,.3) -- (2.05,.15);
\draw[rounded corners = 2mm] (1.75,-.15) -- (1.6,-.3) -- (1.6,-1.7) -- (3,-1.7);
\draw (3,2.3) -- (1,2.3) -- (1,-2.3) -- (3,-2.3);

\end{scope}

\fill[white] (3,0) circle (.7);
\draw (3,0) circle (.7);
\draw (3,0) node{$R_2$};
\draw (2.8,.5) node{\footnotesize{$-$}};
\draw (3.2,.5) node{\footnotesize{$+$}};
\draw (2.8,-.5) node{\footnotesize{$+$}};
\draw (3.2,-.5) node{\footnotesize{$-$}};
\draw (2.6,.3) node{\footnotesize{$+$}};
\draw (2.6,-.3) node{\footnotesize{$-$}};

\fill[white](3,2) circle (.7);
\draw (3,2) circle (.7);
\draw (3,2) node{\reflectbox{$R_1$}};
\begin{scope}[yshift = 2cm]
\draw (2.8,-.5) node{\footnotesize{$+$}};
\draw (3.2,-.5) node{\footnotesize{$-$}};
\draw (2.6,.3) node{\footnotesize{$+$}};
\draw (2.6,-.3) node{\footnotesize{$-$}};
\end{scope}

\fill[white](3,-2) circle (.7);
\draw (3,-2) circle (.7);
\draw (3,-2) node{\reflectbox{$R_3$}};
\begin{scope}[yshift = -2cm]
\draw (2.8,.5) node{\footnotesize{$-$}};
\draw (3.2,.5) node{\footnotesize{$+$}};
\draw (2.6,.3) node{\footnotesize{$+$}};
\draw (2.6,-.3) node{\footnotesize{$-$}};
\end{scope}

\end{scope}

\end{tikzpicture}$$
\caption{An isotopy between $D$ and an alternating diagram when $\operatorname{adj}(u_1,u_2)=2$ and $\operatorname{adj}(v_1,v_2)=1$. To obtain the second diagram from the first, perform two flypes, and to obtain the third diagram from the second perform isotopies of the strands on the left of the tangles.}
\label{figure:altiso1}
\end{figure}
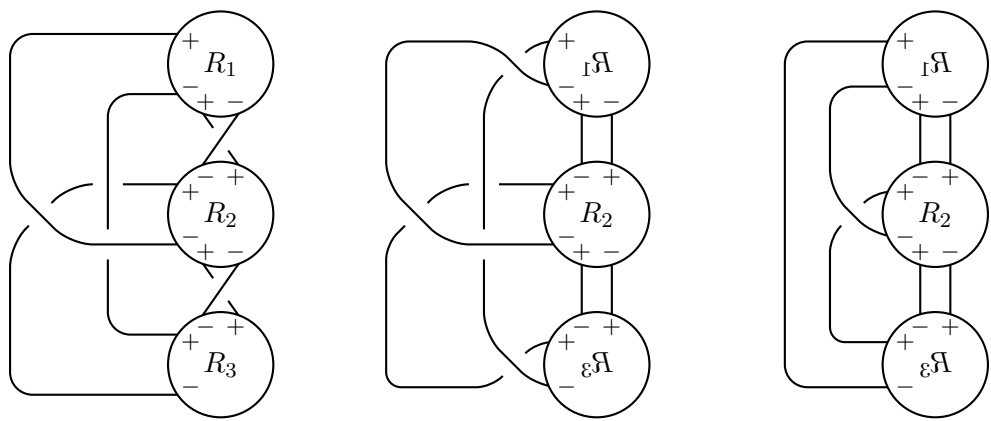

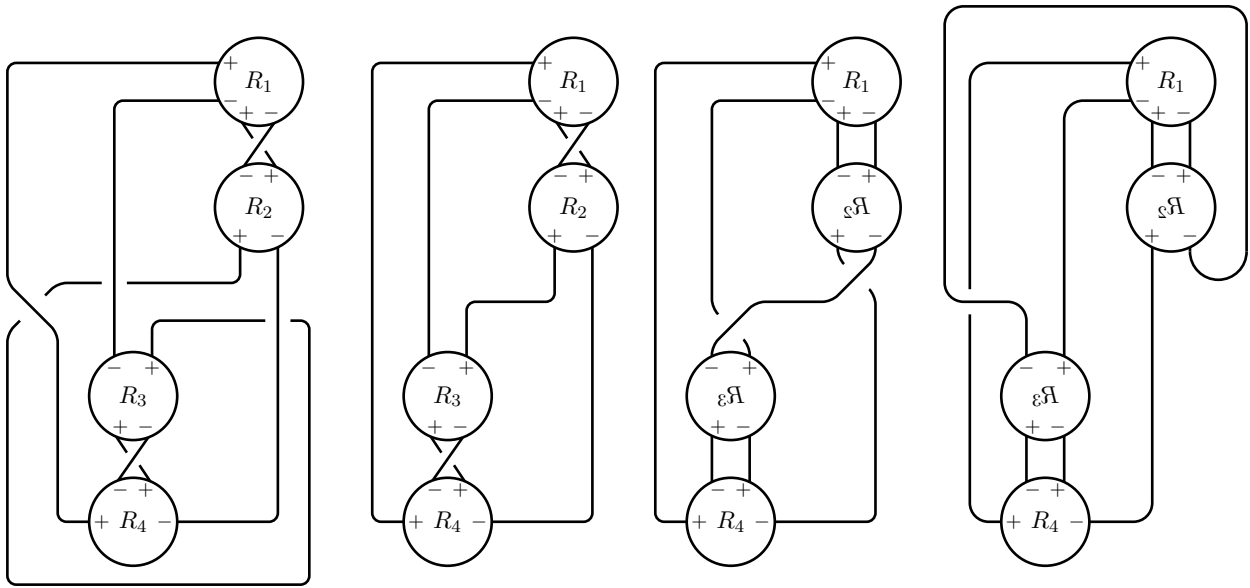
\begin{figure}[h]
\begin{adjustbox}{max width=\textwidth}
$$\begin{tikzpicture}[very thick, scale = 1]

\draw (.7,0) -- (.1,.9);
\draw (-.1,1.1) -- (-.7,2);
\draw (-.7,0) -- (.7,2);

\draw (2.7,5) -- (2.1, 5.9);
\draw (1.9,6.1) -- (1.3,7);
\draw (1.3,5) -- (2.7,7);

\begin{scope}[rounded corners]

\draw(2,6.7) -- (-.3,6.7) -- (-.3,2);
\draw(0,0) -- (2.3,0) -- (2.3,5);
\draw(1.7,5) -- (1.7,3.8) -- (-.1,3.8);
\draw(.3,2) -- (.3,3.2) -- (2.1,3.2);
\draw(2.5,3.2) -- (2.8,3.2) -- (2.8, -1) -- (-2,-1) -- (-2,3) -- (-1.8,3.2);
\draw (0,0) -- (-1.2,0) -- (-1.2,3) -- (-2,3.8) -- (-2,7.3) -- (2,7.3);
\draw (-.5,3.8) -- (-1.2,3.8) -- (-1.4,3.6);

\end{scope}

\fill[white](0,0) circle (.7cm);
\draw (0,0) circle (.7cm);
\draw (0,0) node{$R_4$};
\draw (-.5,0) node{\footnotesize{$+$}};
\draw (.5,0) node{\footnotesize{$-$}};
\draw (-.2,.5) node{\footnotesize{$-$}};
\draw (.2,.5) node{\footnotesize{$+$}};

\fill[white](0,2) circle (.7cm);
\draw (0,2) circle (.7cm);
\draw (0,2) node{$R_3$};
\draw (-.2,1.5) node{\footnotesize{$+$}};
\draw (.2,1.5) node{\footnotesize{$-$}};
\draw (-.3,2.45) node{\footnotesize{$-$}};
\draw (.3,2.45) node{\footnotesize{$+$}};

\fill[white](2,5) circle (.7cm);
\draw(2,5) circle (.7cm);
\draw (2,5) node{$R_2$};
\draw (1.7,4.55) node{\footnotesize{$+$}};
\draw (2.3,4.55) node{\footnotesize{$-$}};
\draw (1.8,5.5) node{\footnotesize{$-$}};
\draw (2.2,5.5) node{\footnotesize{$+$}};

\fill[white](2,7) circle (.7cm);
\draw (2,7) circle (.7cm);
\draw (2,7) node{$R_1$};
\draw (1.8,6.5) node{\footnotesize{$+$}};
\draw (2.2,6.5) node{\footnotesize{$-$}};
\draw (1.55,6.7) node{\footnotesize{$-$}};
\draw (1.55,7.3) node{\footnotesize{$+$}};

\begin{scope}[xshift = 5cm]
\draw (.7,0) -- (.1,.9);
\draw (-.1,1.1) -- (-.7,2);
\draw (-.7,0) -- (.7,2);

\draw (2.7,5) -- (2.1, 5.9);
\draw (1.9,6.1) -- (1.3,7);
\draw (1.3,5) -- (2.7,7);

\begin{scope}[rounded corners]

\draw (2,6.7) -- (-.3,6.7) -- (-.3,2);
\draw (0,0) -- (2.3,0) -- (2.3,5);
\draw (1.7,5) -- (1.7,3.5) --(.3,3.5) -- (.3,2); 

\draw (0,0) -- (-1.2,0) -- (-1.2,7.3) -- (2,7.3);

\end{scope}

\fill[white](0,0) circle (.7cm);
\draw (0,0) circle (.7cm);
\draw (0,0) node{$R_4$};
\draw (-.5,0) node{\footnotesize{$+$}};
\draw (.5,0) node{\footnotesize{$-$}};
\draw (-.2,.5) node{\footnotesize{$-$}};
\draw (.2,.5) node{\footnotesize{$+$}};

\fill[white](0,2) circle (.7cm);
\draw (0,2) circle (.7cm);
\draw (0,2) node{$R_3$};
\draw (-.2,1.5) node{\footnotesize{$+$}};
\draw (.2,1.5) node{\footnotesize{$-$}};
\draw (-.3,2.45) node{\footnotesize{$-$}};
\draw (.3,2.45) node{\footnotesize{$+$}};

\fill[white](2,5) circle (.7cm);
\draw(2,5) circle (.7cm);
\draw (2,5) node{$R_2$};
\draw (1.7,4.55) node{\footnotesize{$+$}};
\draw (2.3,4.55) node{\footnotesize{$-$}};
\draw (1.8,5.5) node{\footnotesize{$-$}};
\draw (2.2,5.5) node{\footnotesize{$+$}};

\fill[white](2,7) circle (.7cm);
\draw (2,7) circle (.7cm);
\draw (2,7) node{$R_1$};
\draw (1.8,6.5) node{\footnotesize{$+$}};
\draw (2.2,6.5) node{\footnotesize{$-$}};
\draw (1.55,6.7) node{\footnotesize{$-$}};
\draw (1.55,7.3) node{\footnotesize{$+$}};
\end{scope}
\begin{scope}[xshift = 9.5cm]
\draw (-.3,0) -- (-.3,2);
\draw (.3,0) -- (.3,2);
\draw (2.3,5) -- (2.3,7);
\draw (1.7,5) -- (1.7,7);

\begin{scope}[rounded corners]

\draw (2.3,5) -- (2.3,4.2) -- (1.6,3.5) --(.4,3.5)--(-.3,2.8) -- (-.3,2);
\draw (1.7,5) -- (1.7,4.2) -- (1.8,4.1);
\draw (2.2,3.7) -- (2.3,3.6) -- (2.3,0) -- (0,0);
\draw (.3,2) -- (.3,2.8) -- (.2,2.9);
\draw (-.2,3.3) -- (-.3,3.4) -- (-.3,6.7) -- (2,6.7);

\draw (0,0) -- (-1.2,0) -- (-1.2,7.3) -- (2,7.3);

\end{scope}

\fill[white](0,0) circle (.7cm);
\draw (0,0) circle (.7cm);
\draw (0,0) node{$R_4$};
\draw (-.5,0) node{\footnotesize{$+$}};
\draw (.5,0) node{\footnotesize{$-$}};
\draw (-.2,.5) node{\footnotesize{$-$}};
\draw (.2,.5) node{\footnotesize{$+$}};

\fill[white](0,2) circle (.7cm);
\draw (0,2) circle (.7cm);
\draw (0,2) node{\reflectbox{$R_3$}};
\draw (-.2,1.5) node{\footnotesize{$+$}};
\draw (.2,1.5) node{\footnotesize{$-$}};
\draw (-.3,2.45) node{\footnotesize{$-$}};
\draw (.3,2.45) node{\footnotesize{$+$}};

\fill[white](2,5) circle (.7cm);
\draw(2,5) circle (.7cm);
\draw (2,5) node{\reflectbox{$R_2$}};
\draw (1.7,4.55) node{\footnotesize{$+$}};
\draw (2.3,4.55) node{\footnotesize{$-$}};
\draw (1.8,5.5) node{\footnotesize{$-$}};
\draw (2.2,5.5) node{\footnotesize{$+$}};

\fill[white](2,7) circle (.7cm);
\draw (2,7) circle (.7cm);
\draw (2,7) node{$R_1$};
\draw (1.8,6.5) node{\footnotesize{$+$}};
\draw (2.2,6.5) node{\footnotesize{$-$}};
\draw (1.55,6.7) node{\footnotesize{$-$}};
\draw (1.55,7.3) node{\footnotesize{$+$}};
\end{scope}
\begin{scope}[xshift = 14.5cm]
\draw (-.3,0) -- (-.3,2);
\draw (.3,0) -- (.3,2);
\draw (2.3,5) -- (2.3,7);
\draw (1.7,5) -- (1.7,7);

\begin{scope}[rounded corners = 3mm]

\draw (.3,2) -- (.3,6.7) -- (2,6.7);
\draw(1.7,5) -- (1.7,0) -- (0,0);
\draw (-.3,2) -- (-.3,3.5) -- (-1.6,3.5) -- (-1.6,8.2) -- (3.2,8.2) -- (3.2,4.3);
\draw(2.3,5) -- (2.3,4.3);
\draw (2.3,4.3) arc (180:360:.45);

\draw (0,0) -- (-1.2,0) --(-1.2,3.3);
\draw (-1.2,3.7) -- (-1.2,7.3) -- (2,7.3);

\end{scope}

\fill[white](0,0) circle (.7cm);
\draw (0,0) circle (.7cm);
\draw (0,0) node{$R_4$};
\draw (-.5,0) node{\footnotesize{$+$}};
\draw (.5,0) node{\footnotesize{$-$}};
\draw (-.2,.5) node{\footnotesize{$-$}};
\draw (.2,.5) node{\footnotesize{$+$}};

\fill[white](0,2) circle (.7cm);
\draw (0,2) circle (.7cm);
\draw (0,2) node{\reflectbox{$R_3$}};
\draw (-.2,1.5) node{\footnotesize{$+$}};
\draw (.2,1.5) node{\footnotesize{$-$}};
\draw (-.3,2.45) node{\footnotesize{$-$}};
\draw (.3,2.45) node{\footnotesize{$+$}};

\fill[white](2,5) circle (.7cm);
\draw(2,5) circle (.7cm);
\draw (2,5) node{\reflectbox{$R_2$}};
\draw (1.7,4.55) node{\footnotesize{$+$}};
\draw (2.3,4.55) node{\footnotesize{$-$}};
\draw (1.8,5.5) node{\footnotesize{$-$}};
\draw (2.2,5.5) node{\footnotesize{$+$}};

\fill[white](2,7) circle (.7cm);
\draw (2,7) circle (.7cm);
\draw (2,7) node{$R_1$};
\draw (1.8,6.5) node{\footnotesize{$+$}};
\draw (2.2,6.5) node{\footnotesize{$-$}};
\draw (1.55,6.7) node{\footnotesize{$-$}};
\draw (1.55,7.3) node{\footnotesize{$+$}};
\end{scope}

\end{tikzpicture}$$
\end{adjustbox}
\caption{An isotopy between $D$ and an alternating diagram when $\operatorname{adj}(u_1,u_2)=2$ and $\operatorname{adj}(v_1,v_2)=1$. In step 1, the strand between $R_2$ and $R_3$ is pulled beneath the diagram. In step 2, two flypes are performed, and finally in step 3, the strand between $R_2$ and $R_3$ is pulled above the rest of the diagram.}
\label{figure:altiso2}
\end{figure}
\end{proof}

\begin{lemma}
\label{lemma:split}
If $D$ is an almost alternating diagram of a link $L$ such that $\operatorname{adj}(u_1,u_2)=\operatorname{adj}(v_1,v_2)=2$, then $L$ is the disjoint union of an alternating link and an unknot.
\end{lemma}
\begin{proof}
Let $D$ be an almost alternating diagram of $L$ satisfying conditions (1) and (2) in Definition \ref{definition:ABalmostalternating} and also such that $\operatorname{adj}(u_1,u_2)=\operatorname{adj}(v_1,v_2)=2$. The checkerboard graph $G$ has two paths between $u_1$ and $u_2$, and the checkerboard graph $\overline{G}$ has two paths between $v_1$ and $v_2$. Moreover, the edges in the paths between $u_1$ and $u_2$ are dual to the edges in the paths between $v_1$ and $v_2$. 

Hence the diagram $D$ is the leftmost picture in Figure \ref{figure:split} where $R$ is an alternating tangle. Figure \ref{figure:split} shows an isotopy between $D$ and the disjoint union of an unknot and an alternating diagram. 
\begin{figure}[h]
$$\begin{tikzpicture}[thick]

\begin{scope}[rounded corners = 3mm]
\draw(3,2.3) -- (2,2.3);
\draw[rounded corners = 2mm] (1.6,2.3) -- (1.4,2.3) -- (1.3,2.2);
\draw (3,1.7) -- (1.4,1.7) -- (.8,2.3) -- (.8,3.7) -- (4.7,3.7) -- (4.7,2.3) -- (3,2.3); 
\draw (1.8,1.9) -- (1.8,3.2) -- (4.2,3.2) -- (4.2,2.5);
\draw(1.8,1.5) -- (1.8,.8) -- (4.2,.8) -- (4.2,2.1);
\draw[rounded corners = 2mm] (.9,1.8) -- (.8,1.7) -- (.8,.3) -- (4.7,.3) -- (4.7,1.7) -- (4.4,1.7);
\draw (3,1.7) -- (4,1.7);

\end{scope}

\fill[white] (3,2) circle (.7);
\draw (3,2) circle (.7);
\draw (3,2) node{$R$};
\draw (2.55,2.3) node{\footnotesize{$+$}};
\draw (2.55,1.7) node{\footnotesize{$-$}};
\draw (3.45,2.3) node{\footnotesize{$-$}};
\draw (3.45,1.7) node{\footnotesize{$+$}};

\begin{scope}[xshift = 5cm]

\begin{scope}[rounded corners = 3mm]
\draw (3,2.3) -- (2,2.3);
\draw(1.65,2.3) -- (1.45,2.3);
\draw (1.1,2.3) -- (.8,2.3) -- (.8,.3) -- (4.7,.3) -- (4.7,.6);
\draw(1.3,2.3) -- (1.3,.8) -- (5.2,.8) -- (5.2,3.7) -- (1.3,3.7) -- (1.3,2.3);
\draw (4.7,1) -- (4.7,1.7) -- (3,1.7);
\draw (3,1.7) -- (1.8,1.7) -- (1.8,3.2) -- (4.2,3.2) -- (4.2,2.3) -- (3,2.3);

\end{scope}

\fill[white] (3,2) circle (.7);
\draw (3,2) circle (.7);
\draw (3,2) node{$R$};
\draw (2.55,2.3) node{\footnotesize{$+$}};
\draw (2.55,1.7) node{\footnotesize{$-$}};
\draw (3.45,2.3) node{\footnotesize{$-$}};
\draw (3.45,1.7) node{\footnotesize{$+$}};

\end{scope}

\begin{scope}[xshift = 10.5cm]

\begin{scope}[rounded corners = 3mm]

\draw[rounded corners = 2mm] (3,2.3) -- (2.2,2.3) -- (2.1,2.2);
\draw[rounded corners = 2mm] (1.7,1.8) -- (1.6,1.7) -- (1.6,.8) -- (4.2,.8) -- (4.2,1.7) -- (3,1.7);
\draw (3,1.7) -- (2.2,1.7) -- (1.6,2.3) -- (1.6,3.2) -- (4.2,3.2) -- (4.2,2.3) -- (3,2.3);
\draw (.9,2) -- (.9,.3) -- (4.7,.3) -- (4.7,3.7) -- (.9,3.7) -- (.9,2);

\end{scope}

\fill[white] (3,2) circle (.7);
\draw (3,2) circle (.7);
\draw (3,2) node{$R$};
\draw (2.55,2.3) node{\footnotesize{$+$}};
\draw (2.55,1.7) node{\footnotesize{$-$}};
\draw (3.45,2.3) node{\footnotesize{$-$}};
\draw (3.45,1.7) node{\footnotesize{$+$}};

\end{scope}

\end{tikzpicture}$$
\caption{On the left is the almost alternating diagram $D$ when $\operatorname{adj}(u_1,u_2)=\operatorname{adj}(v_1,v_2)=2$. The isotopy described by the pictures shows that $L$ is the disjoint union of an unknot and an alternating link.}
\label{figure:split}
\end{figure}
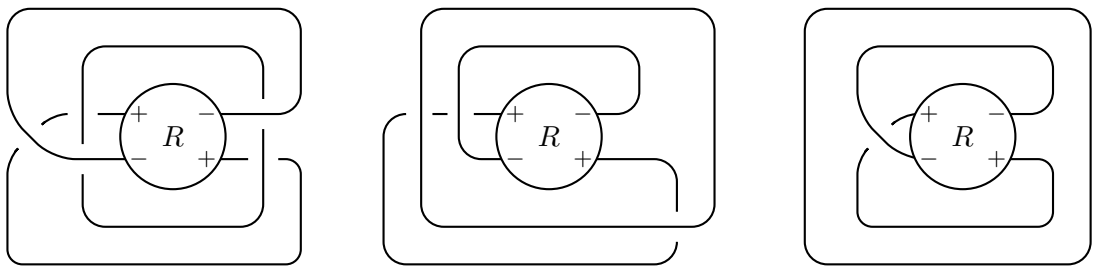
\end{proof}

The previous two lemmas combine to give the following theorem.
\begin{theorem}
\label{theorem:ABaa}
Every almost alternating link is $A$-almost alternating or $B$-almost alternating. 
\end{theorem}
\begin{proof}
Since $G$ and $\overline{G}$ are dual graphs, if $\operatorname{adj}(u_1,u_2) > 2$, then $\operatorname{adj}(v_1,v_2)=0$; if $\operatorname{adj}(v_1,v_2)>2$, then $\operatorname{adj}(u_1,u_2)=0$. If $\operatorname{adj}(u_1,u_2)=\operatorname{adj}(v_1,v_2)=1$, then the isotopy in Figure \ref{figure:AAReduce} yields an almost alternating diagram of the link with two fewer crossings (as also shown in \cite[Theorem 1.3]{DasLow:TuraevJones} and \cite[Theorem 3.3]{LS:AAJones}). Hence the only possibilities remaining where neither $\operatorname{adj}(u_1,u_2)$ or $\operatorname{adj}(v_1,v_2)$ are zero are
\begin{enumerate}
\item $\operatorname{adj}(u_1,u_2)=2$ and $\operatorname{adj}(v_1,v_2)=1$,
\item $\operatorname{adj}(u_1,u_2)=1$ and $\operatorname{adj}(v_1,v_2)=2$, or
\item $\operatorname{adj}(u_1,u_2) = \operatorname{adj}(v_1,v_2)=2$.
\end{enumerate}
Lemmas \ref{lemma:alt} and \ref{lemma:split} show that in each of the three above cases, the link is alternating. Thus either $\operatorname{adj}(u_1,u_2)=0$ or $\operatorname{adj}(v_1,v_2)=0$, and hence $D$ is $A$-almost alternating or $B$-almost alternating.

\end{proof}
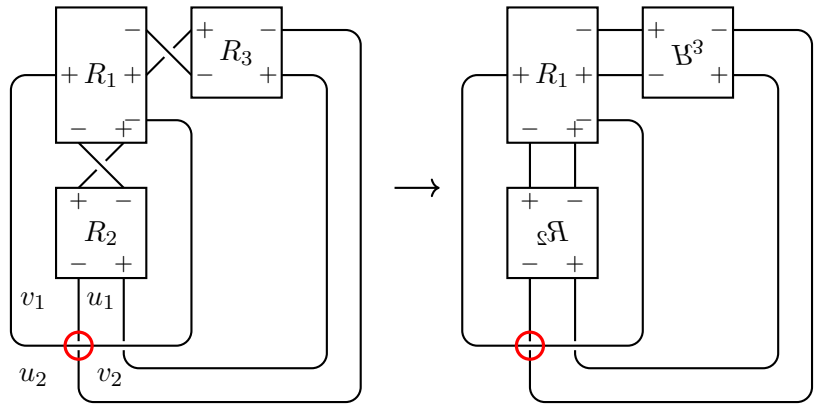
\begin{figure}[h]
$$\begin{tikzpicture}[thick, scale = .6]
\draw (0,0) rectangle (2,2);
\draw (0,3) rectangle (2,6);
\draw (3,4) rectangle (5,6);

\draw (1,4.5) node{$R_1$};
\draw (1,1) node{$R_2$};
\draw (4,5) node{$R_3$};

\draw (.3,4.5) node{\footnotesize{$+$}};
\draw (.5,3.3) node{\footnotesize{$-$}};
\draw (1.5,3.3) node{\footnotesize{$+$}};
\draw (1.7,3.5) node{\footnotesize{$-$}};
\draw (1.7,4.5) node{\footnotesize{$+$}};
\draw (1.7,5.5) node{\footnotesize{$-$}};

\draw (.5,1.7) node{\footnotesize{$+$}};
\draw (1.5,1.7) node{\footnotesize{$-$}};
\draw (.5,.3) node{\footnotesize{$-$}};
\draw (1.5,.3) node{\footnotesize{$+$}};

\draw (3.3,4.5) node{\footnotesize{$-$}};
\draw (3.3,5.5) node{\footnotesize{$+$}};
\draw (4.7,4.5) node{\footnotesize{$+$}};
\draw (4.7,5.5) node{\footnotesize{$-$}};

\draw (.5,3) -- (1.5,2);
\draw (.5,2) -- (.9,2.4);
\draw (1.1,2.6) -- (1.5,3);
\draw (2,5.5) -- (3,4.5);
\draw (2,4.5) -- (2.4,4.9);
\draw (2.6,5.1) -- (3,5.5);
\draw[rounded corners = 2mm] (2,3.5) -- (3,3.5) -- (3,-1.5) -- (-1,-1.5) -- (-1,4.5) -- (0,4.5);
\draw (1.5,0) -- (1.5,-1.4);
\draw[rounded corners = 2mm] (1.5,-1.6) -- (1.5,-2) -- (6,-2) -- (6,4.5) -- (5,4.5);
\draw (.5,0) -- (.5,-1.4);
\draw[rounded corners = 2mm] (.5,-1.6) -- (.5,-2.75) -- (6.75,-2.75) -- (6.75,5.5) -- (5,5.5);

\draw (-.5,-.5) node{$v_1$};
\draw (1.2,-2.2) node{$v_2$};
\draw (1,-.5) node{$u_1$};
\draw (-.5,-2.2) node{$u_2$};

\draw[red,very thick] (.5,-1.5) circle (.3cm);


\begin{scope}[xshift = 10cm]

\draw (0,0) rectangle (2,2);
\draw (0,3) rectangle (2,6);
\draw (3,4) rectangle (5,6);

\draw (1,4.5) node{$R_1$};
\draw (1,1) node{\reflectbox{$R_2$}};
\draw (4,5) node{\raisebox{\depth}{\scalebox{1}[-1]{$R_3$}}};
\draw (.3,4.5) node{\footnotesize{$+$}};
\draw (.5,3.3) node{\footnotesize{$-$}};
\draw (1.5,3.3) node{\footnotesize{$+$}};
\draw (1.7,3.5) node{\footnotesize{$-$}};
\draw (1.7,4.5) node{\footnotesize{$+$}};
\draw (1.7,5.5) node{\footnotesize{$-$}};

\draw (.5,1.7) node{\footnotesize{$+$}};
\draw (1.5,1.7) node{\footnotesize{$-$}};
\draw (.5,.3) node{\footnotesize{$-$}};
\draw (1.5,.3) node{\footnotesize{$+$}};

\draw (3.3,4.5) node{\footnotesize{$-$}};
\draw (3.3,5.5) node{\footnotesize{$+$}};
\draw (4.7,4.5) node{\footnotesize{$+$}};
\draw (4.7,5.5) node{\footnotesize{$-$}};

\draw (.5,3) -- (.5,2);
\draw (1.5,2) -- (1.5,3);

\draw (2,5.5) -- (3,5.5);
\draw (2,4.5) -- (3,4.5);
\draw[rounded corners = 2mm] (2,3.5) -- (3,3.5) -- (3,-1.5) -- (-1,-1.5) -- (-1,4.5) -- (0,4.5);
\draw (1.5,0) -- (1.5,-1.4);
\draw[rounded corners = 1mm] (.5,0) -- (.5,-1.4);
\draw[rounded corners = 2mm] (1.5,-1.6) -- (1.5,-2) -- (6,-2) -- (6,4.5) -- (5,4.5);
\draw (.5,0) -- (.5,-1.4);
\draw[rounded corners = 2mm] (.5,-1.6) -- (.5,-2.75) -- (6.75,-2.75) -- (6.75,5.5) -- (5,5.5);

\draw[red,very thick] (.5,-1.5) circle (.3cm);

\end{scope}

\draw[ thick, ->] (7.5,2) -- (8.5,2);

\end{tikzpicture}$$
\caption{If $\operatorname{adj}(u_1,u_2)=\operatorname{adj}(v_1,v_2)=1$, then the link has an almost alternating diagram with two fewer crossings.}
\label{figure:AAReduce}
\end{figure}

Armond and Lowrance \cite{ArmLow:Turaev} and independently Kim \cite{Kim:TuraevClassification} proved that every non-split Turaev genus one link has a diagram as in Figure \ref{figure:tg1}. Each tangle $R_i$ in Figure \ref{figure:tg1} is an alternating tangle both of whose closures are connected link diagrams. Kim used this classification to prove the following theorem, whose proof we sketch. 
\begin{figure}[h]
$$\begin{tikzpicture}[thick, scale = .8]
\draw [bend left] (0,.5) edge (3,.5);
\draw [bend right] (0,-.5) edge (3, -.5);
\draw [bend left] (3,.5) edge (6,.5);
\draw [bend right] (3,-.5) edge (6, -.5);
\draw [bend left] (6,.5) edge (9,.5);
\draw [bend right] (6,-.5) edge (9, -.5);
\draw [bend left] (7.5,.5) edge (10.5,.5);
\draw [bend right] (7.5,-.5) edge (10.5,-.5);

\fill[white] (7.5,1) rectangle (9.1,-1);
\draw (8.25,0) node{\Large{$\dots$}};

\draw (-.8,.6) arc (270:90:.8cm);
\draw (-.8,-.6) arc (90:270:.8cm);
\draw (11.3,.6) arc (-90:90:.8cm);
\draw (11.3,-.6) arc (90:-90:.8cm);
\draw (-.8,2.2) -- (11.3,2.2);
\draw (-.8,-2.2) -- (11.3,-2.2);

\fill[white] (0,0) circle (1cm);
\draw (0,0) node {$R_1$};
\draw (0,0) circle (1cm);
\fill[white] (3,0) circle (1cm);
\draw (3,0) node {$R_2$};
\draw (3,0) circle (1cm);
\fill[white] (6,0) circle (1cm);
\draw (6,0) node {$R_3$};
\draw (6,0) circle (1cm);
\fill[white] (10.5,0) circle (1cm);
\draw (10.5,0) node {$R_{2k}$};
\draw (10.5,0) circle (1cm);

\draw (-.5,.6) node{$-$};
\draw (-.5,-.6) node{$+$};
\draw (.4,.7) node{$+$};
\draw (.4,-.7) node{$-$};

\draw (2.6,.7) node{$+$};
\draw (2.6,-.7) node{$-$};
\draw (3.4, .7) node{$-$};
\draw (3.4,-.7) node{$+$};

\draw (5.6, .7) node{$-$};
\draw (5.6,-.7) node{$+$};
\draw (6.4,.7) node{$+$};
\draw (6.4,-.7) node{$-$};

\draw (10.1, .7) node{$+$};
\draw (10.1,-.7) node{$-$};
\draw (11,.6) node{$-$};
\draw (11,-.6) node{$+$};

\end{tikzpicture}$$
\caption{Every non-split link of Turaev genus one has a diagram as above. Each $R_i$ is an alternating tangle whose closures are connected link diagrams.}
\label{figure:tg1}
\end{figure}
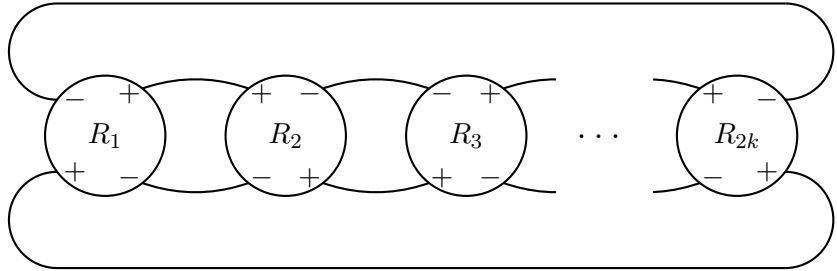

\begin{theorem}[Kim]
\label{theorem:TuraevInadequate}
Every Turaev genus one link has a diagram that is either $A$-adequate, $B$-adequate, or almost alternating.
\end{theorem}
\begin{proof}[Sketch of proof]
Let $D$ be a Turaev genus one diagram as in Figure \ref{figure:tg1}, and let $w_1$ and $w_2$ be the two regions that meet all of the alternating tangles $R_i$ for $i=1,\dots, 2k$. Kim calls any crossing in the boundary of both $w_1$ and $w_2$ an $AB$-loop crossing because it is precisely these crossing that will lead to the diagram being not $A$-adequate or not $B$-adequate. Performing a sequence of flypes can collect all of the $AB$-loop crossings into one twist region, and performing Reidemeister 2 moves in the twist region can ensure each crossing in the twist region has the same sign. 

Upon completing this process, if there are either zero or more than two crossings in the twist region, then the diagram is $A$-adequate or $B$-adequate. If there is exactly one crossing in the twist region and the total number of alternating tangles in the decomposition (other than the twist region) is greater than one, then the diagram is again either $A$-adequate or $B$-adequate. Finally, if there is exactly one crossing in the twist region and exactly one other alternating tangle in the decomposition, then the diagram looks like Figure \ref{figure:aadiagram}, and hence is almost alternating.
\end{proof}
Figure \ref{figure:TReduce} shows the algorithm described in the previous sketch. In the initial diagram, tangles of the form $R_i^j$ combine to make up the tangle $R_i$ from Figure \ref{figure:tg1}. The resulting link diagram has a twist region with two crossing plus four additional alternating tangles, and the final diagram is $B$-adequate.

\begin{figure}[h]
$$\begin{tikzpicture}[thick, scale = .75]

\draw (8,1.5) node{$w_1$};
\draw (8,-1.5) node{$w_2$};

\draw (12,.5) -- (14,.5);
\draw (12,-.5) -- (14,-.5);

\draw (16,.5) arc (-90:90:1);
\draw (0,.5) arc (-90:-270:1);
\draw (0,2.5) -- (16,2.5);

\draw (16,-.5) arc (90:-90:1);
\draw (0,-.5) arc (90:270:1);
\draw (0,-2.5) -- (16,-2.5);

\draw (0,-1) rectangle (1,1);
\draw (2,-1) rectangle (3,1);

\draw (4.5,-1) rectangle (5.5,1);
\draw (6.5,-1) rectangle (7.5,1);

\draw (9,-1) rectangle (10,1);
\draw (11,-1) rectangle (12,1);

\fill[white] (15,0) circle (1.5cm);
\draw (15,0) circle (1.5cm);

\draw (3,.5) -- (4.5,.5);
\draw (3,-.5) -- (4.5,-.5);
\draw (7.5,.5) -- (9,.5);
\draw (7.5,-.5) -- (9,-.5);

\draw (.5,0) node{\small{$R_1^1$}};
\draw (2.5,0) node{\small{$R_1^2$}};
\draw (5,0) node{\small{$R_2^1$}};
\draw (7,0) node{\small{$R_2^2$}};
\draw (9.5,0) node{\small{$R_3^1$}};
\draw (11.5,0) node{\small{$R_3^2$}};
\draw (15,0) node{\small{$R_4$}};

\draw (.25,.5) node{\tiny{$-$}};
\draw (.75,.5) node{\tiny{$+$}};
\draw (.25,-.5) node{\tiny{$+$}};
\draw (.75,-.5) node{\tiny{$-$}};

\begin{scope}[xshift = 2cm]
\draw (.25,.5) node{\tiny{$-$}};
\draw (.75,.5) node{\tiny{$+$}};
\draw (.25,-.5) node{\tiny{$+$}};
\draw (.75,-.5) node{\tiny{$-$}};
\end{scope}

\begin{scope}[xshift = 9cm]
\draw (.25,.5) node{\tiny{$-$}};
\draw (.75,.5) node{\tiny{$+$}};
\draw (.25,-.5) node{\tiny{$+$}};
\draw (.75,-.5) node{\tiny{$-$}};

\begin{scope}[xshift = 2cm]
\draw (.25,.5) node{\tiny{$-$}};
\draw (.75,.5) node{\tiny{$+$}};
\draw (.25,-.5) node{\tiny{$+$}};
\draw (.75,-.5) node{\tiny{$-$}};
\end{scope}

\end{scope}

\begin{scope}[xshift = 4.5cm]
\draw (.25,.5) node{\tiny{$+$}};
\draw (.75,.5) node{\tiny{$-$}};
\draw (.25,-.5) node{\tiny{$-$}};
\draw (.75,-.5) node{\tiny{$+$}};

\begin{scope}[xshift = 2cm]
\draw (.25,.5) node{\tiny{$+$}};
\draw (.75,.5) node{\tiny{$-$}};
\draw (.25,-.5) node{\tiny{$-$}};
\draw (.75,-.5) node{\tiny{$+$}};
\end{scope}

\end{scope}

\draw (14,.5) node{\tiny{$+$}};
\draw (14,-.5) node{\tiny{$-$}};
\draw (16,.5) node{\tiny{$-$}};
\draw (16,-.5) node{\tiny{$+$}};

\begin{knot}[
	consider self intersections,
 	clip width = 3,
 	ignore endpoint intersections = true,
	end tolerance = 2pt
	]
	\flipcrossings{1,4};
	\strand (1,.5) to [out = 0, in =180]
	(1.5,-.25) to [out = 0, in = 180]
	(2,.5);
	
	\strand (1,-.5) to [out = 0 , in =180]
	(1.5,.25) to [out = 0, in=180]
	(2,-.5);
	\strand (5.5,.5) to [out =0, in =180] (6.5,-.5);
	\strand (5.5,-.5) to [out = 0, in = 180] (6.5,.5);
	
	\strand (10,.5) to [out =0, in =180] (11,-.5);
	\strand (10,-.5) to [out = 0, in = 180] (11,.5);
	
\end{knot}
\begin{scope} [yshift = -6.5cm]

\draw (12,.5) -- (14,.5);
\draw (12,-.5) -- (14,-.5);

\draw (16,.5) arc (-90:90:1);
\draw (0,.5) arc (-90:-270:1);
\draw (0,2.5) -- (16,2.5);

\draw (16,-.5) arc (90:-90:1);
\draw (0,-.5) arc (90:270:1);
\draw (0,-2.5) -- (16,-2.5);

\begin{scope}[xshift = .5cm]
\draw (3.5,-1) rectangle (4.5,1);
\draw (2,-1) rectangle (3,1);

\draw (5.5,-1) rectangle (6.5,1);
\draw (7,-1) rectangle (8,1);

\draw (9,-1) rectangle (10,1);
\end{scope}
\draw (11,-1) rectangle (12,1);

\fill[white] (15,0) circle (1.5cm);
\draw (15,0) circle (1.5cm);
\begin{scope}[xshift = .5cm]
\draw (3,.5) -- (3.5,.5);
\draw (3,-.5) -- (3.5,-.5);
\draw (4.5,.5) -- (5.5,.5);
\draw (4.5,-.5) -- (5.5,-.5);
\draw (6.5,.5) -- (7,.5);
\draw (6.5,-.5) -- (7,-.5);
\draw (8,.5) -- (9,.5);
\draw (8,-.5) -- (9,-.5);
\end{scope}
\draw (10.5,.5) -- (11,.5);
\draw (10.5,-.5) -- (11,-.5);

\begin{scope}[xshift = .5cm]
\draw (2.5,0) node{\small{$R_1^1$}};
\draw (4,0) node{\small{$R_1^2$}};
\draw (6,0) node{\small{$R_2^1$}};
\draw (7.5,0) node{\raisebox{\depth}{\scalebox{1}[-1]{$R_2^2$}}};
\draw (9.5,0) node{\raisebox{\depth}{\scalebox{1}[-1]{$R_3^1$}}};
\end{scope}
\draw (11.5,0) node{\small{$R_3^2$}};
\draw (15,0) node{\small{$R_4$}};

\begin{scope}[xshift = 4cm]
\draw (.25,.5) node{\tiny{$-$}};
\draw (.75,.5) node{\tiny{$+$}};
\draw (.25,-.5) node{\tiny{$+$}};
\draw (.75,-.5) node{\tiny{$-$}};
\end{scope}

\begin{scope}[xshift = 2.5cm]
\draw (.25,.5) node{\tiny{$-$}};
\draw (.75,.5) node{\tiny{$+$}};
\draw (.25,-.5) node{\tiny{$+$}};
\draw (.75,-.5) node{\tiny{$-$}};
\end{scope}

\begin{scope}[xshift = 9.5cm]
\draw (.25,.5) node{\tiny{$-$}};
\draw (.75,.5) node{\tiny{$+$}};
\draw (.25,-.5) node{\tiny{$+$}};
\draw (.75,-.5) node{\tiny{$-$}};

\begin{scope}[xshift = 1.5cm]
\draw (.25,.5) node{\tiny{$-$}};
\draw (.75,.5) node{\tiny{$+$}};
\draw (.25,-.5) node{\tiny{$+$}};
\draw (.75,-.5) node{\tiny{$-$}};
\end{scope}

\end{scope}

\begin{scope}[xshift = 6cm]
\draw (.25,.5) node{\tiny{$+$}};
\draw (.75,.5) node{\tiny{$-$}};
\draw (.25,-.5) node{\tiny{$-$}};
\draw (.75,-.5) node{\tiny{$+$}};

\begin{scope}[xshift = 1.5cm]
\draw (.25,.5) node{\tiny{$+$}};
\draw (.75,.5) node{\tiny{$-$}};
\draw (.25,-.5) node{\tiny{$-$}};
\draw (.75,-.5) node{\tiny{$+$}};
\end{scope}

\end{scope}

\draw (14,.5) node{\tiny{$+$}};
\draw (14,-.5) node{\tiny{$-$}};
\draw (16,.5) node{\tiny{$-$}};
\draw (16,-.5) node{\tiny{$+$}};

\begin{knot}[
	consider self intersections,
 	clip width = 3,
 	ignore endpoint intersections = true,
	end tolerance = 2pt
	]
	\flipcrossings{1};
	
	\strand (0,.5) to [out = 0, in = 180]
	(.5,-.25) to [out = 0, in =180]
	(1,.25) to [out = 0, in = 180]
	(1.5,-.25) to [out = 0, in =180]
	(2,.5) to [out = 0, in = 180]
	(2.5,.5);
	
	\strand (0,-.5) to [out = 0, in = 180]
	(.5,.25) to [out = 0, in =180]
	(1,-.25) to [out = 0, in = 180]
	(1.5,.25) to [out = 0, in =180]
	(2,-.5) to [out = 0, in = 180]
	(2.5,-.5);

\end{knot}
\end{scope}
\begin{scope} [yshift = -13cm]

\draw (12,.5) -- (14,.5);
\draw (12,-.5) -- (14,-.5);

\draw (16,.5) arc (-90:90:1);
\draw (0,.5) arc (-90:-270:1);
\draw (0,2.5) -- (16,2.5);

\draw (16,-.5) arc (90:-90:1);
\draw (0,-.5) arc (90:270:1);
\draw (0,-2.5) -- (16,-2.5);

\begin{scope}[xshift = .5cm]
\draw (3.5,-1) rectangle (4.5,1);
\draw (2,-1) rectangle (3,1);

\draw (5.5,-1) rectangle (6.5,1);
\draw (7,-1) rectangle (8,1);

\draw (9,-1) rectangle (10,1);
\end{scope}
\draw (11,-1) rectangle (12,1);

\fill[white] (15,0) circle (1.5cm);
\draw (15,0) circle (1.5cm);
\begin{scope}[xshift = .5cm]
\draw (3,.5) -- (3.5,.5);
\draw (3,-.5) -- (3.5,-.5);
\draw (4.5,.5) -- (5.5,.5);
\draw (4.5,-.5) -- (5.5,-.5);
\draw (6.5,.5) -- (7,.5);
\draw (6.5,-.5) -- (7,-.5);
\draw (8,.5) -- (9,.5);
\draw (8,-.5) -- (9,-.5);
\end{scope}
\draw (10.5,.5) -- (11,.5);
\draw (10.5,-.5) -- (11,-.5);

\begin{scope}[xshift = .5cm]
\draw (2.5,0) node{\small{$R_1^1$}};
\draw (4,0) node{\small{$R_1^2$}};
\draw (6,0) node{\small{$R_2^1$}};
\draw (7.5,0) node{\raisebox{\depth}{\scalebox{1}[-1]{$R_2^2$}}};
\draw (9.5,0) node{\raisebox{\depth}{\scalebox{1}[-1]{$R_3^1$}}};
\end{scope}
\draw (11.5,0) node{\small{$R_3^2$}};
\draw (15,0) node{\small{$R_4$}};

\begin{scope}[xshift = 4cm]
\draw (.25,.5) node{\tiny{$-$}};
\draw (.75,.5) node{\tiny{$+$}};
\draw (.25,-.5) node{\tiny{$+$}};
\draw (.75,-.5) node{\tiny{$-$}};
\end{scope}

\begin{scope}[xshift = 2.5cm]
\draw (.25,.5) node{\tiny{$-$}};
\draw (.75,.5) node{\tiny{$+$}};
\draw (.25,-.5) node{\tiny{$+$}};
\draw (.75,-.5) node{\tiny{$-$}};
\end{scope}

\begin{scope}[xshift = 9.5cm]
\draw (.25,.5) node{\tiny{$-$}};
\draw (.75,.5) node{\tiny{$+$}};
\draw (.25,-.5) node{\tiny{$+$}};
\draw (.75,-.5) node{\tiny{$-$}};

\begin{scope}[xshift = 1.5cm]
\draw (.25,.5) node{\tiny{$-$}};
\draw (.75,.5) node{\tiny{$+$}};
\draw (.25,-.5) node{\tiny{$+$}};
\draw (.75,-.5) node{\tiny{$-$}};
\end{scope}

\end{scope}

\begin{scope}[xshift = 6cm]
\draw (.25,.5) node{\tiny{$+$}};
\draw (.75,.5) node{\tiny{$-$}};
\draw (.25,-.5) node{\tiny{$-$}};
\draw (.75,-.5) node{\tiny{$+$}};

\begin{scope}[xshift = 1.5cm]
\draw (.25,.5) node{\tiny{$+$}};
\draw (.75,.5) node{\tiny{$-$}};
\draw (.25,-.5) node{\tiny{$-$}};
\draw (.75,-.5) node{\tiny{$+$}};
\end{scope}

\end{scope}

\draw (14,.5) node{\tiny{$+$}};
\draw (14,-.5) node{\tiny{$-$}};
\draw (16,.5) node{\tiny{$-$}};
\draw (16,-.5) node{\tiny{$+$}};

\begin{knot}[
	consider self intersections,
 	clip width = 3,
 	ignore endpoint intersections = true,
	end tolerance = 2pt
	]
	\flipcrossings{1};
	
	\strand (0,.5) to [out = 0, in = 180]
	(.5,-.25) to [out = 0, in =180]
	(1,.5) to [out = 0, in = 180]
	(2.5,.5);
	
	\strand (0,-.5) to [out = 0, in = 180]
	(.5,.25) to [out = 0, in =180]
	(1,-.5) to [out = 0, in = 180]
	(2.5,-.5);

\end{knot}
\end{scope}
\end{tikzpicture}$$
\caption{An example of the algorithm described in the proof of Theorem \ref{theorem:TuraevInadequate}.}
\label{figure:TReduce}
\end{figure}
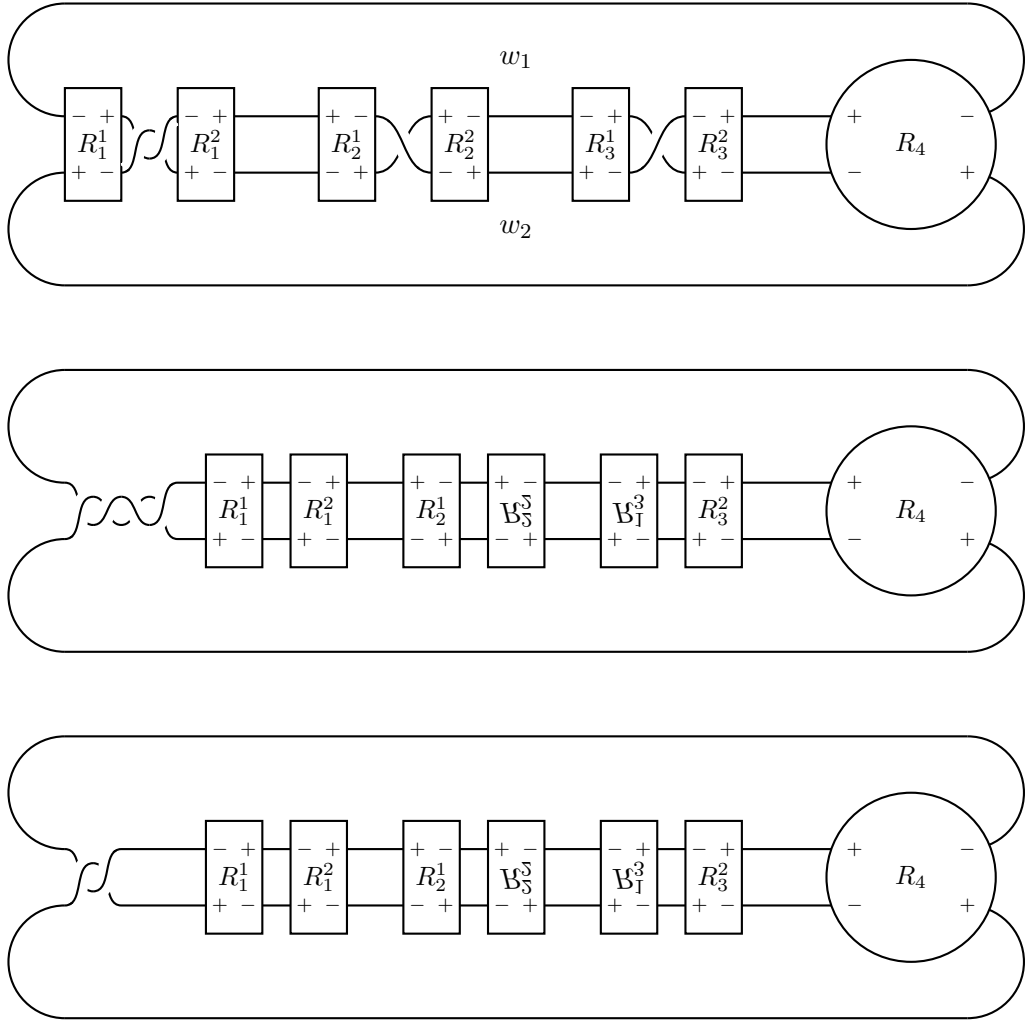

Theorems  \ref{theorem:ABaa} and \ref{theorem:TuraevInadequate} combine to give the following corollary.
\begin{corollary}
\label{corollary:ABgt1}
Every Turaev genus one link is $A$-Turaev genus one or $B$-Turaev genus one.
\end{corollary}

\section{Extremal Khovanov homology}
\label{section:Extremal}

\subsection{$A$- and $B$-adequate links}
Recall that a link diagram $D$ is $A$-adequate if no two arcs in the $A$-resolution of any crossing lie on the same component of the all-$A$ state of $D$.  Similarly a link diagram is $B$-adequate if no two arcs in the $B$-resolution of any crossing lie on the same component of the all-$B$ state. Links can be partitioned into three sets: adequate, semi-adequate, and inadequate. Adequate links are those links with a diagram that is both $A$-adequate and $B$-adequate, semi-adequate links are those links with a diagram that is either $A$-adequate or $B$-adequate, but not both, and inadequate links are those links where every diagram is neither $A$-adequate nor $B$-adequate.

Lickorish and Thistlethwaite \cite{LT:Adequate} introduced adequate and semi-adequate links and studied their Jones polynomials. Thistlethwaite \cite{Thistlethwaite:Kauffman} later studied the Kauffman polynomial of adequate links. Semi-adequate links play an important role in hyperbolic geometry (see \cite{FKP:HyperbolicSemi}, \cite{FKP:Survey} and references within) and in the study of the colored Jones polynomial (see \cite{Armond:HeadTail, FKP:Guts, KL:ColoredJones, Kalfagianni:Adequate}). 

The following theorem is a version of Theorem \ref{theorem:Diagonal} for $A$-adequate or $B$-adequate links. Khovanov \cite{Kh:Patterns} proved that the extremal Khovanov homology of an $A$-adequate or $B$-adequate link is isomorphic to $\mathbb{Z}$ (see Theorem \ref{theorem:ExtAdeq}). Our contribution is to determine the diagonal grading where this $\mathbb{Z}$ summand is supported.
\begin{theorem}
\label{theorem:Semi}
Let $L$ be a non-split link. 
\begin{enumerate}
\item If $L$ is $A$-adequate, then there is an $i_0\in\mathbb{Z}$ such that $$Kh^{*,j_{\min}(L)}(L) = Kh^{i_0(L),j_{\min}(L)}(L)\cong \mathbb{Z}$$ and $2i_0(L)-j_{\min}(L) = \delta_{\min}(L) + 2$.
\item If $L$ is $B$-adequate, then there is an $i_0\in\mathbb{Z}$ such that $$Kh^{*,j_{\max}(L)}(L) = Kh^{i_0(L),j_{\max}(L)}(L) \cong \mathbb{Z}$$ and 
$2i_0(L) - j_{\max}(L) = \delta_{\max}(L) -2$.
\end{enumerate}
\end{theorem}
\begin{proof}
Suppose that $L$ is an $A$-adequate link with $A$-adequate diagram $D$. Theorem \ref{theorem:ExtAdeq} and Equation \ref{equation:shift} imply that $j_{\min}(L) = -s_A(D) + c_+(D) - 2c_-(D)$ and
$$Kh^{*,j_{\min}(L)}(L) = Kh^{-c_-(D),j_{\min}(L)}(L)\cong \mathbb{Z}.$$
Hence the diagonal grading of this summand is 
$$2i - j = -2c_-(D) - (-s_A(D) + c_+(D) - 2c_-(D)) = s_A(D) - c_+(D).$$
Theorem \ref{theorem:DiagonalGrading} implies that $s_A(D)-c_+(D)-2 \leq \delta_{\min}(L)$. Hence Theorem \ref{theorem:KnightMove} and the fact that $j_{\min}(L)$ is the minimum quantum grading imply that $\operatorname{rank} Kh^{p,q}(L)>0$ for $(p,q) = (-c_-(D),j_{\min}(L)+2)$ or $(1-c_-(D),j_{\min}(L)+4)$. In either case, we conclude that $\delta_{\min}(L)=s_A(D) - c_+(D) -2$ and that 
$$Kh^{*,j_{\min}(L)}(L)=Kh^{i_0(L),j_{\min}(L)}(L)\cong \mathbb{Z}$$
where $2 i_0(L)-j_{\min}(L) = \delta_{\min}(L) + 2$. The case where $L$ is $B$-adequate follows from Equation \ref{equation:Mirror}. 
\end{proof}

\subsection{$A$- and $B$-almost alternating links}
\label{section:Main}
In this section, we prove Theorem \ref{theorem:Diagonal}. Since every Turaev genus one link is $A$-adequate, $B$-adequate, $A$-almost alternating, or $B$-almost alternating, it remains to show that Theorem \ref{theorem:Diagonal} holds for almost alternating links. First, we compute the quantum grading of the extremal Khovanov group of an $A$-almost alternating link and show that extremal Khovanov group is isomorphic to $\mathbb{Z}$. Next we show how the diagonal grading of this summand relates to the minimal and maximal diagonal gradings. We use Equation \ref{equation:Mirror} to deduce the analogous result for $B$-almost alternating links.

Our next goal is to prove that if $D$ is an $A$-almost alternating diagram, then $j_{\min}(D)=-s_A(D)+2$ and $\underline{Kh}^{*,j_{\min}(D)}(D)\cong \mathbb{Z}$. Our first task is to show that $\underline{Kh}(D)$ is trivial in quantum gradings below $-s_A(D)+2$.

\begin{theorem}
\label{theorem:lowq}
Let $D$ be an $A$-almost alternating diagram, and let $s_A(D)$ be the number of components in the all-$A$ state of $D$. If $j<-s_A(D)+2$, then $\underline{Kh}^{i,j}(D)=0$ for any $i\in\mathbb{Z}$.
\end{theorem}
\begin{proof}
If $j<-s_A(D)$, then there are no enhanced states with quantum grading $j$, and the result follows. Since the quantum gradings of all states of $D$ are equivalent mod $2$, it suffices to focus on the case where $j=-s_A(D)$. There are exactly two enhanced states with quantum grading $j = -s_A(D)$. The first of these states $S_1$ is the all-$A$ Kauffman state enhanced with an $x$ label on each component. The second of these states $S_2$ is the state where every crossing other than the dealternator is assigned an $A$-resolution, the dealternator is assigned the $B$-resolution, and every component is enhanced with an $x$ label. The differential in the Khovanov complex takes $S_1$ to $S_2$ and so neither state represents a homology class. Thus $\underline{Kh}^{i,-s_A(D)}(D) = 0$ for all $i\in\mathbb{Z}$.
\end{proof}

The following lemma is the main technical tool in our extremal Khovanov computations, and it is a consequence of the long exact sequence in Theorem \ref{theorem:LES}.
\begin{lemma}
\label{lemma:LESjmin}
Let $D$, $D_A$, and $D_B$ be link diagrams that are identical outside of a neighborhood of a crossing in $D$ and where that crossing is replaced by an $A$-resolution in $D_A$ and a $B$-resolution in $D_B$. If $j_{\min}(D_A)-1 < j_{\min}(D_B)$, then $j_{\min}(D)=j_{\min}(D_A)$ and $\underline{Kh}^{i,j_{\min}(D)}(D) \cong \underline{Kh}^{i,j_{\min}(D_A)}(D_A)$ for all $i\in\mathbb{Z}$.
\end{lemma}
\begin{proof}
Suppose that $j<j_{\min}(D_A)$. Then $j-1< j_{\min}(D_A)-1 < j_{\min}(D_B)-1.$ Therefore $\underline{Kh}^{*,j-1}(D_B)$ and $\underline{Kh}^{*,j}(D_A)$ are trivial, and the long exact sequence of Theorem \ref{theorem:LES} implies that $\underline{Kh}^{*,j}(D)$ is also trivial.

Since $j_{\min}(D_A)-1 < j_{\min}(D_B)$, it follows that $\underline{Kh}^{*,j_{\min}(D_A)-1}(D_B)$ is trivial. The long exact sequence of Theorem \ref{theorem:LES} becomes
$$0\to \underline{Kh}^{i,j_{\min}(D_A)}(D) \to \underline{Kh}^{i,j_{\min}(D_A)}(D_A) \to 0$$
for each $i\in\mathbb{Z}$. Therefore $\underline{Kh}^{i,j_{\min}(D_A)}(D) \cong \underline{Kh}^{i,j_{\min}(D_A)}(D_A)$ for each $i\in\mathbb{Z}$, and the result follows.
\end{proof}

In the following theorem, we prove that the extremal Khovanov homology of an almost alternating link is isomorphic to $\mathbb{Z}$. 
\begin{theorem}
\label{theorem:AAKh}
Let $D$ be a non-split almost alternating diagram.
\begin{enumerate}
\item If $D$ is $A$-almost alternating, then $j_{\min}(D) = 2-s_A(D)$ and $$\underline{Kh}^{*,j_{\min}(D)}(D) = \underline{Kh}^{1,j_{\min}(D)}(D) \cong \mathbb{Z}.$$
\item If $D$ is $B$-almost alternating, then $j_{\max}(D) = s_B(D) -2$ and $$\underline{Kh}^{*,j_{\max}(D)}(D) = \underline{Kh}^{c(D)-1,j_{\max}(D)}(D)\cong \mathbb{Z}.$$
\end{enumerate}
\end{theorem}
\begin{proof}
Suppose that $D$ is an $A$-almost alternating diagram such that the diagram obtained by taking an $A$-resolution for any crossing in $R$ is not $A$-almost alternating. Taking an $A$-resolution corresponds to deleting an edge in $G$ and thus does not increase $\operatorname{adj}(u_1,u_2)$. Therefore, it must be the case the diagram obtained by performing an $A$-resolution of any crossing in $R$ has a crossing in the boundary of both $v_1$ and $v_2$. Hence $D$ is the diagram in Figure \ref{figure:basecase}. Figure \ref{figure:basecase} shows that in this case $D$ and $D_{\alt}$ are diagrams of the same alternating link. Since $D_{\alt}$ is alternating, $j_{\min}(D_{\alt})=-s_A(D_{\alt}) = -a-b+4$ and $\underline{Kh}^{*,j_{\min}(D_{\alt})}(D_{\alt})=\underline{Kh}^{0,-a-b+4}(D_{\alt})\cong \mathbb{Z}$. Proposition \ref{prop:Reidemeister} then implies that $j_{\min}(D) = -a-b+1 = -s_A(D)+2$ and that $\underline{Kh}^{*,j_{\min}(D)}(D) = \underline{Kh}^{1,-s_A(D)+2}(D)\cong \mathbb{Z}$, as desired.

Now suppose that $D$ is an $A$-almost alternating diagram such that $D$ contains a crossing in $R$ whose $A$-resolution $D_A$ is also $A$-almost alternating. By induction, $j_{\min}(D_A) = -s_A(D_A)+2$ and $\underline{Kh}^{*,j_{\min}(D_A)}(D_A) = \underline{Kh}^{1,-s_A(D_A)+2}(D_A)\cong \mathbb{Z}$. The diagram $D_B$ is almost alternating, and hence Theorem \ref{theorem:lowq} implies that $j_{\min}(D_B)\geq -s_A(D_B)+2$. Since $s_A(D) = s_A(D_A) = s_A(D_B)+1$, it follows that $j_{\min}(D_A) - 1 = -s_A(D_A) +1 < -s_A(D_B)+2\leq j_{\min}(D_B)$. Thus Lemma \ref{lemma:LESjmin} implies that $j_{\min}(D)=j_{\min}(D_A)=-s_A(D_A)+2=-s_A(D)+2$ and $\underline{Kh}^{*,j_{\min}(D)}(D) = \underline{Kh}^{1,-s_A(D)+2}(D)\cong \mathbb{Z}$.

The case where $D$ is $B$-almost alternating follows from the $A$-almost alternating case and Equation \ref{equation:Mirror}.
\end{proof}
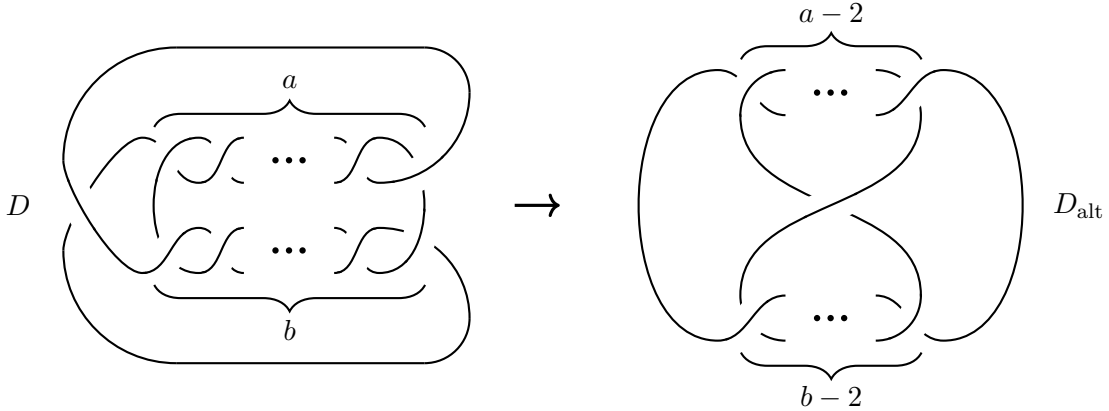
\begin{figure}[h]
$$\begin{tikzpicture}[thick, scale = .6]
\draw (0,1.5) node{$D$};
\draw (23.5, 1.5) node{$D_{\alt}$};
\begin{knot}[
 	clip width = 4,
 	ignore endpoint intersections = true,
	end tolerance = 2pt
	]
	\flipcrossings{1,2,5,8,7};

	\strand(1,.5) to [out = 90, in = 180, looseness=.5]
	(2.75,3) to [out = 0, in = 180]
	(4,2) to [out = 0, in = 180]
	(5,3);
	
	\strand (1,2.5) to [out = 270, in =180, looseness=.5]
	(2.75,0) to [out = 0, in = 180]
	(4,1) to [out = 0, in = 180]
	(5,0);
	
	\strand (5,1) to [out = 180, in = 0]
	(4,0) to [out = 180, in = 270]
	(3,1.5) to [out = 90, in= 180]
	(4,3) to [out = 0, in = 180]
	(5,2);
	
	\strand (7,0) to [out = 0, in = 180]
	(8,1) to [out = 0, in = 90]
	(10,-1);
	
	\strand (7,3) to [out = 0, in = 180]
	(8,2) to [out = 0, in = 270]
	(10,4);
	
	\strand (7,1) to [out = 0, in = 180]
	(8,0) to [out = 0, in = 270]
	(9,1.5) to [out = 90, in = 0]
	(8,3) to [out = 180, in = 0]
	(7,2);
	
\end{knot}

\draw (1,2.5) arc (180:90:2.5cm);
\draw (1,.5) arc (180:270:2.5cm);
\draw (10,-1) arc (0:-90:1cm);
\draw (10,4) arc (0:90:1cm);
\draw (9,-2) -- (3.5,-2);
\draw (9,5) -- (3.5,5);
\fill (6,.5) circle (2pt);
\fill (6.3,.5) circle (2pt);
\fill (5.7,.5) circle (2pt);
\fill (6,2.5) circle (2pt);
\fill (6.3,2.5) circle (2pt);
\fill (5.7,2.5) circle (2pt);

\draw [thick,decorate,decoration={brace,amplitude=10pt,mirror},xshift=0.4pt,yshift=-0.4pt](3,-.25) -- (9,-.25) node[black,midway,yshift=-0.6cm] {$b$};

\draw [thick,decorate,decoration={brace,amplitude=10pt},xshift=0.4pt,yshift=-0.4pt](3,3.25) -- (9,3.25) node[black,midway,yshift=0.6cm] {$a$};

\draw[very thick, ->] (11,1.5) -- (12,1.5);

\begin{scope}[xshift = 12cm, yshift = 1.5cm]

	\begin{knot}[
 	clip width = 4,
 	ignore endpoint intersections = true,
	end tolerance = 2pt
	]
	\flipcrossings{5,2};
	\strand (5,3) to [out = 180, in = 90]
	(4,2) to [out = 270, in = 90]
	(8,-2) to [out = 270, in = 0]
	(7,-3);
	
	\strand (5,2) to [out = 180, in = 0]
	(3.5,3) to [out = 180, in = 180]
	(3.5,-3) to [out = 0, in = 180]
	(5,-2); 
	
	\strand (5,-3) to [out = 180, in = 270]
	(4,-2) to [out = 90, in = 270]
	(8,2) to [out = 90, in = 0]
	(7,3);
	
	\strand (7,2) to [out = 0, in = 180]
	(8.5,3) to [out = 0, in = 0]
	(8.5,-3) to [out = 180, in = 0]
	(7,-2);
	
	\end{knot}
	
	\fill (6,-2.5) circle (2pt);
	\fill (6.3,-2.5) circle (2pt);
	\fill (5.7,-2.5) circle (2pt);
	\fill (6,2.5) circle (2pt);
	\fill (6.3,2.5) circle (2pt);
	\fill (5.7,2.5) circle (2pt);
	
	\draw [thick,decorate,decoration={brace,amplitude=10pt,mirror},xshift=0.4pt,yshift=-0.4pt](4,-3.25) -- (8,-3.25) node[black,midway,yshift=-0.6cm] {$b-2$};
	\draw [thick,decorate,decoration={brace,amplitude=10pt},xshift=0.4pt,yshift=-0.4pt](4,3.25) -- (8,3.25) node[black,midway,yshift=0.6cm] {$a-2$};

\end{scope}
\end{tikzpicture}$$
\caption{The link diagram $D$ can be transformed into the alternating diagram $D_{\alt}$.}
\label{figure:basecase}
\end{figure}

\begin{corollary}
\label{corollary:Diagonal}
Let $L$ be a non-split link. 
\begin{enumerate}
\item If $L$ is $A$-almost alternating, then there is an $i_0\in\mathbb{Z}$ such that $$Kh^{*,j_{\min}(L)}(L) = Kh^{i_0(L),j_{\min}(L)}(L)\cong \mathbb{Z}$$ and $2i_0(L)-j_{\min}(L) = \delta_{\min}(L) + 2$.
\item If $L$ is $B$-almost alternating, then there is an $i_0\in\mathbb{Z}$ such that $$Kh^{*,j_{\max}(L)}(L) = Kh^{i_0(L),j_{\max}(L)}(L) \cong \mathbb{Z}$$ and 
$2i_0(L) - j_{\max}(L) = \delta_{\max}(L) -2$.
\end{enumerate}
\end{corollary}
\begin{proof}
Let $D$ be an $A$-almost alternating diagram of $L$. Theorem \ref{theorem:AAKh} and Equation \ref{equation:shift} imply that $j_{\min}(L) = 2-s_A(D) + c_+(D) - 2c_-(D)$ and that 
$$Kh^{*,j_{\min}(L)}(L) = Kh^{1-c_-(D),j_{\min}(L)}(L)\cong \mathbb{Z}.$$
Thus $2i_0(L)-j_{\min}(L) = 2(1-c_-(D)) - (2-s_A(D)+c_+(D) -2c_-(D)) = s_A(D) - c_+(D)$. Since $Kh^{*,j}(L)=0$ when $j< j_{\min}(L)$, Theorems \ref{theorem:KnightMove} and \ref{theorem:DiagonalGrading} imply that either $Kh^{i_0(L), j_{\min}(L)+2}(L)$ or $Kh^{i_0(L)+1,j_{\min}(L)+4}(L)$ are nontrivial.  Thus $\delta_{\min}(L) = s_A(D)-c_+(D) - 2$ and $2i_0(L) - j_{\min}(L) = \delta_{\min}(L)+2$. 

The case where $L$ is $B$-almost alternating is similar.
\end{proof}

\begin{proof}[Proof of Theorem \ref{theorem:Diagonal}]
If $L$ is $A$-adequate, then Theorem \ref{theorem:Semi} is the desired result. If $L$ is $A$-almost alternating, then Corollary \ref{corollary:Diagonal} is the desired result. If $L$ is $A$-Turaev genus one, then $L$ is either $A$-adequate or $A$-almost alternating. The case where $L$ is $B$-Turaev genus one is handled similarly.
\end{proof}
Corollary \ref{corollary:Main} follows from Theorem \ref{theorem:Diagonal} and Corollary \ref{corollary:ABgt1}.
\begin{proof}[Proof of Corollary \ref{corollary:Main}]
Let $L$ be a non-split link with Turaev genus one. Corollary \ref{corollary:ABgt1} implies that $L$ is $A$-Turaev genus one or $B$-Turaev genus one. If $L$ is $A$-Turaev genus one, then Theorem \ref{theorem:Diagonal} implies $Kh^{*,j_{\min}(L)}(L)\cong \mathbb{Z}$, and if $L$ is $B$-Turaev genus one, then Theorem \ref{theorem:Diagonal} implies $Kh^{*,j_{\max}(L)}(L)\cong\mathbb{Z}$.
\end{proof}

Let $Kh_{\text{odd}}(L)$ denote the odd Khovanov homology of the link $L$. The following theorem is a version of Theorem \ref{theorem:Diagonal} for $Kh_{\text{odd}}(L)$.
\begin{theorem}
\label{theorem:OddDiagonal}
Suppose that $L$ is a non-split link.
\begin{enumerate}
\item If $L$ is $A$-Turaev genus one, then the minimum quantum gradings of $Kh(L)$ and $Kh_{\text{odd}}(L)$ are equal, and there is an $i_0\in\mathbb{Z}$ such that $Kh^{*,j_{\min}(L)}_{\text{odd}}(L) = Kh^{i_0,j_{\min}(L)}_{\text{odd}}(L)\cong \mathbb{Z}$ and $2i_0-j_{\min}(L) = \delta_{\min}(L) + 2$.
\item If $L$ is $B$-Turaev genus one, then the maximum quantum grading of $Kh(L)$ and $Kh_{\text{odd}}(L)$ are equal, and there is an $i_0\in\mathbb{Z}$ such that $Kh^{*,j_{\max}(L)}_{\text{odd}}(L) = Kh^{i_0,j_{\max}(L)}_{\text{odd}}(L) \cong \mathbb{Z}$ and 
$2i_0 - j_{\max}(L) = \delta_{\max}(L) -2$.
\end{enumerate}
\end{theorem}
\begin{proof}
Suppose that $L$ is almost alternating. The proofs of Theorem \ref{theorem:lowq}, Lemma \ref{lemma:LESjmin}, Theorem \ref{theorem:AAKh}, and Corollary \ref{corollary:Diagonal} rely on gradings in the Khovanov chain complex, the fact that both extremal Khovanov homology groups of an alternating link are isomorphic to $\mathbb{Z}$, and the long exact sequence of Theorem \ref{theorem:LES}. Since the chain groups for odd Khovanov homology are isomorphic to the chain groups for Khovanov homology, the extremal odd Khovanov homology groups of an alternating link are isomorphic to $\mathbb{Z}$, and the long exact sequences of Theorem \ref{theorem:LES} also hold in the odd Khovanov setting, Corollary \ref{corollary:Diagonal} holds for odd Khovanov homology. Thus the desired result holds for almost alternating links.

Bloom \cite{Bloom:Mutation} proved that odd Khovanov homology is mutation invariant. Armond and Lowrance \cite{ArmLow:Turaev} proved that if $L$ is a Turaev genus one link, then there is a sequence of mutations transforming $L$ into an almost alternating link. Therefore the desired result holds for Turaev genus one links.
\end{proof}

\section{Examples}
\label{section:Examples}
In the next few examples, we give computations of Khovanov homology. The Khovanov homology is presented via a grid where the summand $Kh^{i,j}(L)$ is represented in the $(i,j)$ cell of the grid. A summand of $\mathbb{Z}^k$ is represented by $k$, a summand of $\mathbb{Z}_p^k$ is represented by $k_p$, and different summands in the same bigrading are separated by commas. All Khovanov homology computations in the examples were performed using the KnotTheory Mathematica package \cite{knotatlas}.
\begin{example}
\label{example:10_132}
Let $K=10_{132}\#\overline{10_{132}}$ be the connected sum of the knot $10_{132}$ and its mirror, as in Figure \ref{figure:10_132}. Both the leading and trailing coefficients of the Jones polynomial of $K$ are $-1$, and hence the results of \cite{DasLow:TuraevJones} do not apply. However, both Khovanov groups in the minimal and maximal quantum gradings have rank two. Therefore Theorem \ref{theorem:Diagonal} implies that $K$ is inadequate and has Turaev genus and dealternating number at least two. Since the Khovanov homology of $K$ lies on four diagonals, this knot was known to be of Turaev genus and dealternating number at least two by \cite{Manturov:Minimal}, \cite{CKS:Graphs}, and \cite{AP:Torsion}.
\end{example}
\begin{figure}[h]
 $$\begin{tikzpicture}
 \begin{scope}[scale = .8, yscale = .9]

 \begin{knot}[
 	clip width= 7,
	consider self intersections,
	ignore endpoint intersections = false,
	]
 	\flipcrossings{2,4,5,7,9,11,14,16,18};
	\strand[thick, rotate = 90] (2.2,1.5) to [curve through = {(3,1.8) .. (3.5,3)  .. (2.5,5) .. (2,5.5) .. (2.5,6) .. (5,6) .. (6,4) .. (5,3.5) .. (4,4) .. (2.5,3) .. (3,2) .. (5,1) .. (8,1.5) .. (9,1.8) .. (9.5,3) .. (8.5,5) .. (8,5.5) .. (8.5,6) .. (11,6) .. (12,4) .. (11,3.5) .. (10,4) .. (8.5,3) .. (9,2) .. (11,1) .. (12,3) .. (11,4) .. (10,3) .. (9.5,1) .. (8.5,1) .. (7.8,2) .. (7,3) .. (7,4) .. (8,5) .. (8.5,5.5) .. (8,6) .. (8,4) .. (5,4) .. (4,3) .. (3.5,1) .. (2.5,1) .. (1.8,2) .. (1,3) .. (1,4) .. (2,5)  .. (2.5,5.5) .. (2,6) .. (2,4) .. (1,2)}] (1.8,1.5);

\end{knot}
\end{scope}

\begin{scope}[yshift = .5cm, xshift = 1cm, scale = .55]
\draw (0,0) rectangle (16,15);
\foreach \i in {1,...,15}
{	
	\draw (\i,0) -- (\i,15);
}
\foreach \i in {1,...,14}
{
	\draw (0,\i) -- (16,\i);
}

\draw (.5,.5) node{\scriptsize{$\minus 13$}};
\draw (.5,1.5) node{\scriptsize{$\minus 11$}};
\draw (.5,2.5) node{\scriptsize{$\minus 9$}};
\draw (.5,3.5) node{\scriptsize{$\minus 7$}};
\draw (.5,4.5) node{\scriptsize{$\minus 5$}};
\draw (.5,5.5) node{\scriptsize{$\minus 3$}};
\draw (.5,6.5) node{\scriptsize{$\minus 1$}};
\draw (.5,7.5) node{\scriptsize{$1$}};
\draw (.5,8.5) node{\scriptsize{$3$}};
\draw (.5,9.5) node{\scriptsize{$5$}};
\draw (.5,10.5) node{\scriptsize{$7$}};
\draw (.5,11.5) node{\scriptsize{$9$}};
\draw (.5,12.5) node{\scriptsize{$11$}};
\draw (.5,13.5) node{\scriptsize{$13$}};

\draw (1.5,14.5) node{\scriptsize{$\minus 7$}};
\draw (2.5,14.5) node{\scriptsize{$\minus 6$}};
\draw (3.5,14.5) node{\scriptsize{$\minus 5$}};
\draw (4.5,14.5) node{\scriptsize{$\minus 4$}};
\draw (5.5,14.5) node{\scriptsize{$\minus 3$}};
\draw (6.5,14.5) node{\scriptsize{$\minus 2$}};
\draw (7.5,14.5) node{\scriptsize{$\minus 1$}};
\draw (8.5,14.5) node{\scriptsize{$0$}};
\draw (9.5,14.5) node{\scriptsize{$1$}};
\draw (10.5,14.5) node{\scriptsize{$2$}};
\draw (11.5,14.5) node{\scriptsize{$3$}};
\draw (12.5,14.5) node{\scriptsize{$4$}};
\draw (13.5,14.5) node{\scriptsize{$5$}};
\draw (14.5,14.5) node{\scriptsize{$6$}};
\draw (15.5,14.5) node{\scriptsize{$7$}};

\draw (.5,14.5) node{\scriptsize{$j \backslash i$}};

\draw (1.5,.5) node{\tiny{$1$}};
\draw (2.5,.5) node{\tiny{$1$}};
\draw (2.5,1.5) node{\tiny{$1_2$}};
\draw (3.5,1.5) node{\tiny{$1{,}1_2$}};
\draw (2.5,2.5) node{\tiny{$1$}};
\draw (3.5,2.5) node{\tiny{$3$}};
\draw (4.5,2.5) node{\tiny{$2{,}1_2$}};
\draw (4.5,3.5) node{\tiny{$3{,}2_2$}};
\draw (5.5,3.5) node{\tiny{$5{,}2_2$}};
\draw (6.5,3.5) node{\tiny{$1$}};
\draw (4.5,4.5) node{\tiny{$2$}};
\draw (5.5,4.5) node{\tiny{$3{,}2_2$}};
\draw (6.5,4.5) node{\tiny{$3{,}5_2$}};
\draw (7.5,4.5) node{\tiny{$1{,}1_2$}};
\draw (5.5,5.5) node{\tiny{$2$}};
\draw (6.5,5.5) node{\tiny{$8{,}1_2$}};
\draw (7.5,5.5) node{\tiny{$7{,}3_2$}};
\draw (8.5,5.5) node{\tiny{$1{,}1_2$}};
\draw (6.5,6.5) node{\tiny{$1$}};
\draw (7.5,6.5) node{\tiny{$4{,}3_2$}};
\draw (8.5,6.5) node{\tiny{$8{,}6_2$}};
\draw (9.5,6.5) node{\tiny{$3{,}1_2$}};
\draw (7.5,7.5) node{\tiny{$3$}};
\draw (8.5,7.5) node{\tiny{$8{,}1_2$}};
\draw (9.5,7.5) node{\tiny{$4{,}6_2$}};
\draw (10.5,7.5) node{\tiny{$1{,}3_2$}};
\draw (8.5,8.5) node{\tiny{$1$}};
\draw (9.5,8.5) node{\tiny{$7{,}1_2$}};
\draw (10.5,8.5) node{\tiny{$8{,}3_2$}};
\draw (11.5,8.5) node{\tiny{$2{,}1_2$}};
\draw (9.5,9.5) node{\tiny{$1$}};
\draw (10.5,9.5) node{\tiny{$3{,}1_2$}};
\draw (11.5,9.5) node{\tiny{$3{,}5_2$}};
\draw (12.5,9.5) node{\tiny{$2{,}2_2$}};
\draw (10.5,10.5) node{\tiny{$1$}};
\draw (11.5,10.5) node{\tiny{$5$}};
\draw (12.5,10.5) node{\tiny{$3{,}2_2$}};
\draw (13.5,10.5) node{\tiny{$2_2$}};
\draw (12.5,11.5) node{\tiny{$2$}};
\draw (13.5,11.5) node{\tiny{$3{,}1_2$}};
\draw (14.5,11.5) node{\tiny{$1$}};
\draw (13.5,12.5) node{\tiny{$1$}};
\draw (14.5,12.5) node{\tiny{$1_2$}};
\draw (15.5,12.5) node{\tiny{$1_2$}};
\draw (14.5,13.5) node{\tiny{$1$}};
\draw (15.5,13.5) node{\tiny{$1$}};

\end{scope}
 
 \end{tikzpicture}$$
 \caption{The knot $10_{132}\# \overline{10_{132}}$ and its Khovanov homology.}
\label{figure:10_132}
\end{figure}

\begin{example} 
\label{example:11n376}
The link $L=L11n376$ and its Khovanov homology are depicted in Figure \ref{figure:L11n376}. The Khovanov homology of $L$ in its maximal quantum grading has rank two, and thus the link is not $B$-adequate, $B$-almost alternating, or $B$-Turaev genus one. The Khovanov homology in its minimal quantum grading has rank one. Also, $i_0(L) = -7$, $j_{\min}(L)=-19$, and $\delta_{\min}(L) = 1$. Therefore
$$2i_0(L) - j_{\min}(L) = 5 \neq 3 = \delta_{\min}(L)+2.$$
Theorem \ref{theorem:Diagonal} implies that $L$ is also not $A$-adequate, $A$-almost alternating, or $A$-Turaev genus one. Therefore $L$ is inadequate and has Turaev genus and dealternating number at least two.
\end{example}

\begin{figure}[h]
$$\begin{tikzpicture}
   
\begin{knot}[
	consider self intersections=true,
   	ignore endpoint intersections = false,
	clip width = 7,
]
\flipcrossings{5,4,9,11};
\strand [thick] (0,2) .. controls +(2.6,0) and +(120:-2.6) .. (210:2) .. controls +(120:2.6) and +(60:2.6) .. (-30:2) .. controls +(60:-2.6) and +(-2.6,0) .. (0,2);
\strand [thick] (-1,.5) circle (.7cm);
\strand [thick] (1,.5) circle (.7cm);

\end{knot}
\begin{scope}[yshift = -3cm, xshift = 4cm, scale = .55]
\draw (0,0) rectangle (9,10);
\foreach \i in {1,...,8}
{	
	\draw (\i,0) -- (\i,10);
}
\foreach \i in {1,...,9}
{
	\draw (0,\i) -- (9,\i);
}

\draw (.5,.5) node{\footnotesize{$\minus 19$}};
\draw (.5,1.5) node{\footnotesize{$\minus 17$}};
\draw (.5,2.5) node{\footnotesize{$\minus 15$}};
\draw (.5,3.5) node{\footnotesize{$\minus 13$}};
\draw (.5,4.5) node{\footnotesize{$\minus 11$}};
\draw (.5,5.5) node{\footnotesize{$\minus 9$}};
\draw (.5,6.5) node{\footnotesize{$\minus 7$}};
\draw (.5,7.5) node{\footnotesize{$\minus 5$}};
\draw (.5,8.5) node{\footnotesize{$\minus 3$}};
\draw (.5,9.5) node{\footnotesize{$j \backslash i$}};

\draw (1.5,9.5) node{\footnotesize{$\minus 7$}};
\draw (2.5,9.5) node{\footnotesize{$\minus 6$}};
\draw (3.5,9.5) node{\footnotesize{$\minus 5$}};
\draw (4.5,9.5) node{\footnotesize{$\minus 4$}};
\draw (5.5,9.5) node{\footnotesize{$\minus 3$}};
\draw (6.5,9.5) node{\footnotesize{$\minus 2$}};
\draw (7.5,9.5) node{\footnotesize{$\minus 1$}};
\draw (8.5,9.5) node{\footnotesize{$0$}};

\draw (1.5,.5) node{\tiny{$1$}};
\draw (2.5,1.5) node{\tiny{$1_2$}};
\draw (2.5,2.5) node{\tiny{$1$}};
\draw (3.5,2.5) node{\tiny{$2$}};
\draw (4.5,3.5) node{\tiny{$2{,}2_2$}};
\draw (4.5,4.5) node{\tiny{$4$}};
\draw (5.5,4.5) node{\tiny{$2{,}1_2$}};
\draw (4.5,5.5) node{\tiny{$1$}};
\draw (5.5,5.5) node{\tiny{$1$}};
\draw (6.5,5.5) node{\tiny{$1{,}2_2$}};
\draw (6.5,6.5) node{\tiny{$2$}};
\draw (7.5,6.5) node{\tiny{$1_2$}};
\draw (7.5,7.5) node{\tiny{$1$}};
\draw (8.5,7.5) node{\tiny{$2$}};
\draw (8.5,8.5) node{\tiny{$2$}};

\end{scope}

\end{tikzpicture}$$
\caption{The link L11n376 and its Khovanov homology.}
\label{figure:L11n376}
\end{figure}
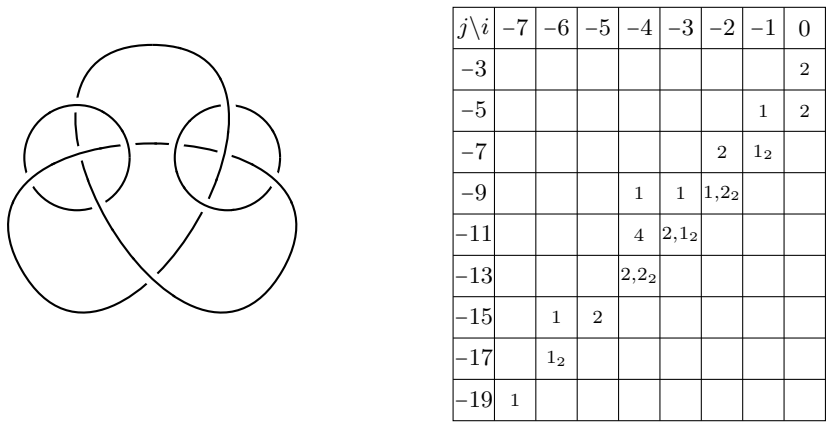

\begin{example}
\label{example:13crossings}
There are four 13 crossing knots with extremal Khovanov homology similar to that of $L11n376$. For each of these knots $Kh^{*,j_{\min}(K)}(K)$ has rank two while $Kh^{*,j_{\max}(K)}(K)$ has rank one. However, in each case $2i_0(K) - j_{\max}(K) \neq \delta_{\max}(K) -2$. Therefore, the knots listed in Table \ref{table:13}  are inadequate and have Turaev genus and dealternating number at least two.
\begin{table}[h]
\begin{tabular}{| c | c | c | c | }
\hline
Name &  $i_0$ & $j_{\max}$ & $\delta_{\max}$ \\
\hline
\hline
$13n_{588}$  & 8 & 23 & -3\\
\hline 
$13n_{1907}$ & 8 & 23 & -3\\
\hline
$13n_{2492}$ & 7 & 21 & -3\\
\hline
$13n_{2533}$ & 4 & 11 & 1\\
\hline
\end{tabular}
\vspace{5pt}
\caption{Knots with 13 crossings whose Khovanov homology does not satisfy Theorem \ref{theorem:Diagonal}. }
\label{table:13}
\end{table}
\end{example}

\section{Signature and extremal Khovanov homology}
\label{section:Signature}

In this section, we prove Theorem \ref{theorem:Signature} and conjecture a generalization of it in Conjecture \ref{conjecture:Signature}.

\begin{proof}[Proof of \ref{theorem:Signature}]
 If $L$ is a non-split alternating link with reduced diagram $D$, then Traczyk \cite{Traczyk:Signature} proved that
$$\sigma(L) = s_A(D)-c_+(D)-1 = -s_B(D) + c_-(D) + 1.$$
Dasbach and Lowrance \cite{DasLow:TuraevJones} proved that if $D$ is a Turaev genus one diagram of the link $L$, then
\begin{equation}
\label{equation:AASig}
\sigma(L) = s_A(D)-c_+(D) \pm 1.
\end{equation}

Suppose that $D$ is both $A$-Turaev genus one and $B$-Turaev genus one. Theorem \ref{theorem:Diagonal} and Corollary \ref{corollary:Diagonal} imply that both $Kh^{*,j_{\min}(L)}(L)\cong\mathbb{Z}$ and $Kh^{*,j_{\max}(L)}(L)\cong\mathbb{Z}$, and moreover $2i_0(L)-j_{\min}(L) =s_A(D)-c_+(D)=2i_0(L)-j_{\max}(L)$. Equation \ref{equation:AASig} then implies the result. 
\end{proof}

We conjecture that the hypotheses of Theorem \ref{theorem:Signature} can be weakened as follows without changing the conclusion.
\begin{conjecture}
\label{conjecture:Signature}
Suppose $L$ is a non-split Turaev genus one link. At least one of the following statements hold.
\begin{enumerate}
\item There is an $i_0\in\mathbb{Z}$ such that $Kh^{*,j_{\min}(L)}(L) = Kh^{i_0(L),j_{\min}(L)}(L)\cong \mathbb{Z}$ and $2i_0(L)-j_{\min}(L)  = \sigma(L)+1$.
\item There is an $i_0\in\mathbb{Z}$ such that $Kh^{*,j_{\max}(L)}(L) = Kh^{i_0(L),j_{\max}(L)}(L) \cong \mathbb{Z}$ and 
$2i_0(L) - j_{\max}(L)  = \sigma(L)-1$.
\end{enumerate}
\end{conjecture}

Conjecture \ref{conjecture:Signature} is open for links that are $A$-Turaev genus one, but not $B$-Turaev genus one (and vice versa). All knots with twelve or fewer crossings that are known to be Turaev genus one satisfy the conjecture. We end the section with an example that would not be Turaev genus one (hence, not almost alternating) assuming that Conjecture \ref{conjecture:Signature} is true.

\begin{example}
The 12 crossing non-alternating knots $12n_{809}$ and $12n_{835}$ have signatures $\sigma(12n_{809})=-2$ and $\sigma(12n_{835})=0$. Their Khovanov homologies are rank two in the maximum quantum grading and rank one in the minimum quantum grading. We have that $i_0(12n_{809})= -4$ and $j_{\min}(12n_{809})=-5$, and thus $2i_0(12n_{809})-j_{\min}(12n_{809}) \neq \sigma(12n_{809})+1$. Similarly, $i_0(12n_{835}) = -7$ and $j_{\min}(12n_{835}) = -13$, and thus $2i_0(12n_{835}) - j_{\min}(12n_{835}) \neq \sigma(12n_{835})+1$. Therefore, if Conjecture \ref{conjecture:Signature} holds, the dealternating number and Turaev genus of both $12n_{809}$ and $12n_{835}$ are at least two.
\end{example}

\section{The maximal Thurston Bennequin number of an almost alternating link}
\label{section:TB}

A Legendrian link in $\mathbb{R}^3$ with the standard contact structure $dz - y \; dx$ projects in the $xz$-plane to a Legendrian front diagram (or simply a front). A front has no vertical tangents, and its singular points are transverse double points and cusps (in place of a vertical tangents). A double point $~
\tikz[baseline=.6ex, scale = .4]{
\draw (0,0) -- (1,1);
\draw (1,0) -- (0,1);
}
~$ in a front diagram should be interpreted as a crossing $~
\tikz[baseline=.6ex, scale = .4]{
\draw (0,0) -- (.3,.3);
\draw (.7,.7) -- (1,1);
\draw (1,0) -- (0,1);
}
~$ where the segment with negative slope passes over the segment with positive slope.

The \textit{Thurston Bennequin number} $\tb(\mathcal{L})$ of an oriented Legendrian link $\mathcal{L}$ with front diagram $F$ is defined as $\tb(\mathcal{L}) = w(F) - c(F)$ where $w(F)$ is the writhe of the front $F$ and $c(F)$ is half the number of cusps in $F$. Ng \cite{Ng:TBKh} proved that the maximal Thurston Bennequin number of a link has the following upper bound given by Khovanov homology.
\begin{theorem}[Ng]
\label{theorem:tbkh}
Let $L$ be a non-split link. Then
$$\mtb(L) \leq \min \{j-i \mid Kh^{i,j}(L)\neq 0\}.$$
\end{theorem}

Ng used the bound in Theorem \ref{theorem:tbkh} to prove the following result that we slightly reformulate to match the notation of this paper.
\begin{theorem}[Ng]
\label{theorem:tbalt}
Let $L$ be a non-split alternating link with reduced alternating diagram $D$. Then
$$\mtb(L) =  w(D) - s_A(D)$$
where $w(D)$ is the writhe of $D$ and $s_A(D)$ is the number of components in the all-$A$ state of $D$.
\end{theorem}

K\'alm\'an \cite{Kalman:TB} extended Ng's result to $A$- and $B$-adequate links by proving the following result.
\begin{theorem}[K\'alm\'an]
\label{theorem:Kalman}
Let $L$ be a link with diagram $D$. If $D$ is $A$-adequate, then
$$\mtb(L) = w(D) - s_A(D).$$
If $D$ is $B$-adequate, then
$$\mtb(\overline{L}) = -w(D) - s_B(D).$$
\end{theorem}
The proofs of Theorems \ref{theorem:tbalt} and \ref{theorem:Kalman} use \textit{Mondrian diagrams}, diagrams consisting of a set of thick disjoint horizontal line segments and a set of thin disjoint vertical line segments such that each vertical segment begins and ends on a horizontal segment and does not intersect any other horizontal segment. Each Mondrian diagram yields a graph embedded in the plane by contracting the horizontal segments into vertices. Ng proved that every plane graph is the contraction of a Mondrian diagram. 

We modify Ng's approach to prove Theorem \ref{theorem:TBAA}. Informally, we start with an almost alternating diagram and then construct its checkerboard graph. The vertices of the checkerboard graph are then stretched horizontally to form the thick line segments of a Mondrian diagram. Finally, a Legendrian front diagram is obtained from a Mondrian diagram by replacing thick horizontal line segments with two-cusped unknots and replacing thin vertical line segments with crossings, as in Figure \ref{figure:M2L}. The crossing associated to the dealternator will need to be modified according to Figure \ref{figure:MondrianAA} in order to obtain a Legendrian front diagram. The resulting Legendrian front diagram has the same checkerboard graph as the original almost alternating diagram, and so the two are isotopic on the projection $S^2$. Figure \ref{figure:T34} shows this process for an almost alternating diagram of the $(3,4)$ torus knot.

\begin{figure}[h]
$$\begin{tikzpicture}

\draw [line width = .1 cm] (0,0) -- (2,0);

\draw[ultra thick, ->] (2.5,0) -- (3.5,0);

\draw (4,0) to [out = 0, in = 180] (4.5,.25) to (5.5,.25) to [out=0, in = 180] (6,0) to [out=180, in=0] (5.5,-.25) to (4.5,-.25) to [out =180, in=0] (4,0);

\draw[thick] (8,.5) -- (8,-.5);
\draw[ultra thick, ->] (8.5,0) -- (9.5,0);
\draw[thick] (10,.5) -- (11,-.5);
\draw[thick] (10,-.5) -- (10.3,-.2);
\draw[thick] (10.7,.2) -- (11,.5);

\end{tikzpicture}$$
\caption{Turning a Mondrian diagram into a Legendrian front diagram.}
\label{figure:M2L}
\end{figure}
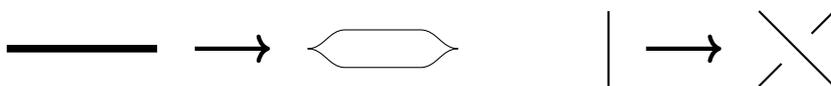

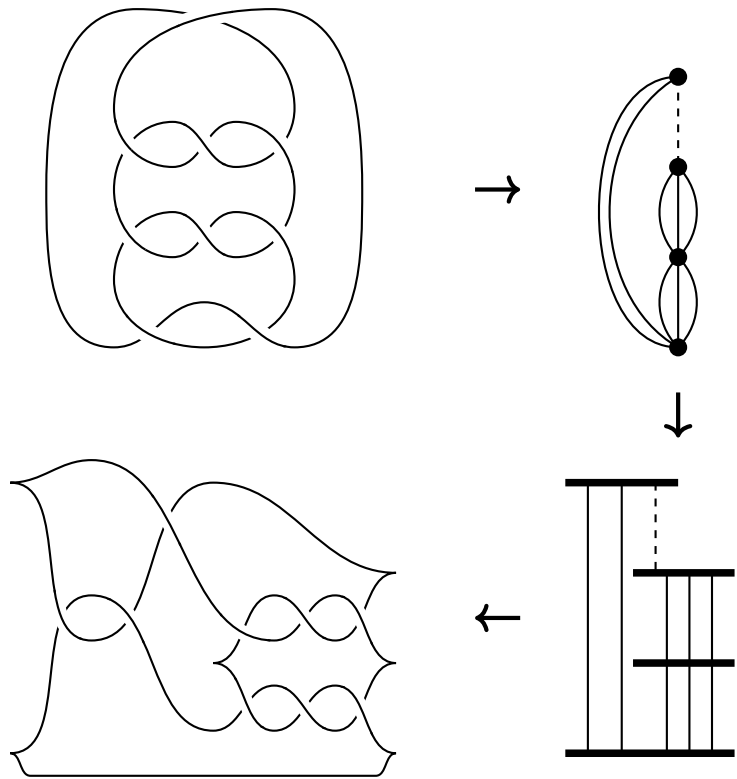
\begin{figure}[h]

$$\begin{tikzpicture}[thick, scale = .6]

\draw[ultra thick, ->] (7.5,3.5) -- (8.5,3.5);
\draw[ultra thick, ->] (8.5,-6) -- (7.5,-6);
\draw[ultra thick, ->] (12,-1) -- (12,-2);


\begin{knot}[
	consider self intersections,
 	clip width = 3,
 	ignore endpoint intersections = true,
	end tolerance = 2pt
	]
	\flipcrossings{5, 7, 9, 1}
	\strand (-2,3.5) to [out = 270, in = 180]
	(-.5,0) to [out = 0, in = 180]
	(1.5,1) to [out = 0, in =180]
	(3.5,0) to [out =0, in=270]
	(5,3.5) to [out = 90, in = 0]
	(3,7.5) to [out = 180, in = 90]
	(-.5,5.3) to [out =270, in = 180]
	(.8,4) to [out = 0, in = 180]
	(2.2,5) to [out = 0, in = 90]
	(3.5,3.5) to [out = 270, in = 0]
	(2.2,2) to [out = 180, in = 0]
	(.8,3) to [out = 180, in = 90]
	(-.5,1.5) to [out = 270, in = 180]
	(1.5,0) to [out =0, in =270]
	(3.5,1.5) to [out = 90, in= 0]
	(2.2,3) to [out =180, in = 0]
	(.8,2) to [out =180, in = 270]
	(-.5,3.5) to [out = 90, in =180]
	(.8,5) to [out = 0, in = 180]
	(2.2,4) to [out = 0, in = 270]
	(3.5,5.3) to [out = 90, in = 0]
	(0,7.5) to [out = 180, in = 90]
	(-2,3.5);
	
	\end{knot}


\begin{scope}[xshift = 10 cm]

\fill (2,6) circle (.2cm);
\fill (2,4) circle (.2cm);
\fill (2,2) circle (.2cm);
\fill (2,0) circle (.2cm);

\draw[thick, dashed] (2,6) -- (2,4);
\draw[thick] (2,4) -- (2,0);
\draw[thick] (2,4) to [out = 225, in = 135] (2,2);
\draw[thick] (2,4) to [out = -45, in = 45] (2,2);
\draw[thick] (2,2) to [out = 225, in = 135] (2,0);
\draw[thick] (2,2) to [out = -45, in = 45] (2,0);
\draw[thick] (2,0) to [out = 150, in = 210] (2,6);
\draw[thick] (2,0) to [out = 180, in = 180] (2,6);

\end{scope}


\begin{scope}[xshift = 9cm, yshift = -9cm]

\draw[line width = .1cm] (.5,6) -- (3,6);
\draw[line width = .1cm] (2,4) -- (4.25,4);
\draw[line width = .1cm] (2,2) -- (4.25,2);
\draw [line width = .1cm] (.5,0) -- (4.25,0);

\draw (1,6) -- (1,0);
\draw[dashed] (2.5,6) -- (2.5,4);
\draw (2.75,4) -- (2.75,0);

\draw (3.25,4) -- (3.25,0);
\draw (3.75,4) -- (3.75,0);
\draw (1.75,6) -- (1.75,0);

\end{scope}


\begin{scope}[yshift = -9cm, xshift = -1cm, xscale = .9]

\begin{knot}[
	consider self intersections,
 	clip width = 7,
 	ignore endpoint intersections = true,
	]
	\flipcrossings{2, 3, 4, 6, 8};
	\strand[thick]
	(-2,6) to [out = 0 , in = 180]
	(0,2.5) to [out = 0, in =180]
	(3,6) to [out = 0, in =180]
	(7.5,4) to [out = 180, in = 0]
	(6,2.5) to [out = 180, in = 0]
	(4.5,3.5) to [out = 180, in =0]
	(3,2) to [out = 0, in=180]
	(4.5,.5) to [out = 0, in =180]
	(6,1.5) to [out = 0, in =180]
	(7.5,0) to [out = 180, in = 0]
	(7,-.5) to [out = 180, in = 0]
	(-1.5,-.5) to [out = 180, in =0]
	(-2,0) to [out = 0, in = 180]
	(0,3.5) to [out = 0, in =180]
	(3,.5) to [out = 0, in = 180]
	(4.5,1.5) to [out =0, in = 180]
	(6,.5) to [out = 0, in =180]
	(7.5,2) to [out = 180, in=0]
	(6,3.5) to [out = 180, in = 0]
	(4.5,2.5) to [out = 180, in = 0] 
	(0,6.5) to [out = 180, in = 0]
	(-2,6);

\end{knot}

\end{scope}

\end{tikzpicture}$$
\caption{An almost alternating diagram of the $(3,4)$ torus knot, its checkerboard graph, Mondrian diagram, and Legendrian front diagram.}
\label{figure:T34}
\end{figure}

In order to show the process of Figure \ref{figure:T34} can be carried out for an arbitrary almost alternating diagram, we need a lemma about the structure of Mondrian diagrams for almost alternating links. It may be helpful for the reader to consult Ng's proof of Proposition 11 in \cite{Ng:TBKh} where he shows that every planar graph is the contraction of a Mondrian diagram.

\begin{lemma}
\label{lemma:Mondrian}
Let $L$ be an almost alternating link. There exists an almost alternating diagram $D$ of $L$ and a Mondrian diagram contracting to the checkerboard graph of $D$ such that the vertical edge associated to the dealternator is the leftmost edge on the top of the horizontal segment containing its lower endpoint and is the rightmost edge on the bottom of the horizontal segment containing its upper endpoint. 
\end{lemma}
\begin{proof}
In any almost alternating diagram of $L$, the edge associated to the dealternator is the unique edge between the two vertices it is incident to. If more than one such edge existed, then, after possibly flyping, the diagram could be transformed into an alternating diagram via a Reidemeister 2 move, as in Figure \ref{figure:Reducible}.

First we assume that $L$ is prime, and let $D$ be an almost alternating diagram of $L$.  Color the complementary regions of $D$ in a checkerboard fashion so that the coloring near the dealternator looks like  $\tikz[baseline=.6ex, scale = .4, thick]{
\fill[black!20!white] (0,0) -- (.5,.5) -- (1,0);
\fill[black!20!white] (1,1) -- (.5,.5) -- (0,1);
\draw (0,0) -- (1,1);
\draw (0,1) -- (.3,.7);
\draw (1,0) -- (.7,.3);
}$ and the coloring near every other crossing looks like $\tikz[baseline=.6ex, scale = .4, thick]{
\fill[black!20!white] (0,0) -- (.5,.5) -- (0,1);
\fill[black!20!white] (1,1) -- (.5,.5) -- (1,0);
\draw (0,0) -- (1,1);
\draw (0,1) -- (.3,.7);
\draw (1,0) -- (.7,.3);
}$. Let $G$ be the checkerboard graph of $D$ whose vertices correspond to the shaded regions. Since $D$ is prime, the graph $G$ is $2$-connected. 

Ng \cite[Proposition 11]{Ng:TBKh} gives an algorithm that constructs a Mondrian diagram that contracts to $G$. The algorithm starts by associating a step-shaped cycle in the Mondrian diagram to a cycle of $G$ that bounds a face in its planar embedding (see Figure \ref{figure:Mondrian}). The algorithm completes the Mondrian diagram by adding vertical edges and adding or extending horizontal edges in the interior of the step shaped cycle. Choose the initial cycle to contain the edge associated to the dealternator, and label the vertices incident to the dealternator edge $1$ and $2$. Any new horizontal edges can be chosen to be below the horizontal edge labeled $2$. Therefore the vertical edge associated to the dealternator is the leftmost edge on the top of the horizontal segment labeled $2$ and is the rightmost edge on the bottom of the horizontal segment labeled $1$.

Now suppose that $L$ is composite, rather than prime. Then the checkerboard graph of a diagram of $L$ may not be $2$-connected. A graph that is not $2$-connected can be decomposed into $2$-connected pieces, called blocks, glued together at vertices. Changing the vertices where the blocks are glued together corresponds to taking connected sums of link diagrams using different arcs of the diagram.  Ng constructs a Mondrian diagram associated to a graph that is not $2$-connected by first forming the Mondrian diagram of each block, and then gluing them together by extending and identifying the horizontal lines in the outer step shaped cycle for each block. By carefully choosing what arcs of the diagram are involved in a connected sum operation, we can ensure that there is a $A$-almost alternating diagram where the vertices incident to the dealternator edge are not part of any other block. Hence the associated Mondrian horizontal edges (edges $1$ and $2$ in the above construction) are not extended or identified with any other edges. Thus the argument for the prime case extends to the composite case.
\end{proof}

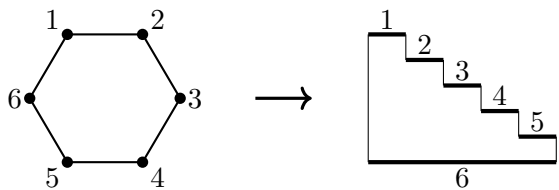
\begin{figure}[h]
$$\begin{tikzpicture}[scale = .5]

\fill (0,1.7) circle (.15);
\fill (1,0) circle (.15);
\fill (1,3.4) circle (.15);
\fill (3,0) circle (.15);
\fill (3,3.4) circle (.15);
\fill (4,1.7) circle (.15);
\draw[thick] (0,1.7) -- (1,0) -- (3,0) -- (4,1.7) -- (3,3.4) -- (1,3.4) -- (0,1.7);

\draw (-.4,1.7) node{6};
\draw (.6,3.8) node{1};
\draw (3.4,3.8) node{2};
\draw (4.4,1.7) node{3};
\draw (3.4,-.4) node{4};
\draw (.6,-.4) node{5};

\draw[ultra thick] (9,0 ) -- (14,0);
\draw[ultra thick] (14,.68) -- (13,.68);
\draw[ultra thick] (13,1.36) -- (12,1.36);
\draw[ultra thick] (12,2.04) -- (11,2.04);
\draw[ultra thick] (11,2.72) -- (10,2.72);
\draw[ultra thick] (10,3.4) -- (9,3.4);

\draw (9,0) -- (9,3.4);
\draw (10,3.4) -- (10,2.72);
\draw (11,2.72) -- (11,2.04);
\draw (12,2.04) -- (12,1.36);
\draw (13,1.36) -- (13,.68);
\draw (14,.68) -- (14,0);

\draw (11.5,-.4) node {6};
\draw (9.5, 3.8) node{1};
\draw (10.5, 3.12) node {2};
\draw (11.5, 2.44) node{3};
\draw (12.5, 1.76) node{4};
\draw (13.5, 1.08) node{5};

\draw[very thick, ->] (6,1.7) -- (7.5,1.7);

\end{tikzpicture}$$
\caption{A cycle in $G$ induces a step-shaped cycle in the Mondrian diagram.}
\label{figure:Mondrian}
\end{figure}

A Mondrian diagram associated to an alternating link diagram yields a Legendrian front diagram of the link by replacing the thick horizontal lines with two-cusped unknots and replacing thin vertical lines with crossings. We slightly modify this construction to give a Legendrian front diagram of an almost alternating link, and we end the paper with the proof of Theorem \ref{theorem:TBAA}.


\begin{proof}[Proof of Theorem \ref{theorem:TBAA}]
Let $D$ be an $A$-Turaev genus one diagram. If $D$ is $A$-adequate, then the result follows from Theorem \ref{theorem:Kalman}. Suppose that $D$ is $A$-almost alternating.


Theorem \ref{theorem:AAKh} and Corollary \ref{corollary:Diagonal} imply that the Khovanov homology of $D$ is nontrivial in homological grading $i_0(L)=1- c_-(D)$ and quantum grading $j_{\min}(L) = -s_A(D) + c_+(D) -2c_-(D)$. Theorem \ref{theorem:tbkh} then implies that $\mtb(L) \leq j_{\min}(L)-  i_0(L) = w(D) - s_A(D) + 1$. Let $D_{\alt}$ be the alternating diagram obtained by performing a crossing change on the dealternator of $D$. Lemma \ref{lemma:Mondrian} states that $D$ has a Mondrian diagram where the edge associated  to the dealternator is the leftmost edge on the horizontal segment containing its bottom endpoint and the rightmost edge on the horizontal segment containing its top endpoint (see the left side of Figure \ref{figure:MondrianAA}). The Mondrian diagram for $D$ is also a Mondrian diagram for $D_{\alt}$. Construct a Legendrian front diagram $\mathcal{D}_{\alt}$ as in Figure \ref{figure:M2L}. Theorem \ref{theorem:tbalt} implies that $D_{\alt}$ has Thurston Bennequin number $\tb(\mathcal{D}_{\alt}) = w(D_{\alt}) - s_A(D_{\alt})$. Changing the dealternator crossing of the Legendrian front diagram $\mathcal{D}_{\alt}$ results in a diagram $D_{\text{alm}}$ of $D$ that is nearly a Legendrian front diagram, but the dealternator has the strand with positive slope passing over the strand with negative slope. Performing the local isotopy in Figure \ref{figure:MondrianAA} changes $D_{\text{alm}}$ into a Legendrian front diagram $\mathcal{D}_{\text{alm}}$ of $L$ with two fewer cusps than $\mathcal{D}_{\alt}$. Therefore $\tb(\mathcal{D}_{\text{alm}}) = w(D) - (s_A(D_{\text{alt}}) - 1) = w(D)-s_A(D),$ and hence $ w(D) - s_A(D) \leq \mtb(L)$, proving the result.
\end{proof}
\begin{figure}[h]
$$\begin{tikzpicture}


\draw [ultra thick] (0,1) -- (2,1);
\draw [ultra thick] (1,0) -- (3,0);
\draw (1.5,1) -- (1.5,0);

\draw (2,-1) node {Mondrian};


\begin{scope}[xshift = 3.5cm]
 
 \begin{knot}[
	consider self intersections,
 	clip width = 5,
 	ignore endpoint intersections = false
 ]
 \flipcrossings{1};
  \strand[thick]
  (2.5,-.3) to [out = 180, in = 315, looseness=.5]
  (1.3, -.1) to [out = 135, in = 0, looseness=1]
  (1,0) to [out = 0, in = 225, looseness=1]
  (1.2,.1) to [out = 45, in = 225, looseness = 1]
  (1.8,.9) to [out = 45, in = 180, looseness =1]
  (2,1) to [out =180, in = 315, looseness = 1]
  (1.7,1.1) to [out = 135, in = 0, looseness = .5]
  (.5,1.3);
  
  \strand[thick]
  (.5,1) to [out = 0, in =135, looseness= .5]
  (1.2, .8) to [out = 315, in = 135, looseness = 1]
  (1.8, .2) to [out = 315, in = 180, looseness = .5]
  (2.5,0);
 \end{knot}
 \draw (1.5,-1) node{$\mathcal{D}_{\alt}$};

\end{scope}


\begin{scope}[xshift = 7cm]
 
 \begin{knot}[
	consider self intersections,
 	clip width = 5,
 	ignore endpoint intersections = false
 ]
  \strand[thick]
  (2.5,-.3) to [out = 180, in = 315, looseness=.5]
  (1.3, -.1) to [out = 135, in = 0, looseness=1]
  (1,0) to [out = 0, in = 225, looseness=1]
  (1.2,.1) to [out = 45, in = 225, looseness = 1]
  (1.8,.9) to [out = 45, in = 180, looseness =1]
  (2,1) to [out =180, in = 315, looseness = 1]
  (1.7,1.1) to [out = 135, in = 0, looseness = .5]
  (.5,1.3);
  
  \strand[thick]
  (.5,1) to [out = 0, in =135, looseness= .5]
  (1.2, .8) to [out = 315, in = 135, looseness = 1]
  (1.8, .2) to [out = 315, in = 180, looseness = .5]
  (2.5,0);
 \end{knot}
 \draw (1.5,-1) node{$D_{\text{alm}}$};

\end{scope}


\begin{scope}[xshift = 10.5cm]
 
 \begin{knot}[
	consider self intersections,
 	clip width = 5,
 	ignore endpoint intersections = false
 ]
  \strand[thick]
  (2.5,-.3) to [out = 180, in = 315, looseness=.5]
  (1.8, .2) to [out = 135, in = 315, looseness=1]
  (1.2,.8) to [out = 135, in = 0, looseness = .5]
  (.5,1.3);

\strand[thick]
(2.5,0) to [out = 180, in = 0, looseness = 1]
(1.8,1) to [out = 180, in = 0 ,looseness = 1]
(1.2,0) to [out  = 180, in = 0, looseness = 1]
(.5,1);

\draw (1.5,-1) node{$\mathcal{D}_{\text{alm}}$};  

 \end{knot}

\end{scope}

\end{tikzpicture}$$
 
\caption{The portion of the Mondrian diagram associated to the dealternator yields the crossing as it looks in the alternating Legendrian diagram $\mathcal{D}_{\alt}$. Changing that crossing gives the no longer Legendrian diagram $D_{\text{alm}}$, but making the pictured local change turns the diagram into the Legendrian front diagram $\mathcal{D}_{\text{alm}}$.}
\label{figure:MondrianAA}
\end{figure}
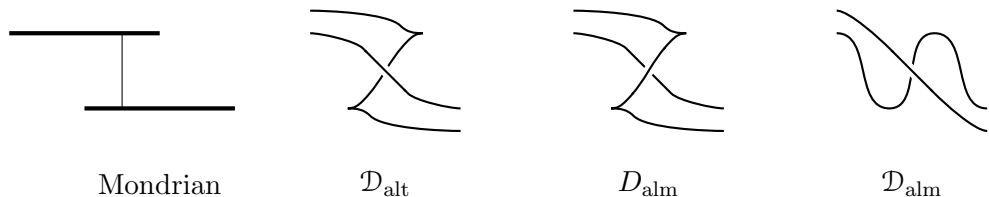

\bibliography{aa}{}
\bibliographystyle {amsalpha}
\end{document}